\documentclass[10pt]{amsart}
\usepackage{graphicx,amssymb,amsfonts,amsmath,amsthm,newlfont}
\usepackage{color}
\usepackage{epsfig,url}
\usepackage{numprint}
\usepackage{axodraw4j}
\usepackage{hyperref}
\usepackage{color}

\newcommand{\CC}{{\mathbb C}}

\newcommand{\PP}{{\mathbb P}}
\newcommand{\QQ}{{\mathbb Q}}
\newcommand{\RR}{{\mathbb R}}

\newcommand{\ZZ}{{\mathbb Z}}

\newcommand{\dd}{\mathrm{d}}

\newcommand{\sE}{\mathcal{E}}

\newcommand{\pp}{{\underline{p}}}
\newcommand{\qq}{{\underline{q}}}

\newcommand{\tr}{\mathrm{tr}\,}
\newcommand{\diag}{\mathrm{diag}}
\newcommand{\eee}{{\underline{e}}}
\newcommand{\one}{{\underline{1}}}

\newcommand{\perm}{\mathrm{perm}\,}
\newcommand{\sgn}{\mathrm{sgn}\,}

\newcommand{\sym}{\mathrm{sym}}

\theoremstyle{plain}
\newtheorem{thm}{Theorem}
\newtheorem{lem}[thm]{Lemma}

\newtheorem{cor}[thm]{Corollary}
\newtheorem{prop}[thm]{Proposition}

\newtheorem{remark}[thm]{Remark}
\newtheorem{defn}[thm]{Definition}

\newtheorem{ex}[thm]{Example}

\newcommand{\LM}{L\mathcal{M}}
\newcommand{\LA}{L\mathcal{A}}
\newcommand{\GL}{\mathrm{GL}}
\newcommand{\trop}{\mathrm{trop}}

\newcommand{\To}{\longrightarrow}
\newcommand{\GC}{\mathcal{GC}}
\newcommand{\lf}{\mathrm{lf}}

\newcommand{\0}{\color{blue}{\mathsf{0}}}

\title{The wheel classes in the locally finite homology of $\mathrm{GL}_n(\mathbb{Z})$, canonical integrals and zeta values}

\author{Francis Brown and Oliver Schnetz}
\address{Francis Brown\\
All Souls College, Oxford, OX1 4AL, United Kingdom}
\email{francis.brown@all-souls.ox.ac.uk}
\address{Oliver Schnetz\\
II. Institute for Theoretical Physics\\
Luruper Chaussee 149\\
22761 Hamburg, Germany}
\email{schnetz@mi.uni-erlangen.de}

\begin{document}
\begin{abstract}
We compute the canonical integrals associated to wheel graphs, and prove that they are proportional to odd zeta values. From this we deduce that wheel classes define explicit non-zero classes in: the locally finite homology of the general linear group $\GL_n(\ZZ)$ in both odd and even ranks, the homology of the moduli space of tropical curves,  and the moduli space of tropical abelian varieties. We deduce the canonical integrals of a doubly infinite family of auxiliary classes in the even commutative graph complex. The proof also leads to a formula for the complete anti-symmetrisation of a product of $2n-1$ matrices of rank $n$, in the spirit of the Amitsur-Levitzki theorem.
\end{abstract}

\maketitle

\section{Introduction}
Let $G$ be a finite connected graph.  Assign   a variable $x_e$ to each of the edges $e\in E(G)$ of $G$, and let 
$\sigma_{G}   \subset  \PP^{E(G)}(\RR)$ 
denote  the unit coordinate simplex in the projective space of dimension $|E(G)|-1$ with  coordinates $x_e, e\in E(G)$. It is defined by the  region where all  $x_e\geq 0$. We consider  integrals over $\sigma_G$  of certain differential  forms which are indexed by a positive integer $k$ and defined   as follows. 

For any family of invertible matrices  $X$ consider the differential $k$ form:
 \begin{equation} \label{intro:omegandef}
 \omega_X^{k} = \tr ((X^{-1} \dd X)^k)\ .
 \end{equation}
It defines a closed  projective differential form in the entries of $X$ which is invariant under left and right multiplication by constant invertible matrices. If the matrices  $X$ are symmetric, then $\omega_X^k$ vanishes unless $k\equiv 1 \pmod{4}$. The forms   $\omega^{4k+1}_X$,  when $X$ ranges over positive definite $n\times n$ symmetric matrices, represent the Borel classes in the stable cohomology \cite{Borel} of the general linear group $\mathrm{GL}_n(\ZZ)$ as $n\rightarrow \infty$.

Now let   $\Lambda_{G}$ denote a graph Laplacian matrix associated to $G$. It is a symmetric $h_G \times h_G$ matrix whose entries are linear forms in the $x_e$, where $h_G$ is the number of independent cycles of $G$ (see Section \ref{sect:1point1}).  Let $n>1$ be odd, and suppose that $G$ has $2n$ edges. In this paper  we  study the projective integrals    
\begin{equation}  \label{intro: IGdef} I_{G} (\omega^{2n-1}) =  \int_{\sigma_{G}}  \omega^{2n-1}_{\Lambda_G} \ . 
\end{equation}
It was shown in  \cite{BrSigma} that these integrals are well-defined and  always finite. They   are generalised Feynman integrals of the kind studied in quantum field theory. 

For any $n \geq 3$ let $W_n$ be the wheel graph with $n$ spokes and $2n$ edges. It will be oriented in such a way that, with our conventions, the integral \eqref{intro: IGdef} is positive or zero. 
    
\begin{thm} \label{thm: introMain} 
For all odd $n\geq 3$,
\begin{equation}\label{wheelseq} 
I_{W_n}(\omega^{2n-1})=n\genfrac(){0pt}{}{2n}{n}\zeta(n).
\end{equation}
\end{thm}
A much  weaker version of this theorem, namely $I_{W_n}(\omega^{2n-1})\neq 0$, already has a considerable number of consequences for the homology of graph complexes, the general linear group, the moduli spaces of tropical curves and abelian varieties, and also for the cohomology of the moduli stacks of curves and abelian varieties. The integrals \eqref{intro: IGdef} are closely related to regulator integrals in algebraic $K$-theory and  are periods of the graph motives introduced in \cite{BEK}. A selection of these consequences will be summarised  here, and  others discussed in greater detail in Section \ref{sect: Consequences}.  

\subsection{Discussion and examples of wheel canonical integrals} \label{sect:1point1}
The version of the graph Laplacian which we consider is defined as follows. Consider the inner product  on $\ZZ^{E(G)}$ which satisfies
$\langle e, e'\rangle = \delta_{e,e'} x_e$, for any pair of edges $e,e'$,  where $x_e$ is the variable associated to edge $e$ as defined above. The graph Laplacian is its restriction to the first homology $H_1(G;\ZZ) \subset \ZZ^{E_G}$, and in any choice of basis of  the free $\ZZ$-module $H_1(G;\ZZ)$ defines an $h_G \times h_G$ matrix whose entries are linear forms in the $x_e.$ Concretely, 
if one chooses an orientation on each edge of $e$, and correspondingly  a cycle basis $c_1,\ldots, c_{h_G} \in \ZZ^{E(G)}$ of $H_1(G;\ZZ)$, then the $i,j$th entry of $\Lambda_G$ is $\langle c_i, c_j\rangle$. 

\begin{figure}[ht]
    \centering
    \fcolorbox{white}{white}{
  \begin{picture}(400,176) (35,-11)
    \SetWidth{0.8}
    \SetColor{Black}
    \Arc[arrow,arrowpos=0.5,arrowlength=5,arrowwidth=2,arrowinset=0.2](112,80)(57.689,124,484)
    \Arc[arrow,arrowpos=0.5,arrowlength=5,arrowwidth=2,arrowinset=0.2](336,80)(57.689,124,484)
    \Text(109,0)[lb]{{\Black{$W_3$}}}
    \Text(333,0)[lb]{{\Black{$W_n$}}}
    \Vertex(71,121){2.236}
    \Vertex(71,39){2.236}
    \Vertex(112,80){2.236}
    \Vertex(170,80){2.236}
    \Vertex(278,80){2.236}
    \Vertex(295,39){2.236}
    \Vertex(295,121){2.236}
    \Vertex(336,80){2.236}
    \Vertex(336,137){2.236}
    \Vertex(377,121){2.236}
    \Vertex(377,39){2.236}
    \Vertex(394,80){2.236}
    \Line[arrow,arrowpos=0.5,arrowlength=5,arrowwidth=2,arrowinset=0.2](71,121)(112,80)
    \Line[arrow,arrowpos=0.5,arrowlength=5,arrowwidth=2,arrowinset=0.2](71,39)(112,80)
    \Line[arrow,arrowpos=0.5,arrowlength=5,arrowwidth=2,arrowinset=0.2](170,80)(112,80)
    \Line[arrow,arrowpos=0.5,arrowlength=5,arrowwidth=2,arrowinset=0.2](278,80)(336,80)
    \Line[arrow,arrowpos=0.5,arrowlength=5,arrowwidth=2,arrowinset=0.2](295,121)(336,80)
    \Line[arrow,arrowpos=0.5,arrowlength=5,arrowwidth=2,arrowinset=0.2](295,39)(336,80)
    \Line[arrow,arrowpos=0.5,arrowlength=5,arrowwidth=2,arrowinset=0.2](336,137)(336,80)
    \Line[arrow,arrowpos=0.5,arrowlength=5,arrowwidth=2,arrowinset=0.2](377,121)(336,80)
    \Line[arrow,arrowpos=0.5,arrowlength=5,arrowwidth=2,arrowinset=0.2](377,39)(336,80)
    \Line[arrow,arrowpos=0.5,arrowlength=5,arrowwidth=2,arrowinset=0.2](394,80)(336,80)

    \Line[arrow,arrowpos=0.5,arrowlength=5,arrowwidth=2,arrowinset=0.2](54.5,81)(54.5,80)
    \Line[arrow,arrowpos=0.5,arrowlength=5,arrowwidth=2,arrowinset=0.2](150,123)(149,124)
    \Line[arrow,arrowpos=0.5,arrowlength=5,arrowwidth=2,arrowinset=0.2](281,62)(282,59)
    \Line[arrow,arrowpos=0.5,arrowlength=5,arrowwidth=2,arrowinset=0.2](282,102)(281,99)
    \Line[arrow,arrowpos=0.5,arrowlength=5,arrowwidth=2,arrowinset=0.2](312,28)(314,27)
    \Line[arrow,arrowpos=0.5,arrowlength=5,arrowwidth=2,arrowinset=0.2](317,134.5)(315,133.5)
    \Line[arrow,arrowpos=0.5,arrowlength=5,arrowwidth=2,arrowinset=0.2](358,133)(356,134)
    \Line[arrow,arrowpos=0.5,arrowlength=5,arrowwidth=2,arrowinset=0.2](389,57)(390,60)
    \Line[arrow,arrowpos=0.5,arrowlength=5,arrowwidth=2,arrowinset=0.2](390,100)(389,103)
    \Text(334,34)[lb]{{\Black{$\ldots$}}}
    \Text(146,130)[lb]{{\Black{$e_1$}}}
    \Text(39,80)[lb]{{\Black{$e_2$}}}
    \Text(147,23)[lb]{{\Black{$e_3$}}}
    \Text(95,104)[lb]{{\Black{$e_4$}}}
    \Text(95,50)[lb]{{\Black{$e_5$}}}
    \Text(146,84)[lb]{{\Black{$e_6$}}}
    \Text(125,112)[lb]{{\Black{$c_1$}}}
    \Text(73,80)[lb]{{\Black{$c_2$}}}
    \Text(125,42)[lb]{{\Black{$c_3$}}}

    \Text(394,106)[lb]{{\Black{$e_1$}}}
    \Text(361,138)[lb]{{\Black{$e_2$}}}
    \Text(305,138)[lb]{{\Black{$e_3$}}}
    \Text(272,107)[lb]{{\Black{$e_4$}}}
    \Text(272,50)[lb]{{\Black{$e_5$}}}
    \Text(395,50)[lb]{{\Black{$e_n$}}}
    \Text(372,61)[lb]{{\Black{$c_n$}}}
    \Text(372,95)[lb]{{\Black{$c_1$}}}
    \Text(349,118)[lb]{{\Black{$c_2$}}}
    \Text(317,118)[lb]{{\Black{$c_3$}}}
    \Text(294,95)[lb]{{\Black{$c_4$}}}
    \Text(294,61)[lb]{{\Black{$c_5$}}}
    \Text(353,72)[lb]{{\Black{$e_{2n}$}}}
    \Text(353,90)[lb]{{\Black{$e_{n+1}$}}}
    \Text(339,105)[lb]{{\Black{$e_{n+2}$}}}
    \Text(312,105)[lb]{{\Black{$e_{n+3}$}}}
    \Text(305,82)[lb]{{\Black{$e_{n+4}$}}}
    \Text(311,47)[lb]{{\Black{$e_{n+5}$}}}
    \Text(341,47)[lb]{{\Black{$e_{2n-1}$}}}
  \end{picture}
}

    \caption{The wheel $W_3$ and the general wheel $W_n$.}
    \label{fig:W3Wn}
\end{figure}

Let $W_3$ denote the wheel with three spokes graph, with its cycle basis as shown
in Figure \ref{fig:W3Wn}:
\[ c_1 = e_1 +e_4 - e_6 \  , \ c_2 = e_2+e_5-e_4 \ , \ c_3 = e_3 + e_6 -e_5\ .\]
Then the graph Laplacian, with respect to this basis, is
\[ \Lambda_{W_3} = \begin{pmatrix}  x_1+x_4+x_6 & -x_4  & -x_6 \\ 
 -x_4 & x_2+x_4+x_5 & -x_5 \\
 -x_6 & -x_5  & x_3+x_5+x_6 
\end{pmatrix}\ .
\]
For general $W_n$, we number the edges in the rim from 1 to $n$, oriented in the anticlockwise direction, and the spokes $n+1,\ldots, 2n$, oriented inwards, such that the spoke $n+i$ meets edges $i,i+1$ (where $i$ is taken modulo $n$). 
Take the cycle basis $c_i$ consisting of the  positively  oriented triangles with edges $1,n+1,2n$ (for $c_1)$,  and  $i,n+i,n+i-1$ (for $c_i$, where $i\leq 2 \leq n$): thus $c_1=e_1+e_{n+1}-e_{2n}$ and $c_i=e_i+e_{n+i}-e_{n+i-1}$ for $i \geq 2.$  With this choice of basis the Laplacian  $\Lambda_{W_n}$ is

\begin{equation} \label{introLambdaW}
\left(\begin{array}{cccccc}\!x_1\!+\!x_{n+1}\!+\!x_{2n}&-x_{n+1}&0&\ldots&0&-x_{2n}\\
-x_{n+1}&\!\!x_2\!+\!x_{n+1}\!+\!x_{n+2}&-x_{n+2}&\ldots&0&0\\
0&-x_{n+2}&\!\!x_3\!+\!x_{n+2}\!+\!x_{n+3}&\ldots&0&0\\
0&  0 & -  x_{n+3}& &0&0\\
\vdots&\vdots&\vdots&\ddots&&\vdots\\
0&0&0&\ldots& -x_{2n-2}& 0 \\
0&0&0&\ldots& x_{n-1}\!+\!x_{2n-2}\!+\!x_{2n-1}&-x_{2n-1}\\
-x_{2n}&0&0&\ldots&-x_{2n-1}&\!\!x_n\!+\!x_{2n-1}\!+\!x_{2n}\end{array}\!\!\right)\ .
\end{equation}
\vspace{0.1in}

\noindent 
It defines a family of symmetric, positive-definite matrices for all $x_e>0.$ Changing the choice of homology basis for $H_1(W_n;\ZZ)$ results in a different, but equivalent matrix $P^T \Lambda_{W_n} P$ for some $P \in \GL_n(\ZZ).$

With the help of a computer\footnote{using ad-hoc code written in Maple by the first author. See Section \ref{sectcalc} for a discussion on the use of computer algebra.} one may check that 
\[  \omega^5_{W_3} = -10 \frac{\Omega_{W_3}}{\det(\Lambda_{W_3})^2}\ ,\]
where we write $\Omega_G = \sum_{i=1}^{2n} (-1)^i x_i \dd x_1  \wedge \ldots \wedge \widehat{\dd x}_i\wedge  \ldots \wedge \dd x_{2n}$. 
Two more examples of  canonical wheel integrands  were previously known  \cite[\S10]{BrSigma} (with a different sign convention), see Corollary \ref{cor: Wheelintegrandintermsofccoeffs} and Lemma \ref{lem: cmk} below: 
\[   \omega^9_{W_5} = 18 \left( \frac{1}{\det(\Lambda_{W_5})^2} + 12 \,\frac{ x_6x_7x_8x_9x_{10}}{ \det(\Lambda_{W_5})^3} \right) \Omega_{W_5}\ ,\]
\[   \omega^{13}_{W_7} = -26 \left( \frac{1}{\det(\Lambda_{W_7})^2} +60 \,\frac{ x_8x_9\ldots x_{14}}{ \det(\Lambda_{W_7})^3}
+360 \,\frac{ (x_8x_9\ldots x_{14})^2}{ \det(\Lambda_{W_7})^4}\right) \Omega_{W_7}\]
from which one notices that the integral is non-vanishing since the integrand is strictly positive or strictly negative on the region of integration. 
The canonical integral of the graph $W_3$ is $10$ times the  Feynman residue of $W_3$ (see \eqref{IFeynWndef} below) which implies that $I_{W_3}(\omega^5) = 60 \zeta(3).$ The two computations above, via \cite{BorinskySchnetz} yielded
\[  I_{W_5}(\omega^9) = 1260 \, \zeta(5) \quad , \quad I_{W_7}(\omega^{13}) = 24024 \, \zeta(7) \]
which led to the conjectured form of the previous theorem.  Note that these are not equal to the corresponding Feynman residues, since they have lower weight (for instance, for $W_5$ and $W_7$ the usual Feynman residues are proportional to $\zeta(7)$ and $\zeta(11)$, respectively).

\subsection{Applications to graph homology} Let $\GC_2$ be the commutative, even graph complex.   By \cite[Corollary 8.8]{BrSigma}, the non-vanishing of $I_{W_n}(\omega^{2n-1})$ immediately implies: 

\begin{cor} \label{cor: wheelclassinGC} Let $n>1$ be odd.  The wheel classes are non-zero in  homology: 
\[  0\neq [W_n]  \in H_{0}\left(\GC_2\right) \ .\]
\end{cor}
\begin{remark} This corollary was previously proved by  Rossi-Willwacher \cite{RossiWillwacher} by  computing  a certain   `Kontsevich weight' associated to the wheel graphs. The value of the weight they obtain  is proportional to the canonical integral above. This suggests that   the canonical integrals may be related to the Rossi-Willwacher weights via a version of the  Schwinger trick for  Feynman integrals. 
\end{remark} 

Using the techniques of \cite{BrSigma} we may also deduce the existence of infinitely many non-zero classes 
\[ \Xi_{m,n} \in H_{k_{m,n}} (\GC_2) \]
for all odd integers $1<m<n$,   for some   odd number $1 \leq k_{m,n} \leq 2m-3 $, whose associated  `canonical integrals' with  integrand $\omega^{2m-1} \wedge \omega^{2n-1}$ are a product of two odd zeta values $\zeta(m)\zeta(n)$, see Theorem \ref{thm: Ximn} below. 
This result supports a conjecture made in \cite{BrSigma} which in particular implies the existence of an embedding of the graded exterior algebra on the canonical forms $\omega^{4k+1}$ into the cohomology of $\GC_2$.

\begin{table}[h]
\begin{center}
\vspace{0.1in} 
\begin{tabular}{c|cccccccccc}
$H_8$ &&&&&&&&&& $\0$   \\
$H_7$ &&&&&&&&& $\0$& 1   \\
$H_6$ &&&&&&&& $\0$ & 0& 0 \\
$H_5$ &&&&&&& $\0$ &  0 & 0& 0 \\
$H_4$ &&&&&&  $\0$ &  0 &  0& 0 &0\\
$H_3$ &&&&& $\0$ & $\Xi_{3,5}$   & 0&  $\Xi_{3,7}$ &1 &  2  \\
$H_2$ &&&& $\0$ & 0&0 & 0 & 0  & 0 & 0 \\
$H_1$ &&& $\0$ & 0 & 0 & 0 & 0 & 0 & 0 &  0 \\
$H_0$ && $\0$ & \color{red}{$W_3$} & 0&\color{red}{$W_5$} &0  & \color{red}{$W_7$} &   $\xi_{3,5}$ & \color{red}{$W_9$} & $\xi_{3,7}$ \\ 
\hline  
$h_G$ & $1$  & $2$ & $3$ & $4$ & $5$ & $6$ & $7$ & $8$& $9$& $10$
\end{tabular}
\vspace{0.2in} 
\end{center}
\caption{A number indicates the dimension of $H_k(\GC_2)$ at low loop order obtained by computer \cite{GraphComplexComputations}.  All entries above the blue line of zeros necessarily vanish.  All  named entries (such as `$W_n$') indicate a  one-dimensional vector space, generated  by the named element.  In this range, the location of some of the new classes $\Xi_{m,n}$ can be pinned down exactly. If we furthermore assume that $H_1(\GC_2)$ vanishes at $h_G=11$ (or more weakly, that the images  $\delta \xi_{5,7}$, $\delta \xi_{3,9}$ of  homology classes denoted by $ \xi_{5,7}$ and $ \xi_{3,9}$ in $h_G=12$ under  the edge deletion differential $\delta$ vanish), then we  can also  deduce that the two dimensional space $H_3(\GC_2)$  at $h_G=10$ in the right-most column is  spanned by $\Xi_{5,7}$ and $\Xi_{3,9}$. 
 } 
\label{default}\label{table:knownresults} 
\end{table}

Recent results of Chan, Galatius and Payne \cite{CGP} imply that the top weight cohomology of the moduli stack $\mathcal{M}_g$ of curves of genus $g$ is computed by the homology of the graph complex $\GC_2$ for $h_G=g.$ Thus the existence of the classes above give rise to top-weight classes in the cohomology  of $\mathcal{M}_g$.  

\subsection{Homology of \texorpdfstring{$\GL_n(\ZZ)$}{the general linear group}.}

Let $\mathcal{P}_n$ denote the space of positive-definite real symmetric 
$n\times n$ matrices. It has a  right action 
\begin{eqnarray} 
\mathcal{P}_n \times \GL_n(\ZZ) & \To & \mathcal{P}_n \nonumber  \\ 
X. P & \mapsto & P^T X P  \ . \nonumber 
\end{eqnarray} 
Now let $n>1$ be odd and consider the region $\tau_{W_n} \subset \mathcal{P}_n$ defined by: 
\[ \tau_{W_n} = \{ \Lambda_{W_n} (x_e) :  x_e >0 \} \ .  \]
It consists of the subspace of symmetric positive-definite matrices of banded form as depicted in \eqref{introLambdaW}. Its  diagonal elements are positive, off-diagonals are negative, and all row and column sums are positive.

\begin{thm} \label{thm: introBMclass}
Let  $H^{\mathrm{lf}}$ denote  locally finite, or Borel-Moore, homology (which is dual to compactly supported cohomology).  Let $n>1$ be odd. 

(i) The region  $\tau_{W_n}$ defines a closed, locally finite  chain which is non-zero in homology:  
\begin{equation} \label{tauWinlf} [\tau_{W_n}]  \ \in \   H^{\mathrm{lf}}_{2n} \left( \mathcal{P}_{n}/\GL_{n}(\ZZ)\right) \ .  
\end{equation}

(ii) Its inflation $\mathrm{ifl}(\tau_{W_n})$ defined in  \cite{TopWeightAg}, is  a closed locally finite chain which is non-zero in homology:
\begin{equation}  \label{inflationClasses} [\mathrm{ifl}(\tau_{W_n})]  \ \in \   H^{\mathrm{lf}}_{2n+1} \left( \mathcal{P}_{n+1}/\GL_{n+1}(\ZZ)\right) \ . \end{equation}
\end{thm}
A conjecture of Church-Farb-Putman \cite{ChurchFarbPutman} implies that if $n$ is even, then   $H^{\mathrm{lf}}_{k} \left( \mathcal{P}_{n}/\GL_{n}(\ZZ)\right)$ should vanish for $k\leq  2n-2 .$\footnote{To make the connection precise, note that by Poincar\'e duality
\[H^{\lf}_k( \mathcal{P}_n/\GL_n(\ZZ);\RR) \cong \begin{cases}  H^{\binom{n+1}{2}-k}(\mathcal{P}_n/\GL_n(\ZZ);\RR) \quad\;\;\, \hbox{ if  n is odd\ ,} \\ 
  H^{\binom{n+1}{2}-k}(\mathcal{P}_n/\GL_n(\ZZ);\det) \quad \hbox{ if  n is even\ .}
\end{cases} \] 
By        \cite[\S7.2]{ElbazVincentGanglSoule}, 
$H^k(\mathrm{SL}_n(\ZZ);\RR) = H^k(\GL_n(\ZZ);\RR) $ if $n$ is odd, and  $H^k(\mathrm{SL}_n(\ZZ);\RR) = H^k(\GL_n(\ZZ);\RR) \oplus H^k(\GL_n(\ZZ);\det)$ if $n$ is even. It follows that $H^{\lf}_k( \mathcal{P}_n/\GL_n(\ZZ);\RR)$ is a summand of 
$H^{\binom{n+1}{2}-k}(\mathrm{SL}_n(\ZZ);\RR) $ in both cases $n$ odd and even. 
Church-Farb-Putman conjecture that $H^{\binom{n-1}{2}-i}(\mathrm{SL}_n(\ZZ);\RR)=0$ for all $i \leq n-2$.} The theorem above provides  an  explicit non-vanishing class \eqref{inflationClasses} in degree $k=2n-1$.  Some of these classes have recently been constructed by Ash \cite{AshSharblies}
via a completely different method.

\begin{figure}[h]
\[
\begin{array}{c|cccccccccccccccccc}
n    & 
 \\ \hline
2  &         \\
 3 &  \color{red}{H^{\lf}_6} \quad    \\
 4 &   \quad \color{blue}{H^{\lf}_7}    \\ 
 5 &   & & \color{red}{H^{\lf}_{10}}  \quad   & & &   H^{\lf}_{15} \quad   \\ 
 6 &    & &  \quad \color{blue}{H^{\lf}_{11}} & H^{\lf}_{12} \quad  & &\quad  H^{\lf}_{16}  \\ 
 7 &   & &   &  \quad  H^{\lf}_{13} & \color{red}{H^{\lf}_{14}} &&& &   H^{\lf}_{19} & & H^{\lf}_{23} & && H^{\lf}_{28}  \\ 
\end{array}
\]
\caption{Locally-finite homology of $\mathcal{P}_n/\GL_n(\ZZ)$ in the known ranges  \cite{SouleSL3, LeeSzczarba, ElbazVincentGanglSoule}. See also \cite{GL8} for some partial results for $n=8$. An entry $H^{\lf}_i$ in the table  indicates that $H^{\lf}_i(\mathcal{P}_n/\GL_n(\ZZ);\RR)$ is non-zero and of dimension $1$.  The classes in red are spanned by $\tau_{W_3}, \tau_{W_5}, \tau_{W_7},\ldots $;  the classes indicated in blue are their images under inflation. }\label{fig:2}
\end{figure}

By a theorem of \cite{AlexeevBrunyate, MeloViviani} the cell $\tau_{W_n}$ is the Voronoi cell associated to a positive-definite quadratic form.
Indeed, it is an interesting  exercise to check that  the quadratic form
\begin{equation} 
Q_n (z_1,\ldots, z_n)= (z_1+\ldots + z_n)^2 + \sum_{i\in \ZZ/n\ZZ} \left(z_i^2 + (z_i+z_{i+1})^2\right)
\end{equation}
has $4n$ non-zero minimal vectors $\pm\eee_i$ and $\pm (\eee_i-\eee_{i+1})$ where $i\in \ZZ/n\ZZ$ and $\eee_i=(0,\ldots, 0, 1, 0,\ldots, 0)$ has a single $1$ in the $i^{\mathrm{th}}$ slot. These minimal vectors have length $4$ with respect to the norm $Q_n$,   and hence the associated Voronoi cell, which is the convex hull of the set of symmetric matrices $\xi\xi^T$, where $\xi$ ranges over  the set of  such minimal vectors of $Q_n$,  is precisely $\tau_{W_n}$. 

\begin{cor} The infinite families \eqref{tauWinlf} and \eqref{inflationClasses} of classes in the locally finite homology of the general linear group constructed above may be represented by a  single cographical cone.  
\end{cor}

We are not aware of any other classes with this property. 
\begin{remark} 
    It is reasonable to expect that the class in $H_{15}^{\lf}(\mathcal{P}_5/\GL_5(\ZZ))$, which is known by \cite{brown2023bordifications} to pair non-trivially with the class of the form   $\omega^5 \wedge \omega^9$  (denoted by $[\omega^5 \wedge \omega_c^9]$ in \emph{loc.\ cit.})   can be related to the class $\Xi_{3,5}$ in the graph complex by  application of Stokes theorem \cite[Theorem 8.5]{BrSigma}. We expect that $\Xi_{3,7}$ and  the class in  $H_{19}^{\lf}(\mathcal{P}_7/\GL_7(\ZZ))$  depicted in Figure \ref{fig:2}  are related in a similar way. 
\end{remark}

When $n$ is odd, the orbifold $\mathcal{P}_n /\GL_n(\ZZ)$ is orientable, and Poincar\'e duality implies that:
\[  H^{\lf}_k(\mathcal{P}_n /\GL_n(\ZZ);\RR) = H_{d_n- k} (\mathcal{P}_n /\GL_n(\ZZ);\RR)^{\vee} = H_{d_n- k} (\GL_n(\ZZ);\RR)^{\vee}    \]
where $d_n = \frac{n(n+1)}{2}$ is the dimension of $\mathcal{P}_n$.

\subsection{Homology of the moduli  spaces  of tropical curves and abelian varieties}
Let $ \mathcal{M}^{\trop}_g$ denote the moduli space of tropical curves of genus $g$. Denote its link by 
$\LM^{\trop}_g$, which is obtained by gluing together quotients of simplices $\sigma_G/\mathrm{Aut}(G)$ along common faces, where $G$ ranges over  all stable vertex-weighted graphs  of genus $g$. In particular, there is a natural continuous map 
\[ i: \sigma_{W_n} \To \LM_{n}^{\trop} \]
for all odd $n>1$. Let $\partial \LM_n^{\trop} \subset \LM_n^{\trop}$ denote the subspace generated by graphs with at least one vertex of weight $>0$.  An equivalent formulation of Corollary \ref{cor: wheelclassinGC}  is that   
\[  [i \sigma_{W_n} ] \ \in \  H_{2n-1} ( \LM_n^{\trop}, \partial  \LM_n^{\trop} ) \]
is non-zero for all $n>1$ odd, since the relative homology group is computed by the graph complex.  However, it is easy to show, since the wheel is built out of triangles,  that the image of $\sigma_{W_n}$ has no boundary. Thus we have:
\begin{cor} The image   $i \sigma_{W_n}$ in $\LM_n^{\trop}$ is a closed chain whose homology class is non-zero: 
\[ 0 \neq  [i \sigma_{W_n}] \ \in \ H_{2n-1}( \LM_n^{\trop})  \ . \]
\end{cor}

Now consider  $\mathcal{A}^{\trop}_g$,  the moduli space of tropical abelian varieties in genus $g$, and let $\LA_g^{\trop}$ denote its link. As a set, the former  is in bijection with $\mathcal{P}^{\mathrm{rt}}_g/\GL_g(\ZZ)$ where $\mathcal{P}^{\mathrm{rt}}_g$ is the rational closure of $\mathcal{P}_g$, which consists of positive semi-definite matrices with rational kernel.  The boundary $\partial \mathcal{A}^{\trop}_g$ is the image of $\mathcal{P}_g^{\mathrm{rt}}\setminus \mathcal{P}_g$, consisting of matrices with vanishing determinant, modulo $\GL_g(\ZZ)$. 
Thus one has
\[ \mathcal{A}_g^{\trop} \setminus \partial \mathcal{A}_g^{\trop}   \cong   \mathcal{P}_g / \GL_g(\ZZ)\ . \]
The link $\LA_g^{\trop}$ of $\mathcal{A}_g^{\trop}$ is canonically identified,   as a set, with $\RR_{>0}^{\times} \setminus \left(\mathcal{P}^{\mathrm{rt}}_g \setminus \{0\}\right)/\GL_g(\ZZ)$. 

The tropical Torelli map   \cite{Nagnibeda, Baker, CaporasoViviani, MikhalinZharkov, BMV} 
 is a continuous map
\begin{equation} \label{intro: TropTorelli}  \lambda:  \mathcal{M}_g^{\trop} \To \mathcal{A}_g^{\trop} \end{equation}
which, to a stable weighted metric graph associates the class of any  graph Laplacian matrix. 
It restricts to a map  $\mathcal{M}_g^{\mathrm{red},\trop} \rightarrow \mathcal{A}_g^{\trop}$ on the locus $\mathcal{M}_g^{\mathrm{red},\trop}  \subset  \mathcal{M}_g^{\trop}$ spanned by 3-connected graphs and passes to a map between 
their links, denoted  by   the same symbol $\lambda: \LM_g^{\mathrm{red}, \trop} \rightarrow \LA_g^{\trop}$. The wheel graphs (and all their quotients) are 3-connected, and hence $i\sigma_{W_n} \subset \LM_n^{\mathrm{red}, \trop}$.

\begin{thm} \label{thm: introLAgclass} For all odd $n>1$, the image of the  wheel cycle $\lambda i\sigma_{W_n}$ defines a  closed chain  in $\LA_n^{\trop}$ whose homology class is  non-zero:
\[   0  \ \neq \  [   \lambda  i  \sigma_{W_n}  ]  \  \in \  H_{2n-1} \left(   \LA_n^{\trop}\right)  \ .   \]
   It pairs non-trivially with the class  $[\omega_c^{2n-1} ] \in H_ {\mathrm{dR}}^{2n-1} \left(\LA_n^{\trop} \right) $ defined in \cite{brown2023bordifications} to give the canonical  integral $I_{W_n}(\omega^{2n-1})$  \eqref{wheelseq}. 
\end{thm}

It was shown in \cite{TopWeightAg, OdakaOshima} that the top weight cohomology of $\mathcal{A}_g$ 
is computed by the homology of $\LA_g^{\trop}.$ Thus the previous theorem gives rise to a family of top-weight classes in $\LA_g^{\trop}$. A very interesting question is whether the periods $I_{W_n}(\omega^{2n-1})$   are the periods of extension classes in the cohomology of $\mathcal{A}_g$, as seems to be the case for $g=3$. Indeed, Hain proved  in \cite{Hain3folds}, using results of \cite{Looijenga}, that the class in $\mathrm{gr}^W_6 H^6 (\mathcal{A}_3)$ sits in  a non-trivial extension  of $\QQ(-3)$ by $\QQ(0)$.

\subsection{Wheel motives and periods}
The wheel motives were defined in \cite{BEK}. For any finite connected graph, let  us denote the graph polynomial by 
\begin{equation}  \label{intro: psiGdef}\Psi_{G} (x_e) = \det \Lambda_{G} \  \in \  \ZZ [x_e, e\in E_G] \ .  \end{equation}
It is also known as the first Symanzik or Kirchhoff polynomial. The graph hypersurface $X_G \subset \PP^{E_G}$ is defined to be its zero locus in projective space. There exists a canonical `wonderful' compactification 
\[ \pi_G: P^G \To \PP^{E_G}\]
obtained by blowing up  certain linear subspaces of $\PP^{E_G}$ which  are contained in $X_G$, in increasing order of dimension. Let $Y_G \subset P^G$ denote the strict transform of $X_G$, and let $B \subset P^G$ denote the total transform of the coordinate hyperplanes $x_e=0$ in $\PP^{E_G}$. Then $B$ is a simple normal crossings divisor and the graph `motive' is defined to be the object 
\[ \mathrm{mot}_G = H^{|E_G|-1} \left(P^G \setminus Y_G , B \setminus (B \cap Y_G)  \right)   \]
in a suitable abelian category of (realisations of) motives.  It was shown in \cite{BEK} that if $\widetilde{\sigma}_G$ denotes the topological closure of $\pi_G^{-1}(\overset{\circ}{\sigma}_G)$ inside $P_G(\RR)$ then $\widetilde{\sigma}_G$ does not meet $Y_G.$ Thus in particular it defines
a relative homology class 
\[ [\widetilde{\sigma}_G] \  \in  (\mathrm{mot}_G^B)^{\vee}=  H_{|E_G|-1} \left(P^G \setminus  Y_G(\CC) , B \setminus  (B \cap Y_G)  (\CC) \right) \  . \]
Theorem  \ref{thm: introMain} is deduced by first proving Formula \eqref{intro: omegaformula} in the following theorem. 
\begin{thm} Let $n>1$ be odd.  The canonical form $\omega^{2n-1}_{\Lambda_{W_n}}$ defines a non-vanishing cohomology class
\[  [\omega^{2n-1}_{W_n} ]  \ \in \    \mathrm{mot}_{W_n}^ {\mathrm{dR}} =H_ {\mathrm{dR}}^{2n-1} \left(P^{W_n} \setminus Y_{W_n} , B \setminus (B \cap Y_{W_n})  \right)  \]
which pairs non-trivially with $[\widetilde{\sigma}_{W_n}]$ to give the integral \eqref{wheelseq}. We may write the form explicitly as
\begin{equation} \label{intro: omegaformula} \omega^{2n-1}_{W_n} =  \sum^{\frac{n-3}{2}}_{k=0}   c_k    \left(\frac{x_{n+1}\ldots x_{2n}}{\Psi_{W_n}}\right)^k \frac{\Omega_{2n}}{\Psi^2_{W_n}}\ ,
\end{equation}
where $x_{n+1},\ldots, x_{2n}$ are the edge variables associated to the $n$ spokes of $W_n$, $\Omega_{2n}$ is the  differential form  $ \sum_{i=1}^{2n} (-1)^i x_i \dd x_1 \wedge \ldots \wedge  \widehat{\dd x_i} \wedge \ldots \wedge \dd x_{2n}$, and the coefficients  $c_k=(-1)^{(n-1)/2}(2n-1)2^{k+1}c_{n-1,k+1}$ are given explicitly in \eqref{ceq1}.
\end{thm} 

\subsection{Computations of canonical integrals for other graphs}
\label{sectcalc} Using the formulae  derived in this paper, we  were able to calculate the canonical  forms  \eqref{intro:omegandef}
\[ \omega^{|E_G|-1}_G = \tr( (\Lambda_G^{-1} \dd \Lambda_G )^{|E_G|-1}) \]
for all  graphs $G$ with $h_G\leq7$ independent cycles. It was shown in \cite{BrSigma} that $\omega_G^{|E_G|-1}$ vanishes unless $|E_G| = 2h_G$ and $h_G$ is odd.  The graphs were generated using {\tt nauty} \cite{NAUTY}.

All calculations were performed using the computer algebra system Maple.
The actual integration is based on algorithms which are
contained in the Maple package {\tt HyperlogProcedures} \cite{hlog}\footnote{
The extra procedures which are specific to the calculation of canonical forms
and integrals are in the file `{\tt CanonicalIntegrals}' which is attached to the article
as supplementary material. After the download of {\tt HyperlogProcedures}
and {\tt CanonicalIntegrals} one can reproduce all canonical integrals up to five loops
with Maple using the commands {\tt read(CanonicalIntegrals)} and {\tt dograph(<loop number>,<graph number>)}. So, e.g., {\tt dograph(3,1)} calculates the canonical integral
of $W_3$. For most graphs at seven loops {\tt dograph} is too slow, so that it is necessary
to parallelize the computations or even  resort to exact numerical methods (see below).
Of course, the algorithms can equally well be implemented in any other computer algebra
system.}

For $h_G\leq5$ only three non-zero canonical  integrals exist, and  were computed in \cite{BrSigma}: 
$$
I_{W_3}(\omega^5)=60\zeta(3),\qquad I_{W_5}(\omega^9)=1260\zeta(5),\qquad I_{Z_5}(\omega^9)=630\zeta(5),
$$
where $Z_n$ denotes the (closed) zig-zag graph with $n$ independent cycles. 

For graphs $G$ with seven independent cycles, we used Formula  (\ref{formula2})   below  to compute the  associated invariant forms $\omega^{13}_G$. The non-zero cases involve a number of monomials which ranges between 3 (for  $W_7$)   and 61942.
We were able to calculate the invariant integrals analytically for the subset  of all {\em constructible} graphs, of  which there are 37. This class of graphs was defined in  \cite{schnetz2014graphical, BorinskySchnetz} and has period integrals (or graphical functions) which are easier to compute than the general case. 
The results for the 37 constructible canonical integrals with $h_G=7$ independent cycles are
\begin{equation} \label{IGAnsatz}
I_G = N(G)\,3003\zeta(7)\qquad\text{with}\qquad N(G)\in\{0,1,2,3,4,6,8\}.
\end{equation}
In particular,
$
I_{Z_7}(\omega^{13})=9009\zeta(7).
$
Assuming that a similar pattern continues  for non-constructible graphs, we ran numerical evaluations of the remaining canonical  integrals using
M. Borinsky's tropical Monte-Carlo method \cite{BorinskyTMCQ}, and  obtained numerical confirmations of the
ansatz \eqref{IGAnsatz} with half-integer values of  $N(G)$ for all graphs $G$ with $h_G=7.$
In total we found 56 non-zero canonical integrals  with frequencies listed in Table \ref{tabfreq}.
Note that $W_7$ is the only graph with $N(G)=8$. The only graph with $N(G)=6$ is the closed graph in Figure \ref{fig1}. It too forms part of an infinite family. 

\begin{table}[h]
\begin{center}
\begin{tabular}{l|ccccccccc}
$N(G)$&$1/2$&$1$&$3/2$&$2$&$5/2$&$3$&$4$&$6$&$8$\\
\# of graphs&8&20&6&8&1&8&3&1&1
\end{tabular}
\end{center}
\caption{Invariant integrals of graphs with seven independent cycles. Some integrals are obtained by an exact numerical method.}
\label{tabfreq}
\end{table}

\begin{figure}[ht]
\begin{center}
\fcolorbox{white}{white}{
  \begin{picture}(112,108) (56,-25)
    \SetWidth{0.5}
    \SetColor{Black}
    \Vertex(80,44){2.5}
    \Vertex(64,12){2.5}
    \Vertex(96,12){2.5}
    \Vertex(112,44){2.5}
    \Vertex(128,12){2.5}
    \Vertex(112,-20){2.5}
    \Vertex(144,-20){2.5}
    \Vertex(160,12){2.5}
    \Line(64,12)(80,44)
    \Line(80,44)(112,-20)
    \Line(96,12)(112,44)
    \Line(80,44)(112,44)
    \Line(112,44)(144,-20)
    \Line(112,-20)(128,12)
    \Line(112,-20)(144,-20)
    \Line(160,12)(144,-20)
    \Line(64,12)(160,12)
    \Arc[clock](112,28)(50.596,-161.565,-378.435)
  \end{picture}
}
\end{center}
\caption{The closed graph with $I_G(\omega^{13})=18018\zeta(7)$.}
\label{fig1}
\end{figure}
A list of all graphs up to seven loops with their canonical integrals is included in {\tt HyperlogProcedures}.

For $h_G\geq9$ one would need optimizations to calculate all  the invariant forms.
We do not expect that the general picture is significantly different  for $h_G=9$, although the existence of multiple zeta values in weight $11$ and depth 3 suggests that more interesting canonical integrals  might arise for $h_G=11$. A conjecture due to Panzer and Portner   \cite[Conjecture 1.2]{Portner} states that the canonical integrals of primitive forms $\omega^{4k+1}$ are equal  to suitably normalised  Rossi-Willwacher integrals. The  latter are  single-valued multiple zeta values, which follows from their  definition  and \cite{BrownDupont, SchlottererSchnetz,  VanhoveZerbini,  BanksPanzerPym}.
\begin{remark}
The canonical primitive forms $\omega^{4k+1}_G$    vanish unless $|E_G| =2 h_G =4k+2 $,
which corresponds to  homological degree $0$ in the graph complex $\GC_2$. By Willwacher \cite{WillwacherGRT}, $H_0(\GC_2)$ is isomorphic to the dual of the Grothendieck-Teichm\"uller Lie algebra $\mathfrak{grt}$, which  admits a map from the motivic Lie algebra $\mathfrak{g}^{\mathfrak{m}} \rightarrow  \mathfrak{grt}$. It   is injective \cite{BrMTZ}, as   conjectured by Deligne \cite{deligneP1}. The motivic Lie algebra $\mathfrak{g}^{\mathfrak{m}}$ is isomorphic to the free Lie algebra with one generator in every odd degree $\geq 3$. A question of Drinfeld \cite{Drinfeld} asks whether $\mathfrak{g}^{\mathfrak{m}}\rightarrow \mathfrak{grt}$ is surjective. If so, then this would imply that every closed and primitive element in $H_0(\GC_2)$ is proportional to a wheel class. 
This  would suggest the following consequence  for canonical integrals: if $ X \in \GC_2$ is closed with  $4k+2$  edges and   $2k+1$ loops where $k\geq 1$, and the homology class $[X]$ is primitive,  then   $I_X(\omega^{4k+1})  \in  \QQ \zeta(2k+1).$

For completely general canonical integrals of exterior products of forms $\omega^{4k+1}$ over domains $\sigma_G$, an optimistic scenario along the lines of \cite[\S9.5]{BrSigma}, is that they are periods of  mixed Tate motives over  $\ZZ$. This  is certainly  consistent with all  the current  evidence. 
However,  \cite{BelkaleBrosnan} and the modular counter-examples of \cite{BrownSchnetz} provide evidence of Feynman residues which are \underline{not} multiple zeta values, and  so it is also possible  that canonical integrals are just as complicated. 
The number-theoretic nature of canonical integrals, therefore, is an open question and  an   interesting topic for further investigation.

\end{remark}

\subsection{Strategy of proof of Theorem  \ref{thm: introMain} and contents}  Section \ref{sect: Consequences} discusses further applications of Theorem \ref{thm: introMain}
and questions for future research. It  contains the proofs of theorems \ref{thm: introBMclass}  and \ref{thm: introLAgclass},  as well as a discussion on graph homology, wheel motives, and the subtle relationship between the canonical wheel integrals and the Borel regulator.

The rest of the paper, starting in Section \ref{sec: BOmeganotations}, is concerned with  computing  the invariant differential forms
\[  \omega_{\Lambda_G}^{4k+1} = \tr \left( (\Lambda_G^{-1} \dd \Lambda_G)^{4k+1} \right)\ ,\]
where $G$ is a graph with $4k+2$ edges and $2k+1$ loops, and hence $\Lambda_G$ is $(2k+1)\times (2k+1)$ symmetric matrix whose entries are linear forms in  the variables $x_e$ for $e\in E_G$. 
It turns out that this problem is surprisingly complicated, but very highly structured. Part of the difficulty is the total absence of theorems for working with matrices whose entries  anti-commute, and so we are obliged to use  linear algebra concepts (determinants, permanents, \emph{etc}) which have historically been developed for studying matrices with \emph{commuting} entries, but whose non-commuting analogues do not yet seem to have been studied (we hope that the results in this paper are a useful first step  in this direction).

The strategy of this paper, starting from Section \ref{sec: 3},  is to solve a more  general problem concerning matrices whose entries are 1-forms,  and proceed by successive specialisation. Every specialisation produces additional structure which can be exploited at each stage. 
\begin{enumerate}
\item First of all, we let $B=(b_{ij})_{1\leq i, j\leq n}$ be an $n\times n$ matrix with generic entries, and $\Omega= (\omega_{ij})_{1\leq i, j\leq n}$  an $n\times n$ matrix whose entries are generic  one-forms. 
One can prove that since $(B \Omega)^{2n}$ vanishes identically, the matrix $(B \Omega)^{k}$ in the maximal non-vanishing degree $k=2n-1$ has additional structure; for example one sees immediately that all its entries must be divisible by $\det(B)$. 
The first goal in Section \ref{sec: 3} is to explain the terms which go into the  formula:
\[  (B \Omega)^{2n-1}  = \det(B) \sum_{\nu} \Phi_{\nu}(B) \omega_{\nu} \]
and illustrate  them with  examples. The formula itself is proven in Section \ref{Appendix1} since its proof is somewhat technical, but may be of independent interest.   The terms  $\omega_{\nu}$ are $2n-1$ forms obtained as exterior products of entries of $\Omega$; the terms $\Phi_{\nu}(B)$ are matrices whose entries are  certain polynomials in the entries of $B$ which are  the   permanents of certain additional auxiliary matrices constructed out of $B$.  Our formula for $(B\Omega)^{2n-1}$  implies a general formula for the anti-symmetrisation of a product of $2n-1$ matrices of rank $n$ which is proven in the Appendix, and extends the classical Amitsur-Levitzki theorem. It  too may  be of independent interest.

\item The second reduction is to assume that $B$ and $\Omega$ are generic \emph{symmetric}. This induces cancellations between the permanents in our formula for the trace of $(B\Omega)^{2n-1}$.  The passage to the symmetric case is discussed in Section \ref{sectsymbi}, and the main result is Theorem \ref{antisymthm}. This is the first point at which our formulae cease to be fully general: Theorem \ref{antisymthm} only concerns components of a particular (diagonal) type, but we expect that the general case can be derived  similarly. 
The main simplification at this step occurs by  applying a crucial identity which expresses a signed linear combination of permanents as a combination of \emph{determinants}.  This enables us to replace all permanents in our formula with determinants in the symmetric case. The key identity is explained in Section \ref{sect: AntiSymPerm}, Theorem \ref{thm: PermSigmaGeneric}, and may be of independent interest. Its proof is given in Section \ref{Appendix2}. Our first proof was representation-theoretic, but we settled for a recursive proof since it is ultimately shorter.

\item  The next step, which is the most straightforward, is to specialise further to the case when $B= A^{-1}$ and $\Omega =\dd A$, for $A$ a generic invertible $n\times n$  symmetric matrix. This is done in Section  \ref{sectformulaA}. A key point is that since our formula at the previous step only involves determinants of minors of $B$, they can be efficiently related to determinants of minors of $A$ by an identity due to Jacobi (Lemma \ref{detlem}). This leads to a formula for the `leading terms' of $\omega^{2n-1}_{A}$ which only involves the matrix $A$, and not its inverse (Theorem \ref{invformthm}). Section \ref{sectformula}
specialises to the   case when $A= \Lambda_G$ is a Laplacian matrix. Theorem \ref{Qdefthm} provides  a fully general and efficient formula for the canonical integrand associated to a graph which only involves determinants of minors of its graph Laplacian matrix. 
\item The final step, in Section \ref{sec8}, is to consider the specific case when $G=W_{n}$ is the  wheel with $n$ spokes, which reduces the canonical integral to an explicit linear combination of Feynman periods.  The formula for the integrand of the canonical wheel integrals is given in Corollary \ref{cor: Wheelintegrandintermsofccoeffs} and involves certain coefficients given in \eqref{ceq1}. This formula already proves that the canonical integral is positive which suffices for most of the applications in the introduction. 
The actual  wheel integrals are finally  computed  in the remainder of this section using results on graphical functions from \cite{BorinskySchnetz}. 
\end{enumerate}

\subsection{Acknowledgements}

Francis Brown thanks Trinity College, Dublin for a Simons visiting Professorship during which much of this work was carried out,  the University of Geneva, where it was completed, and M. Chan, S. Galatius, S. Grushevsky and S. Payne for  discussions on the moduli space of tropical abelian varieties. Both authors thank  Michael Borinsky, Erik Panzer,  Peter Patzt and Thomas Willwacher for  helpful  feedback.
This project has received funding from the European Research Council (ERC) under the European Union's 
Horizon 2020 research and innovation programme (grant agreement no. 724638).

\section{Further geometric applications and their proofs} \label{sect: Consequences}

We now describe in further detail some of the consequences of Theorem \ref{thm: introMain}. 
\subsection{Graph homology}

In \cite{BrSigma}, the non-vanishing of the wheel integrals for $W_3, W_5$ was used to deduce the existence of a non-trivial  class $\Xi_{3,5}$ in the homology of the graph complex $\GC_2$ (or equivalently, $\GC_0$). In view of Theorem \ref{thm: introMain}, this argument may  now be generalised to all orders.  

\begin{thm}  \label{thm: Ximn} For all  odd integers $3\leq m<n$, there exists  a  linear combination of oriented graphs  in the commutative even graph complex with edge degree $2(m+n)-1$,  denoted by:
\[  \Xi_{m,n} \in \GC_0\ , \]
which is closed, i.e.\ $\dd \ \Xi_{m,n} =0$,  and whose canonical integral satisfies
\[ I_{\Xi_{m,n}}(\omega^{2m-1}\wedge \omega^{2n-1})=\int_{\sigma_{\Xi_{m,n}}} \omega^{2m-1} \wedge \omega^{2n-1} = \zeta(m) \zeta(n)\ .  \]
It defines a non-trivial element in graph homology
\[ 0 \neq [\Xi_{m,n} ] \ \in \  H_{2(m+n)-1}  (\GC_0)\ . \]
Equivalently, it defines a non-zero class in $ [\Xi_{m,n} ]  \in  H_{k}  (\GC_2)$ in some (unknown) odd  degree $k>0$. 
\end{thm}

\begin{proof} By Willwacher \cite{WillwacherGRT}, the zeroth cohomology of the graph complex $\GC_2$ is isomorphic to the Grothendieck-Teichm\"uller Lie algebra, which, by \cite{BrMTZ} contains the free graded Lie algebra generated by one element $\sigma_{2k+1}$ in degree $-(2k+1)$ for every $k\geq 1$. By Rossi-Willwacher \cite{RossiWillwacher},
the element $\sigma_{2k+1}$  pairs non-trivially with the wheel homology class $W_{2k+1}$. We deduce the existence of a closed element $\xi_{m,n} \in \GC_2$ in degree zero, which is dual to $[\sigma_m, \sigma_n]$ and whose reduced coproduct satisfies $\Delta'  [\xi_{m,n}] =[W_m] \otimes [W_n] - [W_n] \otimes [W_m]$ in graph homology. The argument now proceeds as in  \cite[Section 10.3]{BrSigma}. 
First, by  applying the Stokes formula of  \cite[\S8]{BrSigma}, we deduce from Theorem \ref{thm: introMain} that 
\[  \int_{\delta  \xi_{m,n}} \omega^{2m-1} \wedge \omega^{2n-1} = 2 \int_{W_m} \omega^{2m-1} \int_{W_n} \omega^{2n-1}   \ \in  \  \QQ^{\times}  \zeta(m)\zeta(n)\ ,  \]
where $\delta$ denotes the `second differential' in the graph complex \cite{SpectralSequenceGC2} (and arises as a component  of $\Delta'$).  Repeated application of Stokes' formula  using Corollary 8.10 of \cite{BrSigma} produces a sequence of elements in the graph complex which terminates with a non-trivial homology class $\Xi_{m,n}$ whose canonical integral is in $\QQ^{\times}  \zeta(m)\zeta(n)$. The element $\Xi_{m,n}$ may be normalised so that its canonical integral is as stated. Since it has edge degree $2(n+m)-1$ it may be viewed as a  non-zero homology class in $\GC_0$ in that degree.
The degree of $[\Xi_{m,n}]$ in $\GC_2$ is given by $k=2(m+n)-1-2g$, where $g$ is its loop number (genus). It follows that $k$  is  odd. 
\end{proof} 

\begin{remark}  This argument is  closely related to the `waterfall' spectral sequence of \cite{SpectralSequenceGC2} but contains the additional information about the value of the canonical period integral. There is some limited control on the genus (loop number) in which the class $\Xi_{m,n}$ lies. Indeed,   the vanishing of the forms $\omega^{2m-1}, \omega^{2n-1}$ provide a lower bound: the class $\Xi_{m,n}$ must occur in genus $g$ where $g$ satisfies $\mathrm{max}\{m,n\}+1  \leq g \leq m+n-1$ by \cite[Lemma 8.4]{BrSigma}, i.e., strictly to the right of both $W_m$ and $W_n$ in  Table \ref{table:knownresults}. Equivalently, $[\Xi_{m,n}]\in H_k(\GC_2)$ where $1\leq k \leq 2 \min \{m,n\}-3 $ is odd.    \end{remark}

The theorem provides infinitely many classes of odd degree in the homology of  $\GC_2$ which can also be deduced from    \cite{SpectralSequenceGC2}. By \cite{CGP}, each one provides an odd degree class  in the cohomology of a moduli stack of curves $\mathcal{M}_g$. 
\subsection{Homology of \texorpdfstring{$\GL_n(\ZZ)$}{the general linear group}: Theorems \ref{thm: introBMclass} and \ref{thm: introLAgclass}}\pagebreak[4]

\subsubsection{Proof of Theorem  \texorpdfstring{\ref{thm: introBMclass}(i)}{.} }
There is a canonical isomorphism which identifies  the reduced homology of the genus $g$ component of the graph complex  $\GC^{(g)}_0$ with the relative homology of the link $\LM^{\trop}_g$: 
\[ \widetilde{H_{\bullet}} (\GC^{(g)}_0)   \cong  H_{\bullet-1} (\LM_g^{\trop}, \partial \LM_g^{\trop}  ) \ . \]
 It associates to an oriented stable graph $G$ with $h_G=g$ the image of the oriented projective simplex  $\sigma_G$. The link $ \lambda:  \LM_g^{\mathrm{red},\trop} \rightarrow  \LA_g^{\trop}$ of the  restriction of the tropical Torelli map \eqref{intro: TropTorelli} 
induces a map on homology:
\[ \lambda:  H_{\bullet} (\LM_g^{\mathrm{red},\trop}, \partial \LM_g^{\mathrm{red}, \trop}  ) \To   H_{\bullet} (\LA_g^{\trop}, \partial \LA_g^{\trop}  )   \cong H^{\mathrm{lf}}_{\bullet} (\LA_g^{\circ, \trop} )  \ . \]
The open $\LA_g^{\circ, \trop}=  \LA_g^{\trop}\setminus   \partial \LA_g^{\trop} $ may be canonically identified with the space $\RR^{\times}_{>0} \setminus \mathcal{P}_g /\GL_g(\ZZ)$. 
Since for all odd $n>1$ the  wheel $W_n$ is closed in the homology of the graph complex and is 3-connected, we deduce that the  link of the image $  \tau_{W_n}$ of the corresponding chain in $\mathcal{P}^{\mathrm{rt}}_n /\GL_n(\ZZ)$, which equals $ \lambda i \sigma_{W_n}$, is  closed and hence defines a homology class: 
\[   [ L  \tau_{W_n} ]   \  \in  \   H^{\mathrm{lf}}_{\bullet} (\RR^{\times}_{>0} \setminus \mathcal{P}_n /\GL_n(\ZZ))\ .  \]
 By  \cite{brown2023bordifications}, the  canonical form $\omega^{2n-1}$  defines a relative cohomology class
\[ [(\omega^{2n-1},0)] \in   H_{\mathrm{dR}}^{2n-1} \left(\LA_n^{\trop}, \partial \LA_n^{\trop} \right) \] 
which has  a compactly supported representative $\omega_c^{2n-1}$  of $\omega^{2n-1}$ on  $\RR^{\times}_{>0} \setminus \mathcal{P}_n /\GL_n(\ZZ)$. By Theorem \ref{thm: introMain}  it pairs non-trivially with $\sigma_{W_n}$, which proves, via the de Rham theorem of \cite{brown2023bordifications} for relative (co)homology, that  both $[L\tau_{W_n}] \in  H_{2n-1} (\LA_n^{\trop}, \partial \LA_n^{\trop}  ) $ and $[\omega_c^{2n-1}] \in H^{2n-1}_c(\RR^{\times}_{>0} \setminus \mathcal{P}_n /\GL_n(\ZZ))$ are non-zero. The latter fact was already established in \emph{loc.\ cit.} by different means. 
Since $ \mathcal{P}_n /\GL_n(\ZZ)$ is an $\RR^{\times}_{>0}$ fibration over $\RR^{\times}_{>0} \setminus \mathcal{P}_n /\GL_n(\ZZ)$ and $H_c^{1}(\RR^{\times}_{>0}) \cong \RR$ and vanishes in all other degrees,  we deduce that $[\tau_{W_n}]$ is non-vanishing in $H^{2n}_c(\mathcal{P}_n /\GL_n(\ZZ)) \cong      H^{2n-1}_c(\RR^{\times}_{>0} \setminus \mathcal{P}_n /\GL_n(\ZZ))$. 

\subsubsection{Proof of Theorem \ref{thm: introLAgclass}}
Theorem  \ref{thm: introLAgclass} may be proved by a similar method. The cellular homology of $\LM_g^{\mathrm{red}, \trop}$ is computed by a variant of the graph complex, which we shall denote by $\GC'_0$, in which one considers only 3-connected graphs, and tadpole graphs are not quotiented out. Since the wheel graphs have the property that every edge lies in a triangle,  they are already closed in the complex $\GC'_0.$
 It follows   that the cell $i \sigma_{W_n}$  associated to the wheel graph defines a  closed cycle of degree $2n-1$ in  $\LM_n^{\mathrm{red}, \trop}$. 
Its image under the tropical Torelli map therefore defines a closed cycle  $\lambda i\sigma_{W_n}$, which is the image of $\tau_{W_n}$ in $\LA_n^{\trop}$, whose class is
\[ [\lambda i \sigma_{W_n}] \in H_{2n-1} (\LA_n^{\trop}) \ .  \]
It is non-zero since its image  in  $H_{2n-1} (\LA_n^{\trop},\partial \LA_n^{\trop})$ is   non-zero by Theorem  \ref{thm: introBMclass}  $(i)$.

\subsubsection{Proof of Theorem  \ref{thm: introBMclass} (ii).}
A certain subcomplex of the perfect cone complex, called the inflation complex, was defined in \cite{TopWeightAg}  and was shown to  be acyclic. This  result implies that the long exact relative homology sequence:
\[ \ldots \To H_{k}(\partial \LA_g^{\trop}) \To H_k(\LA_g^{\trop}) \To H_k\left(\LA_g^{\trop}, \partial \LA_g^{\trop}\right) \To H_{k-1}(\partial \LA_g^{\trop})\To \ldots \]
splits into short exact sequences, since the inclusion of the boundary $\partial \LA_g^{\trop}$, which is isomorphic to $\LA_{g-1}^{\trop}$, into $\LA_g^{\trop}$ is trivial in homology. Thus we obtain short exact sequences
\[   0 \To H_k(\LA_g^{\trop}) \To H_k\left(\LA_g^{\trop}, \partial \LA_g^{\trop}\right) \overset{\partial}{\To} H_{k-1}(\LA_{g-1}^{\trop})\To 0   \ .  \]
Apply this to $g=n+1$ and $k=2n+1$. The inflation $\mathrm{ifl}( \lambda i\sigma_{W_n})$ defines a  chain 
in $\LA_{n+1}^{\trop}$ whose boundary is  the image of $\lambda i \sigma_{W_n}$ in $\LA_{n}^{\trop} \cong \partial \LA_{n+1}^{\trop}$.  
Thus one has 
\begin{eqnarray}   \partial: H_{2n+1}\left(\LA_{n+1}^{\trop}, \partial \LA_{n+1}^{\trop}\right) & \To&  H_{2n}(\LA_{n}^{\trop})   \nonumber \\
{[}\mathrm{ifl}(\lambda  i \sigma_{W_n})] & \mapsto & [\lambda i\sigma_{W_n}]\ . \nonumber 
\end{eqnarray}
Since, by Theorem \ref{thm: introLAgclass}, the class $[\lambda i \sigma_{W_n}]$ is non-vanishing, the same is true of ${[}\mathrm{ifl}(\lambda  i \sigma_{W_n})]$. We conclude by identifying $H_{2n+1}\left(\LA_{n+1}^{\trop}, \partial \LA_{n+1}^{\trop}\right)$
with $H_{2n+1}^{\lf}(\RR_{>0}^{\times} \setminus \mathcal{P}_{n+1}/\GL_{n+1}(\ZZ)) \cong  H_{2n+2}^{\lf}( \mathcal{P}_{n+1}/\GL_{n+1}(\ZZ)) $ and by noting that $\mathrm{ifl}(\lambda  i \sigma_{W_n})$ is the link of $\mathrm{ifl}(\tau_{W_n})$. 

\subsection{Questions meriting further investigation}

\subsubsection{Other graph complexes} It would be very interesting to apply the techniques of this paper to try to  prove  the non-vanishing of homology classes in other graph complexes (see \cite{Kontsevich93,ConantVogtmann}).

\subsubsection{Alternative approaches to Theorem \ref{thm: introMain}}
A theorem of Bismut-Lott \cite{BismutLott} implies that the canonical form $\omega^{2n-1}$ is zero in the cohomology of $\RR^{\times}_{>0} \setminus \mathcal{P}_{n}/\GL_{n}(\ZZ)$. It would be very interesting to write down an explicit primitive $\alpha^{2n}$ such that $\omega^{2n-1} = \dd \alpha^{2n}$. In this case, Stokes' formula would imply that
the canonical integral $I_{W_n}(\omega^{2n-1})$ may be computed by the integral of $\alpha^{2n}$ over the boundary of the blow-up $\widetilde{\sigma}_{W_n}$ of the simplex $\sigma_{W_n}$ associated to $W_n$, which admits a combinatorial description in terms of sub- and quotient graphs of $W_n$, which  are particularly simple. This, together with \cite{Grobner}, suggests that there may be an automorphic (and simpler) proof of Theorem \ref{thm: introMain} via residues of generalised Eisenstein series.

\subsubsection{Wheel motives and relation between canonical integrals and Feynman residues}
The wheel motives have not to our knowledge been completely worked out at the time of writing but some partial information is known \cite{BEK}. One expects that $\mathrm{mot}_{W_n}$ is a mixed Tate motive over $\ZZ$ whose weight-graded pieces are $\QQ(0)$, and $\QQ(-2k-1)$ for $1\leq k\leq 2n.$
It is known that the Feynman residue for \emph{any} $n\geq 3$ satisfies \begin{equation} \label{IFeynWndef} I^{\mathrm{Feyn}}_{W_n} =  \int_{\sigma_{W_n}}  \frac{\Omega_{2n}}{\Psi^2_{W_n}}  = \genfrac(){0pt}{}{2n-2}{n-1}\zeta(2n-3)  \end{equation}
and that $[ \frac{\Omega_{2n}}{\Psi^2_{W_n}}]$ spans a copy of $\QQ(-2n+3)$ in the highest weight  quotient of $\mathrm{mot}_ {\mathrm{dR}}$. 
It would be interesting to show that for all $n>1$  odd: \vspace{0.05in}
\begin{enumerate} 
\item $ [\omega^{2n-1}_{W_n} ]$ lies in $\mathcal{W}_{2n} \cap F^n \,  \mathrm{mot}_ {\mathrm{dR}}$ (where $\mathcal{W}$, $F$ denote the weight and Hodge filtrations), and spans a copy of $\QQ(-n)$. This should follow from the fact that  it is the pullback of the Borel class.   \vspace{0.05in}

\item For weight reasons, it vanishes in the  ordinary   cohomology group $H_ {\mathrm{dR}}^{2n-1} \left(P^{W_n} \setminus Y_{W_n}\right)$ (as opposed to the relative cohomology which defines the graph motive). 
\vspace{0.05in}

\item Therefore  by repeated application of Stokes' formula \cite{BrSigma} the canonical wheel integral for $W_n$ can be reduced to a multiple of the Feynman residue for a \emph{different} wheel graph $W_{\frac{n+3}{2}}$. 
\vspace{0.05in}
\end{enumerate}

These ideas are discussed from a slightly different perspective in \cite[\S9.5.1]{BrSigma}. 
\subsection{Regulators} 
The relationship between regulators and the integrals \eqref{intro: IGdef} is subtle, since the Borel regulator involves the pairing between \emph{stable} homology and cohomology classes, whereas \eqref{wheelseq}  involves \emph{unstable}, \emph{locally finite} classes. To compare the two,  let $n>1$ be odd.

On the one hand, the wheel integral  for $\zeta(n)$   is the integration pairing between 
\begin{equation} \label{WheelSetUp}
[\sigma_{W_n}]\in H^{\lf}_{2n-1}(L\mathcal{P}_n/ \GL_n(\ZZ))   \quad \hbox{ and }  \quad  \omega^{2n-1}  \ , 
\end{equation} where the latter differential form represents a relative cohomology class.
 The regulator
\begin{equation} \label{regulator} r_{2n-1}: K_{2n-1}(\ZZ) \otimes_{\ZZ}\QQ \rightarrow \CC \ , 
\end{equation}
on the other hand, is obtained as the  pairing between  a rational homology class:
\begin{equation} \label{RegulatorSetUp} 
[\gamma_n] \in H_{2n-1}^{\lf}(L\mathcal{P}_g/ \GL_g(\ZZ)) \qquad \hbox{ and } \qquad  [\omega^{2n-1}] \in H^{2n-1} (L \mathcal{P}_g/\GL_g(\ZZ))\ , 
\end{equation}
for any $g$ in the stable range.  Since the $\omega^{2n-1}$ form a compatible projective system on the $L \mathcal{P}_g/\GL_g(\ZZ)$ with respect to the  stabilisation maps $L \mathcal{P}_g/\GL_g(\ZZ) \rightarrow L \mathcal{P}_{g+1}/\GL_{g+1}(\ZZ)$, the strongest  possible result is obtained when  $g$  is as small as possible, provided that one can define $\gamma_n$ as in \eqref{RegulatorSetUp}. Note that,   in order to compute the regulator,  Borel takes  $g$ very large. The smallest possible value of $g$ satisfies $n<g\leq n+2$, since by \cite{BismutLott}, we know that $\omega^{2n-1}$ is trivial in the  cohomology of $L \mathcal{P}_g/\GL_g(\ZZ)$ for $g=n$, but non-zero in its cohomology for all $g\geq n+2$ (by \cite{brown2023bordifications}). It is expected that $g=n+2$ is optimal.  
The upshot of this discussion is that the regulator and the wheel integrals  (or even their inflations) necessarily take place in different locally symmetric spaces      $L \mathcal{P}_g/\GL_g(\ZZ) $   for different values of $g$.

For instance, when $n=3$  the wheel integral for $\zeta(3)$ may be computed  on $\GL_3(\ZZ)$ (or, at a push, after  inflation on $\GL_4(\ZZ)$), but the regulator may  only be computed on  $\GL_g(\ZZ)$ for $g\geq 5.$

Nevertheless, it is in fact possible to relate the wheel and regulator integrals, in the following way. Let $g>3$ be odd and let $\mathcal{F}_g \in H^{\lf}_{d_g}( \RR_{>0}^{\times} \setminus P_g / \GL_g(\ZZ))$  denote the fundamental class in locally finite homology. We  shall use the  coaction on homology which is dual to the cup-product between compactly supported and non-compactly supported cohomology. Denote it by  
\[ \Delta:    H^{\lf}_{d_g}  \To  \bigoplus_{m+n= d_g} H^{\lf}_m \otimes H_n \ . \]
Let $\alpha_g \in  H^{d_g-2g+3}(L\mathcal{P}_g/\GL_g(\ZZ);\QQ )$ denote a non-zero rational (singular) cohomology class which is proportional (over $\RR^{\times}$) to the class of the form 
$ \omega^5 \wedge \omega^9 \wedge \ldots \wedge \omega^{2g-5}$. The latter is non-zero  by \cite{brown2023bordifications}.

We deduce the existence of a  rational homology   class 
\begin{equation} \label{rhoclass}   \rho_g: = \left( \mathrm{id} \otimes  \alpha_g\right)  \Delta \mathcal{F}_g  \ \in  \  H^{\lf}_{2g-1}( L\mathcal{P}_g/\GL_g(\ZZ);\QQ) \end{equation}
which has the property that 
\[ \mathrm{vol}(\mathcal{F}_g)  = \int_{\mathcal{F}_g} \omega^5 \wedge \omega^9 \wedge \ldots \wedge \omega^{2g-1} =  \int_{\rho_g} \omega^{2g-1}  \times  \int_{\phi_g}  \omega^5 \wedge \omega^9 \wedge \ldots \wedge \omega^{2g-5}  \]
for some rational  homology class $\phi_g$. The class $\phi_g$ may in turn be decomposed via the coproduct on ordinary homology (dual to the cup-product on cohomology) into classes  
\[ [\gamma_{2n-1} ]  \ \in \ H_{2n-1} (L\mathcal{P}_g / \GL_g(\ZZ);\QQ)    \]
such that 
\[ \int_{\phi_g}  \omega^5 \wedge \omega^9 \wedge \ldots \wedge \omega^{2g-5}  = \int_{\gamma_5} \omega^5 \times \ldots  \times \int_{\gamma_{2g-5}} \omega^{2g-5} \ ,  \]
which, by the above discussion, is proportional over $\QQ^{\times}$ to a  product of  regulators $r_5 r_9 \ldots r_{2g-5}$ (where we denote by $r_{2n-1}$ the value of \eqref{regulator} on a generator). Putting the pieces together, we deduce that 
\[   \int_{\rho_g} \omega^{2g-1}   = \frac{   \mathrm{vol}(\mathcal{F}_g)   } {\int_{\phi_g}    \omega^5 \wedge \omega^9 \wedge \ldots \wedge \omega^{2g-5}      } = \frac{   \mathrm{vol}(\mathcal{F}_g)   } { r_5 r_9 \ldots r_{2g-5}}  \ .   \]
By Minkowski, $ \mathrm{vol}(\mathcal{F}_g) $ is proportional to  $\zeta(3)\zeta(5) \ldots \zeta(2g-1) $. By Borel, the regulator $r_{2n-1}$ for $n>1$ odd is proportional to $\zeta(n)$. Thus we conclude that, up to a non-zero rational multiple
\[  \int_{\rho_g} \omega^{2g-1}   =  \frac{ \zeta(3) \zeta(5) \ldots \zeta(2g-1)}{\zeta(3) \zeta(5)\ldots \zeta(2g-5)} = \zeta(2g-1) \ .  \]
We expect that the class $\rho_g$ constructed in this manner out of the locally finite fundamental class is proportional to the wheel class $[\sigma_{W_g}]$.  It is true for  $g=3,5,7$ since in this case $H^{\lf}_{g}(\GL_g(\ZZ);\QQ)$ is one-dimensional  (see Table 1). Then the previous calculation provides a conceptual explanation for why the wheel integrals are single zeta values, and a  connection between them and regulators: in brief,  from this point of view, the wheel integrals are dual to a product of regulators.  

The argument may of course be reversed and gives a computation of Borel regulators in terms of wheel integrals. The first interesting case (where  $\sim_{\QQ^{\times}}$ denotes equality up to an element in $\QQ^{\times}$) is 
\[ r_5 \sim_{\QQ^{\times}}  \frac{\mathrm{vol}(\mathcal{F}_7)}{ I_{W_5}} \sim_{\QQ^{\times}}   \frac{\zeta(3) \zeta(5)}{\zeta(5)} \sim_{\QQ^{\times}}   \zeta(3) \ .\]

\begin{remark}
Curiously, the  interpretation of wheel integrals as periods of the wheel motives gives a different  way to relate them to algebraic  $K$-theory. Indeed, one expects that the wheel period integrals correspond to a simple  extension  of $\QQ(-n)$ by $\QQ(0)$, which corresponds to a class in $K_{2n-1}(\ZZ)\otimes_{\ZZ} \QQ$.
\end{remark}

This completes our discussion of cohomological implications of Theorem \ref{thm: introMain}.

\section{Formula for  generic  \texorpdfstring{$(B\Omega)^{2n-1}$}{matrices}} \label{sec: 3}
In this section we state a formula for $(B\Omega)^{2n-1}$ where $B$ is an $n\times n$ matrix with generic entries  $b_{ij}$ and $\Omega$ is an $n\times n$ matrix whose entries are generic differential one-forms $\omega_{ij}.$ 
The formula breaks up into isotypical pieces indexed by \emph{types} $\nu$, involving differential  forms $\omega_{\nu}$ of degree $2n-1$ in the entries of $\Omega$, and matrices $\Phi_{\nu}(B)$ whose entries are \emph{permanent polynomials}  in the $b_{ij}$.
 More precisely: 
\[   (B\Omega)^{2n-1} = \left( \det B\right) \sum_{\nu} \Phi_{\nu}(B) \,\omega_\nu \ . \]
The  following example may give a sense of the content of this formula.

\begin{ex}  \label{ex:Bomeganis2} For $n=2$, let 
\[ B  = \begin{pmatrix} b_{11} & b_{12} \\ b_{21} & b_{22} \end{pmatrix} 
\qquad \hbox{ and } \qquad   \Omega=   \begin{pmatrix} \omega_{11} & \omega_{12} \\ \omega_{21} & \omega_{22} \end{pmatrix} \ . \] 
Then the formula above reduces to  the following expression for $(B\Omega)^3 = B \Omega B \Omega B \Omega$:
 \begin{align*}
(B\Omega)^3 &= (\det B) \Bigg(\left(\begin{array}{cc}2b_{11}&0\\b_{21}&b_{11}\end{array}\right)
\omega_{11}\wedge\omega_{12}\wedge\omega_{21}
+\left(\begin{array}{cc}b_{21}&b_{11}\\0&2b_{21}
\end{array}\right)
\omega_{11}\wedge\omega_{12}\wedge\omega_{22}\\
&\qquad\quad\,+\,\left(\begin{array}{cc}2b_{12}&0\\b_{22}&b_{12}\end{array}\right)
(-\omega_{11}\wedge\omega_{21}\wedge\omega_{22})
+\left(\begin{array}{cc}b_{22}&b_{12}\\0&2b_{22}\end{array}\right)
(-\omega_{12}\wedge\omega_{21}\wedge\omega_{22})\Bigg).
\end{align*}
\end{ex}

Although it will not be required for the time being, we record the following statement, which follows from the same method of proof of \cite[Proposition 4.5]{BrSigma}.

\begin{prop}  \label{prop: BOmega2nvanishes}
For $B, \Omega$ as above, $(B \Omega)^{2n}=0$. 
\end{prop}
This section is therefore devoted to computing the highest non-vanishing power of $(B \Omega)$. 

\subsection{Basic set-up and notations}  \label{sec: BOmeganotations}

\subsubsection{Matrices}  \label{subsubsectMatrices}
Let $I, J$ be ordered sets, with possible repeats, of equal length $n\geq 1$. We denote by 
\[ B_{I,J} = (b_{ij})_{i\in I, j\in J}\]
the $n\times n$ matrix whose entries are indexed by  the elements of $I,J$ in order. 
For example, we have
$$
B_{12,23}=\left(\begin{array}{cc}b_{12}&b_{13}\\b_{22}&b_{23}\end{array}\right),\quad B_{22,23}=\left(\begin{array}{cc}b_{22}&b_{23}\\b_{22}&b_{23}\end{array}\right) \ . 
$$
When the integer $n\geq 1$ is clear from the context, we shall simply write
\[ B = B_{\{1,\ldots, n\}, \{1,\ldots, n \}} = (b_{ij})_{1\leq i,j\leq n}\ \]
to be the $n\times n$ matrix with (generic) entries $b_{ij}$ in row $i$ and column $j$.

\subsubsection{Forms} \label{sect:Forms} Let $\Omega= (\omega_{ij})_{1\leq i, j \leq n}$ denote the $n\times n$ matrix whose entries are generic one-forms. 
To make sense of the product $B\Omega$, and its powers, we work in the graded-commutative  graded algebra (or DGA with zero differential) $R= \bigoplus_{m\geq 0} R^m$,  
where  $R^0 = \QQ[b_{ij}]$ is the polynomial ring generated by the entries of $B$, and $R^1= \bigoplus_{i,j} R^0 \omega_{ij} $ 
is the free $R^0$-module generated by the entries of $\Omega.$ The degree $m$ component $R^m$ is generated by exterior products of degree $m$ in the elements of $\Omega$.  The entries of $(B\Omega)^{2n-1}$ lie in the component of degree $2n-1$:
\[ R^{2n-1}  =  \bigoplus_{\nu}  \left(  R^0  \!  \bigwedge_{(i,j) \in \nu } \omega_{ij} \right)\ ,\]
where $\nu$ is a \emph{type} of rank $n$,  which is defined to be a subset 
\begin{equation}  \label{defn: nu} \nu  \ \subset \  \{ (i,j): 1\leq i, j\leq n\} \end{equation}
 of cardinality $2n-1$. 
 \begin{remark} \label{rem:typeisgraph}
 A type   may be interpreted as a bipartite graph   with $2n$ vertices (indexed by two sets of integers numbered  from $1$ to $n$), with  $2n-1$ edges between them.
(Note that in Section \ref{sect:BIn}, we shall interpret a type differently, namely as an oriented graph with $n$ vertices (and self-loops),
 where a pair $(i,j)$ is an edge from vertex $i$ to vertex $j$.)
 \end{remark}

\subsubsection{Isotypical components} 
There is a natural projection of graded algebras 
$ R \rightarrow  R_{\nu}$, 
where $R_{\nu} \subset R$ is the graded subalgebra generated over $R_0$  by the $\omega_{ij}$ for only those $(i,j) \in \nu$.
It is defined by sending $\omega_{kl}$ to zero for all $(k,l) \notin \nu. $ It extends to a map on matrices whose elements are in $R$
\[ M_{n\times n}(R) \To M_{n\times n}(R_{\nu})\ , \]
which we  denote by $X\mapsto X_{\nu}$. Since this map is an algebra homomorphism we immediately deduce:

\begin{prop}  \label{propBOmegaasSumoftypes} The following identity holds in the ring $M_{n\times n}(R)$:
   \begin{equation} \label{BOmegaAsSumIsotypical} (B \Omega)^{2n-1}  =  \sum_{\nu}  (B \Omega_{\nu})^{2n-1} \ ,
   \end{equation}
   where $\Omega_{\nu}$ is the $n\times n$ matrix whose entries are all zero except for $\omega_{ij}$ in position $(i,j)$
whenever $(i,j) \in \nu$.
\end{prop}
In particular, it suffices to find a formula for $ (B \Omega_{\nu})^{2n-1}$.

\begin{ex} \label{example of types}
    Consider the following three types of rank $3$:
\[ \nu_1 = \{11,12,22,23,33\} \ , \   \nu_2 = \{11,12,13,22,33\} \ ,  \  \nu_3 = \{13,23,33,31,32\}\ . \]  
The corresponding matrices of one-forms are 
\[\Omega_{\nu_1} = \begin{pmatrix} \omega_{11} & \omega_{12} & \\ & \omega_{22} & \omega_{23} \\ & & \omega_{33} 
\end{pmatrix} \quad , \quad   
\Omega_{\nu_2} = \begin{pmatrix} \omega_{11} & \omega_{12} &  \omega_{13}\\ & \omega_{22} &  \\ & & \omega_{33} 
\end{pmatrix} \quad  , \quad   \Omega_{\nu_3} = \begin{pmatrix} & &  \omega_{13} \\ 
&&\omega_{23}\\
\omega_{31} &\omega_{32}& \omega_{33}
\end{pmatrix}\ ,
\] 
where here, and elsewhere, blank entries in a matrix denote an entry which is zero.  
\end{ex}

\subsection{Weights and differential forms associated to a type}  \label{sect: weightsoftypes} To any type $\nu$ of rank $n$ we will associate a pair of weight vectors, which will play  
a   prominent role in the theory. 
\begin{defn}\label{defn: weight}
To any type $\nu = \{(i_k,j_k): k=1,\ldots, 2n-1\} $ of rank $n$,  we define its \emph{weight vectors} 
\[  w(\nu) = (\pp(\nu), \qq(\nu))\ , \]
 where $\pp(\nu) = (\pp(\nu)_1,\ldots, \pp(\nu)_n)$, and  $ \qq(\nu) = (\qq(\nu)_1,\ldots,\qq(\nu)_n)$ are row vectors of length $n$ defined by
 \[ \pp(\nu)_i =  \left|\{i: (i,j) \in \nu \hbox{ for some } j \}  \right|\qquad \ , \qquad  \qq(\nu)_j =  \left| \{j: (i,j) \in \nu \hbox{ for some } i  \} \right| \ .  \]
\end{defn}
\begin{remark} As in Remark \ref{rem:typeisgraph}, the entries in the  weight vectors of a type are simply the  degrees of each vertex in the corresponding bipartite graph. 
    \end{remark}

  For the three types considered in Example \ref{example of types} we have
\[ w(\nu_1) = ((2,2,1),(1,2,2))\ ,\qquad
 w(\nu_2) = ((3,1,1),(1,2,2))\ ,\qquad
 w(\nu_3) = ((1,1,3),(1,1,3)) \ .   \]

We next associate a differential form $\omega_{\nu}$ to a type $\nu$ as follows.

\begin{defn} \label{defn: omeganu}
Choose any ordering on the elements in $\nu$ (for instance, the lexicographic ordering).

(i). Define a $(2n-1) \times 2n$ matrix  $M_{\nu}$ 
whose rows are indexed by the elements in $\nu$. In the $kth$ row indexed by $(i_k,j_k) \in \nu$, place a $-1$ in column $i_k$, a $1$ in column $n+j_k$, and a zero in every other column.

(ii). Define a differential $(2n-1)$-form $\eta_{\nu} \in R^{2n-1}$ to be 
\[ \eta_{\nu} = \omega_{i_1,j_1} \wedge \omega_{i_2,j_2} \wedge \ldots \wedge\omega_{i_{2n-1}, j_{2n-1}} \]
where $\nu = ((i_1,j_1), (i_2,j_2), \ldots ,  (i_{2n-1},j_{2n-1}) )$, in that order.

(iii). Choose any $1    \leq k \leq 2n$, and let  $M_{\nu}^{\emptyset,k}$ be the  $(2n-1) \times (2n-1)$  matrix obtained by deleting the $k^{\mathrm{th}}$ column from  $M_{\nu}$.  Finally, define
\begin{equation}  \label{omeganudefnANDepsilonk}
\omega_{\nu} = (-1)^{\binom{n}{2}+k-1}\det(M_{\nu}^{\emptyset,k})\, \eta_{\nu}.
\end{equation}
\end{defn}
\begin{remark} As in Remark \ref{rem:typeisgraph}, the matrix $M_{\nu}$ is simply the (signed) edge-vertex incidence matrix of the bipartite graph $\Gamma_{\nu}$ associated to $\nu$, after choosing some ordering on the edges.  It follows from the exact sequence
$0 \rightarrow H_1(\Gamma_{\nu};\ZZ) \rightarrow \ZZ^{E_{\Gamma_{\nu}}} \overset{M_{\nu}^T}{\rightarrow} \ZZ^{V_{\Gamma_{\nu}}} \rightarrow H_0(\Gamma_{\nu};\ZZ) \rightarrow 0$ that the incidence matrix has rank $2n-1-b_1=2n-b_0$, where $b_i = \mathrm{rank}\, H_i(\Gamma_{\nu};\ZZ)$ ($i=0,1$) are the Betti numbers of $\Gamma_{\nu}$.  In particular, $\omega_{\nu}$ is non-zero if and only if $\Gamma_{\nu}$ is connected and cycle-free, in which case it is a tree.
This implies Lemma \ref{onulem} (ii), (iv), and (v) stated below.
\end{remark}
    
\begin{ex}\label{ex3M}
For the three types considered in Example \ref{example of types} we have
\[ 
M_{\nu_1} 
= \left(\begin{array}{ccc|ccc}  
\!\!-1  &  &   &\!1  & &   \\ 
\!\!-1  &  &   &  & 1 &   \\ 
  & \!\!\!\!-1 & &   &1  &   \\ 
  & \!\!\!\!-1 & &  & &1   \\ 
  &  & \!\!\!\!-1\!&  & & 1   
\end{array} \right)
, \ 
M_{\nu_2} 
= \left(\begin{array}{ccc|ccc}  
\!\!-1  &  &   &\!1  & &   \\ 
\!\!-1  &  &   &  & 1 &   \\ 
\!\!-1  &  & &   &  & 1   \\ 
  & \!\!\!\!-1 & &  & 1&    \\ 
  &   & \!\!\!\!-1\!&  & & 1   
\end{array} \right)
, \ 
M_{\nu_3} 
= \left(\begin{array}{ccc|ccc}  
\!\!-1  &  &   &  & & 1  \\ 
  & \!\!\!\!-1 &   &  &  & 1  \\ 
  &  & \!\!\!\!-1\!&   &  & 1   \\ 
  & & \!\!\!\!-1\!&\! 1 &  &    \\ 
  &  & \!\!\!\!-1\!&  & 1 &     
\end{array} \right).
\]
The vertical lines are merely visual aids.  The reader may check that the sums of the entries in each column give back the weight vectors $w(\nu)$ (with $-\pp(\nu)$ to the left of the vertical line, and $\qq(\nu)$ to the right).

The corresponding differential forms are 
\begin{align*}
\omega_{\nu_1}&=\omega_{11}\wedge\omega_{12}\wedge\omega_{22}\wedge\omega_{23}\wedge\omega_{33} \ ,\\
\omega_{\nu_2}&=-\omega_{11}\wedge\omega_{12}\wedge\omega_{13}\wedge\omega_{22}\wedge\omega_{33}  \ ,\\
\omega_{\nu_3}&=-\omega_{13}\wedge\omega_{23}\wedge\omega_{33}\wedge\omega_{31}\wedge\omega_{32}  \ .
\end{align*}

\end{ex}

\begin{lem}\label{onulem} The differential form $\omega_\nu$ has the following fundamental properties:

 $(i).$ It does not depend on any choices. 

 $(ii).$ If $\pp(\nu)_i=0$ or $\qq(\nu)_i=0$ for some $i\in\{1,\ldots,n\}$,
 then $\omega_\nu=0$.

 $(iii).$ For any $\nu$, we have $\det(M_{\nu}^{\emptyset,k})\in\{-1,0,1\}$.
 
 $(iv).$ If there exist $i_1,i_2, j_1,j_2$ such that 
$ \{ (i_1,j_1), (i_1,j_2), (i_2,j_1), (i_2,j_2)\} \in \nu $, or equivalently, if 
 $\Omega_\nu$ contains a two-by-two submatrix consisting of non-zero entries,  then $\omega_\nu=0$.

$(v).$ If $\Omega_\nu$ is block-diagonal with $\geq 2$ blocks then $\omega_\nu=0$.

$(vi).$ For any set $\bar\nu$ of $2n$ pairs in $\{1,\ldots,n\}$ and any
$i\in\{1,\ldots,n\}$ we have
\begin{equation}\label{Onubar1}
\sum_{k:(i,k)\in\bar\nu}\omega_{ik}\wedge\omega_{\bar\nu\setminus (i,k)}=0=\sum_{k:(k,i)\in\bar\nu}\omega_{ki}\wedge\omega_{\bar\nu\setminus (k,i)}\ .
\end{equation}

$(vii)$. Suppose that $\nu$ contains $(i,i), (j,j)$ and $(i,j)$ but not $(j,i)$ for some indices $i\neq j$. Let $\nu'$ denote the type  obtained from $\nu$ by replacing $(i,j)$ with $(j,i)$. Then  there exists a $(2n-2)$-form $\alpha$ such that: 
\begin{equation} \label{flipomegaij} \omega_{\nu} = \omega_{ij} \wedge \alpha  \qquad \hbox{ and } \qquad \omega_{\nu'} = -  \omega_{ji} \wedge \alpha\ .
\end{equation}
In other words,  replacing $(i,j)$ with $(j,i)$ in  $\nu$  replaces $\omega_{ij}$ with $-\omega_{ji}$.  
\end{lem}
\begin{proof}
$(i).$ Consider the matrix obtained from  $M_\nu$ by adding an additional  row $(1,0,\dots,0,-1,0,\ldots,0)$ where the entry $-1$ is in the $k^{\mathrm{th}}$ position. 
This yields a $2n\times 2n$ matrix $\bar M_\nu$ which annihilates the column vector
$(1,\ldots, 1)^T$.
It follows that $\det(\bar M_{\nu})=0$. Expansion along the
final row shows that $\omega_\nu$ does not depend on the choice of $k$. Furthermore, $\omega_{\nu}$ does not depend on the choice of ordering of $\nu$, since changing the ordering by a  permutation $\sigma$ corresponds to performing a row permutation on  $M_{\nu}$, which multiplies  $\det(M_{\nu}^{\emptyset,k})$ by $\mathrm{sign}(\sigma)$. Likewise, permuting the factors  in $\eta_{\nu}$  by $\sigma$ has the exact same effect. These two effects cancel since $\mathrm{sign}(\sigma)^2=1.$

$(ii).$   If $\pp(\nu)_i=0$ then $M_{\nu}^{\emptyset,k}$ has zero column $i$ for $k\neq i$.
If $\qq(\nu)_i=0$ then $M_{\nu}^{\emptyset,k}$ has zero column $n+i$  for $k\neq n+i$. In either case
$\det(M_\nu^{\emptyset,k})=0$.

$(iii).$  
Either one column of $M_{\nu}^{\emptyset,k}$ for $k>n$ is empty or there exists a column $i\neq k$, $i\leq n$ with a single non-zero entry in some row corresponding to a pair $(i,j)\in\nu$. In the first case $\det(M_{\nu}^{\emptyset,k})=0$. In the second case we expand the determinant
along this column yielding $\det(M_{\nu}^{\emptyset,k})=\pm\det(M_{\nu\setminus(i,j)}^{\emptyset,k-1})$. Proceed by induction to obtain  the result.

$(iv).$  
The alternating sum of the four rows  of $M_{\nu}^{\emptyset,k}$ labelled by
$(i_1,j_1), (i_1,j_2), (i_2,j_2), (i_2,j_1)$ is zero. Therefore  $\det(M_{\nu}^{\emptyset,k})=0$.

$(v).$ Assume $\Omega_\mu$ has two blocks  corresponding to a disjoint union $\nu = \nu_1  \cup \nu_2$ where $\nu_1$ (resp.\ $\nu_2$) are the row labels of the first (resp. second) block.  Then $M_{\nu}  = M_{\nu_1} \oplus M_{\nu_2}$ where $M_{\nu_i}$ is the submatrix of $M_{\nu}$ 
with rows in $\nu_i$. Since  both $M_{\nu_1}$ and $M_{\nu_2}$
 have corank $\geq 1$, it follows that  $M_{\nu}$ has corank $\geq 2$. 

$(vi).$ Consider the $2n\times 2n$ matrix $M_{\bar \nu}$, defined in a similar manner to $M_{\nu}$, whose rows correspond to
$\bar\nu=\{(i,j_1),(i,j_2),\ldots\}$ ordered such that  all indices of the form $(i,k) \in \bar \nu$ with initial term $i$ occur before all indices of the form $(j,k)$ with $j\neq i.$
We have $\det(M_{\bar\nu})=0$ because $M_{\bar\nu}$ annihilates
the column vector $(1,\ldots, 1)^T$. 
Expanding $\det(M_{\bar\nu})=0$ along column $i$ yields
$\sum_{k:(i,j_k)\in\bar\nu}(-1)^{i+k}\det(M_{\bar\nu\setminus (i,j_k)}^{\emptyset,i})=0$.
Multiplying this equation with $\eta_{\bar\nu}$ implies the  identity on the left,  because moving $\omega_{ij_k}$ in $\eta_{\bar\nu}$ to the first slot produces a sign $(-1)^{k-1}$ which compensates the $k$-dependence of the sign in the sum.
To prove the second identity we reorder $\bar\nu$ to start with pairs whose second entry is $i$ and expand along column $n+i$.

$(vii)$. Consider $\bar\nu=\nu\cup (j,i)=\nu'\cup (i,j)$. Because of $(iv)$ the sums in
(\ref{Onubar1}) only have the two non-zero terms $k=i,j$. We define the $(2n-2)$-form $\alpha$
by $\omega_{ij}\wedge\omega_{ji}\wedge\alpha=\omega_{ii}\wedge\omega_{\bar\nu\setminus (i,i)}$ and obtain
$$
\omega_{ij}\wedge\omega_{\nu'}=-\omega_{ij}\wedge\omega_{ji}\wedge\alpha=
\omega_{ji}\wedge\omega_{\nu} \ .
$$
Equation (\ref{flipomegaij}) follows.
\end{proof}
Although $(vi)$ is not directly used in this paper, it may be of use in subsequent applications since it  expresses the compatibility of our formula for $\Omega^{2n-1}_{\nu}$ with the equation $\Omega_{\bar\nu}^{2n}=0$.

\begin{prop}\label{prop: omegaexpand}
Let $\nu$ be a type of rank $n$.
Assume the matrix $\Omega_\nu$ has a column $c$ containing a single
non-zero entry $\omega_{ic}$.
Further assume that after deletion of column $c$ the matrix $\Omega_\nu^{\emptyset,c}$ has a row $r$ with a single non-zero
entry $\omega_{rj}$. Then
\begin{equation}
\omega_\nu=(-1)^{c+r}\omega_{ic}\wedge\omega_{rj}\wedge\omega_{\nu\setminus\{(i,c),(r,j)\}}\ ,
\end{equation}
where $\nu\setminus\{(i,c),(r,j)\}$ is considered of rank $n-1$ with
no row $r$ and no column $c$.
If either  assumption does not hold  then $\omega_\nu=0$.
\end{prop}
\begin{proof}
We write $\nu=\{(i,c),(r,j)\}\cup\nu'$ in this
order, where  the union is disjoint ($\nu'$ consists of all pairs $(k,\ell) \in \nu$ where $k\neq r$, $\ell \neq c$). 
In the definition \eqref{omeganudefnANDepsilonk} of $\omega_\nu$ we choose $k=2n-\delta$ where $\delta=0$ if $c<n$ and $\delta=1$ if $c=n$. (The column labelled by $n+c$  in $M_{\nu}$ is therefore retained in $M^{\emptyset,k}_\nu$.) We expand the determinant of $M_{\nu}^{\emptyset, k}$ along column  $n+c-\delta$ which has
the single entry $1$ in row $1$. After passing to the corresponding minor (delete the first row and  $(n+c-\delta)$th column) the resulting matrix has a  single non-zero entry $-1$ in column $r$ and row $1$. Deleting this row and column gives the minor $M_{\nu'}^{\emptyset, 2n-2}$.  We therefore  obtain
$$
\omega_\nu=(-1)^{\binom{n}2+k-1} (-1)^{1+n+c-\delta}(-1)^{r}\det(M_{\nu'}^{\emptyset,2n-2})\,\omega_{ic}\wedge\omega_{rj}\wedge\eta_{\nu'}\ ,
$$
where $k=2n-\delta$. Because $\binom{n}2=\binom{n-1}2+n-1$, the result follows from (\ref{omeganudefnANDepsilonk}).

If $\Omega_{\nu}$ has a zero row or column, then $\omega_\nu=0$ by Lemma \ref{onulem} $(ii)$. 
It remains to consider the case when the row $r$ and column $c$ share a common element $\omega_{rc}$ and all other rows and columns have two non-zero
entries (adding up to $2n-1$ entries in total).
In this case we choose $k\neq r,n+c$ in Definition \ref{defn: omeganu} and find
that the column in $M_\nu^{\emptyset,k}$ which refers to $r$ is the negative
of the column that refers to $c$ (namely
$(0,\ldots,0,-1,0,\ldots,0)^T$). Therefore the corank of $M_\nu^{\emptyset,k}$ is $\geq1$
and $\omega_\nu=0$.
\end{proof}
Proposition \ref{prop: omegaexpand} describes an inductive way to compute $\omega_\nu$.

\begin{ex}\label{n2ex}  
(i).   For $n=2$ there are four possible types: 
\begin{align*}
\omega_{\{11,12,21\}}=\omega_{11}\wedge\omega_{12}\wedge\omega_{21},&\qquad\omega_{\{11,12,22\}}=\omega_{11}\wedge\omega_{12}\wedge\omega_{22},\\
\omega_{\{11,21,22\}}=-\omega_{11}\wedge\omega_{21}\wedge\omega_{22},&\qquad\omega_{\{12,21,22\}}=-\omega_{12}\wedge\omega_{21}\wedge\omega_{22}.
\end{align*}

(ii).   \label{nex}
Let $\nu=\{11,22,\ldots,nn,12,23,\ldots,(n-1)n\}$ so that $\Omega_\nu$ has non-zero entries along the diagonal and the line immediately above the diagonal. Proposition \ref{prop: omegaexpand}
gives
\begin{eqnarray}\label{oWn}
\omega_{\nu} &= & \omega_{11}\wedge\omega_{12}\wedge\omega_{22}\wedge\omega_{23}\wedge\ldots\wedge\omega_{nn} \nonumber \\
 &=&(-1)^{\binom{n}2}\,\omega_{11}\wedge \omega_{22} \wedge \ldots\wedge\omega_{nn}\wedge\omega_{12} \wedge \omega_{23} \wedge\ldots
\wedge\omega_{(n-1)n}.
\end{eqnarray}

(iii). It can happen that $\omega_{\nu}$ is non-zero, even though the matrix $\Omega_{\nu}^{2n-1}$ is identically zero. 
An example is given by the type $\nu=\nu_2$ considered in Example \ref{example of types}. 
\end{ex}

\subsection{Permanents associated to \texorpdfstring{$B$}{B}}
Recall that the permanent of  an $n\times n$ matrix $A = (a_{ij})$ equals:
\[ \perm(A)  =   \sum_{\sigma\in \Sigma_n}  \, a_{1 ,\sigma(1)}a_{2 ,\sigma(2)} \ldots a_{n,\sigma(n)} \ ,\]
where $\Sigma_n$ denotes the symmetric group of order $n$.

\begin{defn}
Let  $\pp=(p_1,\ldots, p_k)$ and $\qq=(q_1,\ldots, q_k)$ be vectors  of integers $p_i,q_i \geq 0$ satisfying $\sum_{i} p_i = \sum_i q_i$.  
Let $S_{\pp}=\{ 1^{p_1}, \ldots, k^{p_k}\}$ denote the multiset  with $p_1$ one's, $p_2$ two's, $\ldots$, and $p_k$ copies of $k$. 
Define the \emph{permanent polynomial} of weights $(\pp,\qq)$ to be:  
\begin{equation}\label{Pdef}
P_{\pp, \qq}(B)=\perm  (B_{S_\pp,S_\qq}) \ . 
\end{equation}  
It is a homogeneous polynomial  in $R^0$ of degree $p_1+\ldots+p_k$ in the indeterminates  $b_{ij}$. In the case where $\pp$ (or $\qq$) has a negative component $p_i<0$ (or $q_i<0$), we shall set $P_{(\pp,\qq)}(B)=0.$
\end{defn}

\begin{ex} Referring to the examples given in Section  \ref{subsubsectMatrices}, the definition gives
\[  P_{(1,1,0),(0,1,1)} (B)= \perm (B_{12,23}) = b_{12}b_{23}+b_{13}b_{22} 
\ , \qquad   
P_{(0,2,0),(0,1,1)} (B)= \perm (B_{22,23}) =  2 b_{22}b_{23}  \ . 
\]
\end{ex}

We shall define a matrix of permanents associated to a type as follows.

\begin{defn} Let $\nu $ be a type of rank $n$. Define
$   \Phi_{\nu}(B) $ in $M_{n\times n}(R_0)$ to be the matrix whose entries for $1 \leq i , j \leq n$ are 
\begin{equation} \label{Phidef}
\left( \Phi_{\nu}(B)\right)_{ij} =    (\qq(\nu)_j+\delta_{i,j}-1)
P_{\qq(\nu)+\eee_i-\eee_j-\one,\pp(\nu)-\one}(B)\ ,
\end{equation}
where $\delta_{ij}$ is the Kronecker delta, $\eee_i=(0,\ldots, 1, \ldots 0)$ is the row vector of length $n$ with a single $1$ in the $i^{\mathrm{th}}$ slot, and $\one=(1,\ldots,1)$ is the vector  of length $n$ with all entries 1.
\end{defn}
\begin{ex} Let $\nu= (11,12,22)$ with weight vectors   $\pp(\nu)= (2,1)$ and $\qq(\nu)=(1,2)$. Then 
\[ \Phi_{\nu}(B) = \begin{pmatrix}   1 \times P_{(0,1), (1,0)}(B)  &   1 \times  P_{(1,0), (1,0)}(B)  \\
    0 \times   P_{(-1,2), (1,0)} (B)   &   2 \times  P_{(0,1), (1,0)}(B)  \\ 
\end{pmatrix} =   \begin{pmatrix}  \perm((b_{21})) &  \perm((b_{11})) \\ 
0 &  2  \perm((b_{21}))
\end{pmatrix}\ .\]
This matrix occurs  in the second term of Example \ref{ex:Bomeganis2}. 
    \end{ex}

\subsection{Statement of the identity} \label{sect: statementid}

\begin{thm}\label{thm1}
Let $B$ and $\Omega$ be as above. We have
\begin{equation}\label{eqthm1a}
(B\Omega_{\nu})^{2n-1} = \left( \det B\right) \, \Phi_{\nu}(B) \,\omega_\nu,
\end{equation}
\end{thm} 
The proof of this theorem is postponed to Section \ref{Appendix1}. Using  Proposition \ref{propBOmegaasSumoftypes}, the  theorem gives a closed formula for $(B\Omega)^{2n-1} = \sum_{\nu} (B\Omega_{\nu})^{2n-1}$. 

\begin{ex}\label{BOmegaexplicitex} Let $\nu_1,\nu_2,\nu_3$ be as in Example \ref{example of types}. Then 
\begin{eqnarray}  (B \Omega_{\nu_1} )^5  &= &    \det(B) \begin{pmatrix}   b_{21}b_{32}+b_{22}b_{31} &  
b_{11}b_{3 2}+b_{1 2}b_{3 1}   &   b_{1 1}b_{2 2}+b_{1 2}b_{2 1}   \\ 
 0 &  2  (b_{21}b_{32}+b_{22}b_{31})   & 2 b_{21}b_{2 2} \\ 
0 &  2 b_{31} b_{32}  &   2(b_{21}b_{32}+b_{22}b_{31})  \\
 \end{pmatrix} \omega_{\nu_1} \ ,  \nonumber \\ 
  (B \Omega_{\nu_2} )^5 & =  &   \det(B) \begin{pmatrix}   2 b_{21}b_{3 1}  & 2  b_{1 1} b_{3 1} & 2 b_{11} b_{21} \\
  0 & 4b_{21}b_{3 1}   & 2b_{21}^2 \\ 
  0 &  2 b_{3 1}^2 &   4 b_{21}b_{31}
  \end{pmatrix} \omega_{\nu_2}\ ,\nonumber \\
   (B \Omega_{\nu_3} )^5 & =  &   \det(B)
   \begin{pmatrix}
    2 b_{33}^2 &  0 & 4 b_{13} b_{33}  \\ 
    0 &  2 b_{33}^2  &  4 b_{23} b_{33} \\
    0 &  0 &   6 b_{33}^2 
   \end{pmatrix}  \omega_{\nu_3} \ .\nonumber
   \end{eqnarray}
  \end{ex}

\begin{cor}\label{eqtrcor}
By taking the trace of \eqref{eqthm1a}   we obtain: 
\begin{equation}\label{eqtr}
\tr  \left( \left( B\Omega_{\nu}\right)^{2n-1}\right)  =(2n-1)\det(B)  \, 
P_{\qq(\nu)-\one,\pp(\nu)-\one}(B)\,\omega_\nu .
\end{equation}
\end{cor}
\begin{proof} Since $\sum_{i=1}^n\qq(\nu)_i=2n-1$, one has   $\tr \,\Phi_{\nu}(B) = (2n-1) P_{\qq(\nu)-\one,\pp(\nu)-\one}(B)$. 
\end{proof}

\section{Antisymmetrised Permanents} \label{sect: AntiSymPerm}
The next step in the reduction from generic to symmetric matrices will require a formula for antisymmetrised permanents which we state here. Its proof is in Section \ref{Appendix2}. 

\subsection{Notation}
Let  $m>0$  and let 
$S_m = \{s_1,\ldots, s_m\}$, $T_m = \{t_1,\ldots, t_m\}$
be two sets with $m$ elements. 
We work in the polynomial ring over $\ZZ$ generated by symbols $b_{uv}$, for   $u,v \in S_m \cup T_m$, subject to the relation $b_{uv} = b_{vu}$.  
We fix an order in $S_m$ and $T_m$ and consider the generic symmetric $m \times m$ matrix
\[  B_{S_m,T_m}  =  (b_{s_it_j})_{i,j}, \] 
which is indexed by the (generic) sets $S_m$ and $T_m$.

\subsubsection{Antisymmetrisation}
Consider the involutions 
$$g_1,\ldots, g_m   \ \in \  \Sigma_{S_m \cup T_m}$$
acting on the set $S_m \cup T_m$, where $g_i$ interchanges $s_i$ and $t_i$, and leaves all other elements fixed.  These automorphisms generate a group 
$G_m \cong \left( \ZZ/2\ZZ\right)^m$
which acts on $\ZZ[b_{uv}, u,v \in S_m \cup T_m]$ by permuting indices. Consider the character
$ \chi : G_m  \rightarrow  \ZZ^{\times} = \{\pm 1\},$  
which is defined by  
$\chi( g_i) =  -1$ for each $i$. 
For any subgroup $G<G_m$ we define 
\[ \pi_{G} = \sum_{g\in G} \chi(g) g \]
to be the projector onto the $\chi$-isotypical part of a $G$-representation, i.e.,  which is anti-invariant with respect to every $g_i\in G$.  
If $G$ is generated by $g_i$, $i\in I\subseteq\{1,\ldots,m\}$,
we shall simply write $\pi_I$ for $\pi_G$. Note that 
$\pi_i = 1-g_i$, 
for  any $i\in\{1,\ldots,m\}$. 

\begin{defn} We denote the fully antisymmetrised  generic permanent by
\begin{equation} \perm B_{[S_m,T_m]}  =  \pi_{\{1,\ldots,m\}}\, \perm B_{S_m,T_m} \ .
\end{equation}
It is the image of $\perm B_{S_m,T_m}$ under the full antisymmetrisation operator $\pi_{G_m}$.
\end{defn}

Our goal is to give a formula for $\perm B_{[S_m,T_m]}$ in terms of determinants of matrices obtained from $B_{S_m,T_m}.$  Consider the following examples for $m\leq 2$.

\begin{ex}\label{N12ex} Let $m=1$. Then 
\[ B_{S_1,T_1} = \begin{pmatrix} b_{s_1t_1} \end{pmatrix}  \]
and therefore  $\perm(B_{S_1,T_1})  = b_{s_1t_1}$ which is symmetric with respect to the operator $g_1$ which interchanges $s_1$ and $t_1$.
Therefore its antisymmetrisation
$ B_{[S_1,T_1]}=\pi_1 \perm(B_{S_1,T_1})=0  $
vanishes. 

Now let $m=2$, for which
\[B_{S_2,T_2} = \begin{pmatrix} b_{s_1t_1} & b_{s_1t_2} \\ b_{s_2t_1} & b_{s_2t_2} \end{pmatrix} \qquad \hbox{ with } \qquad \perm B_{S_2,T_2}  =  b_{s_1t_1}b_{s_2t_2} + b_{s_1t_2}b_{s_2t_1}\ . \]
We find  that the antisymmetrised permanent can be rewritten as a  determinant of  a different matrix:
$$
\pi_{1,2}   \, \perm(B_{S_2,T_2}) =  2 \left(  b_{s_1t_2} b_{s_2t_1} - b_{s_1s_2} b_{t_1t_2} \right)= -  2 \det B_{s_1t_1, s_2t_2}.
$$
 Thus we obtain
 \begin{equation}\label{N2ex}
 \perm\, B_{[S_2,T_2]}=-2 \det B_{s_1t_1, s_2t_2}.
 \end{equation}
\end{ex} 

 We generalise this formula to matrices of all ranks below. It will turn out that all antisymmetrised permanents of odd rank vanish, leaving only those of even rank $m$.

 \subsubsection{Indexing sets} To state our formula, we need to define some indexing sets. 

\begin{defn}
    Let $m\geq 2$ even. For all $k\geq 1$, define   $\mathcal{D}^k_{m}$ to be the  set of (unordered) $k$-tuples
    \[     \{\{I_1,J_1\} , \ldots, \{ I_k, J_k\} \} \]
    such that: 
    \begin{enumerate}
        \item  $\{1,\ldots,m\}$ is the  disjoint union of $I_1,\ldots, I_k, J_1, \ldots J_k$. 
        \item  $I_i$ and $J_i$ are non-empty, and have the same cardinality $|I_i|=|J_i|$, for all $1\leq i \leq k.$
    \end{enumerate}
   Since $I_i$, $J_i$ are subsets of $\{1,\ldots, m\}$, they inherit an ordering on their elements. The formulae below  do not ultimately depend upon this ordering.
\end{defn}
One may construct the set $\mathcal{D}^k_{m}$ in two stages. First, partition $\{1,\ldots, m\} = A_1 \cup \ldots \cup A_k$ into a disjoint union of $k$ sets, each of which has an even number of elements  and is non-empty. Second, choose a decomposition of $A_i= I_i \cup J_i$ as a disjoint union of two sets of equal size $|I_i|= |J_i|$, for every  $1 \leq i \leq k$.

\begin{ex} \label{examplesofDk} Note that $\mathcal{D}^k_{m}$ is empty unless $k\leq\frac{m}2$. 
\begin{itemize} 
\item Let $m=2$.  Then  $\mathcal{D}^1_2$ has a unique element, given by $\{I_1,J_1\}= \{\{1\},\{2\}\}.$
\item Let  $m=4$.  Then  $\mathcal{D}^1_4$ has three elements, given by 
\[ \{I_1,J_1\} :   \quad  \{\{1,2\}, \{3,4\}\} \ , \  \{\{1,3\}, \{2,4\}\} \ , \  \{\{1,4\}, \{2,3\}\}  \]
corresponding to the three partitions of $\{1,\ldots, 4\}$ into two sets with two elements, and $\mathcal{D}^2_4$ also has three elements where $ \{ \{I_1,J_1\}  , \{I_2,J_2\} \}$  equals:
\[     \{   \{\{1\}, \{2\}\}   \  , \  \{\{3\}, \{4\}\}  \}    \quad , \quad     \{   \{\{1\}, \{3\}\}   \ , \ \{\{2\}, \{4\}\}  \}  \quad , \quad   \{   \{\{1\}, \{4\}\}   \ , \ \{\{2\}, \{3\}\}  \}\ .   \]
   \end{itemize}  
\end{ex}

If $I$ is the ordered set $\{i_1,\ldots, i_k\} $ one  may  equip the set $S_I \cup T_I$ with the ordering $s_{i_1}, t_{i_1}, \ldots, s_{i_k}, t_{i_k} $, or alternatively, as  $s_{i_1},\ldots, s_{i_k}, t_{i_1}, \ldots, t_{i_k}$. We shall do the former, but either convention, when applied consistently, leads to the same formula for $\Sigma$ as a polynomial of determinants in the following definition.

\begin{defn}\label{defn: Sigma} Let $B_{S_{m},T_{m}}$ be as above. Define the polynomial:
\begin{equation}\label{eq: Sigma}
\Sigma(B,S_{m}\cup T_{m}) = \sum_{k=1}^{m/2}   (-2)^k\,  k! \,     \sum_{  \{\{I_1,J_1\} , \ldots, \{ I_k, J_k\} \} \in \mathcal{D}^k_{m}  } 
\det(B_{S_{I_1}\cup T_{I_1}, S_{J_1}\cup T_{J_1} })   \cdots \det(B_{S_{I_k}\cup T_{I_k}, S_{J_k}\cup T_{J_k} })   \ . 
\end{equation}
The right-hand side is well-defined because $B$ is symmetric, so $\det(B_{S_I\cup T_I,S_J\cup T_J}) = \det(B_{S_J\cup T_J,S_I\cup T_I}) $.  
\end{defn}

\begin{ex}\label{symex1}
For $m=2$ we obtain
\[\Sigma(B,\{s_1,t_1,s_2,t_2\})= -2\det\left(\begin{array}{cc}b_{s_1s_2}&b_{s_1t_2}\\b_{t_1s_2}&b_{t_1t_2}\end{array}\right).\]
This is exactly the right-hand side of \eqref{N2ex}.
For $m=4$ we obtain
\begin{align*}
&\Sigma(B,\{s_1,t_1,s_2,t_2,s_3,t_3,s_4,t_4\})=-2\det\left(\begin{array}{cccc}
b_{s_1s_3}&b_{s_1t_3}&b_{s_1s_4}&b_{s_1t_4}\\
b_{t_1s_3}&b_{t_1t_3}&b_{t_1s_4}&b_{t_1t_4}\\
b_{s_2s_3}&b_{s_2t_3}&b_{s_2s_4}&b_{s_2t_4}\\
b_{t_2s_3}&b_{t_2t_3}&b_{t_2s_4}&b_{t_2t_4}\end{array}\right)\\
&\quad-2\det\left(\begin{array}{cccc}
b_{s_1s_2}&b_{s_1t_2}&b_{s_1s_4}&b_{s_1t_4}\\
b_{t_1s_2}&b_{t_1t_2}&b_{t_1s_4}&b_{t_1t_4}\\
b_{s_3s_2}&b_{s_3t_2}&b_{s_3s_4}&b_{s_3t_4}\\
b_{t_3s_2}&b_{t_3t_2}&b_{t_3s_4}&b_{t_3t_4}\end{array}\right)
-2\det\left(\begin{array}{cccc}
b_{s_1s_2}&b_{s_1t_2}&b_{s_1s_3}&b_{s_1t_3}\\
b_{t_1s_2}&b_{t_1t_2}&b_{t_1s_3}&b_{t_1t_3}\\
b_{s_4s_2}&b_{s_4t_2}&b_{s_4s_3}&b_{s_4t_3}\\
b_{t_4s_2}&b_{t_4t_2}&b_{t_4s_3}&b_{t_4t_3}\end{array}\right)\\
&\quad +8\det\left(\begin{array}{cc}b_{s_1s_2}&b_{s_1t_2}\\b_{t_1s_2}&b_{t_1t_2}\end{array}\right)\det\left(\begin{array}{cc}b_{s_3s_4}&b_{s_3t_4}\\b_{t_3s_4}&b_{t_3t_4}\end{array}\right)+8\det\left(\begin{array}{cc}b_{s_1s_3}&b_{s_1t_3}\\b_{t_1s_3}&b_{t_1t_3}\end{array}\right)\det\left(\begin{array}{cc}b_{s_2s_4}&b_{s_2t_4}\\b_{t_2s_4}&b_{t_2t_4}\end{array}\right)\\
&\quad +8\det\left(\begin{array}{cc}b_{s_1s_4}&b_{s_1t_4}\\b_{t_1s_4}&b_{t_1t_4}\end{array}\right)\det\left(\begin{array}{cc}b_{s_2s_3}&b_{s_2t_3}\\b_{t_2s_3}&b_{t_2t_3}\end{array}\right).
\end{align*}
\end{ex}

\subsection{Statement of the theorem for antisymmetrised permanents}

\begin{thm} \label{thm: PermSigmaGeneric} Let $B_{S_m,T_m}$ be as above. If $m$ is odd then $\perm B_{[S_m,T_m]}$ vanishes. If $m$ is even, then 
\begin{equation} \label{genericPermasSigma} \perm B_{[S_m,T_m]} =  \Sigma(B,S_m\cup T_m)\ . \end{equation}
\end{thm}
The proof of this theorem is postponed to Section \ref{Appendix2}. 

\subsubsection{Restriction to a specific type}  Theorem \ref{thm: PermSigmaGeneric} may be specialised to the cases when the matrices are not generic, by applying a ring homomorphism from $\ZZ[b_{uv}]$ to another ring. Typically we shall allow the indexing sets $S_{m},T_{m}$ to have repeat indices. When this happens, many determinants in the formula for $\Sigma$ vanish.  To express this efficiently, we need some more notation. 

\begin{defn} \label{defn: sE} Let $\sE\subset \{(i,j): 1\leq i\leq j\leq n\}$ be an ordered set of $m$ pairs, and let $I \subset \{1,\ldots, m\}$. Define an ordered tuple with $2|I|$ elements:
\[ \sE_I = ( \sE_{i})_{i\in I}\ ,\]
where the elements in each pair $\sE_i=(s_i,t_i)$, for $i \in I$, $s_i,t_i\in\{1,\ldots,n\}$, are taken in that order: thus, if $I=\{i_1,\ldots, i_k\}$ in increasing order, then  $\sE_I=(s_{i_1},t_{i_1},\ldots, s_{i_k},t_{i_k})$.
\end{defn}

If $I$ is a singleton, the notation $\sE_i$ is consistent with the usual notation for elements in a set $\sE$. 

\begin{ex} \label{ExamplesofEIforWheels}
    Let $n=3$, $m=2$  and let $\sE =( (1,2), (2,3)) $. Then $\mathcal{D}^1_{2}$ corresponds to
$$
\sE_{\{1\}}\ ,\ \sE_{\{2\}}     =   (1,2) \ ,\ (2,3)\ .$$

\vspace{0.05in}
    Let $n=5$, $m=4$  and let $\sE =( (1,2), (2,3), (3,4), (4,5)) $.  The three elements of $\mathcal{D}^1_{4}$ correspond to: 

\vspace{0.05in}
\begin{center}
   \begin{tabular}{ccc}
       $\sE_{\{1,2\}}$ , $\sE_{\{3,4\}}$   &  = &  $(1,2,2,3)$ , $(3,4,4,5)$  \\
         $\sE_{\{1,3\}}$ , $\sE_{\{2,4\}}$   &  = &  $(1,2,3,4)$ , $(2,3,4,5)$  \\
          $\sE_{\{1,4\}}$ , $\sE_{\{2,3\}}$   &  = &  $(1,2,4,5)$ , $(2,3,3,4)$ \\
    \end{tabular}
\end{center}
  \vspace{0.05in}
  
\noindent and the three elements of $\mathcal{D}^2_{4}$ correspond to: 

\vspace{0.05in}
\begin{center}
 \begin{tabular}{c|ccc|l}
       $\sE_{\{1\}}$ , $\sE_{\{2\}}$   &   $\sE_{\{3\}}$ , $\sE_{\{4\}}$  &  = &  $(1,2)$  , $(2,3)$ &  $(3,4)$ ,   $(4, 5)$  \\
         $\sE_{\{1\}}$ , $\sE_{\{3\}}$   &   $\sE_{\{2\}}$ , $\sE_{\{4\}}$  &   = &  $(1,2)$  , $(3,4)$ & $(2,3)$ , $(4,5)$   \\
          $\sE_{\{1\}}$ , $\sE_{\{4\}}$  &   $\sE_{\{2\}}$ , $\sE_{\{3\}}$   &  = &  $(1,2)$  , $(4,5)$  &  $(2,3)$ , $(3,4)$\ ,  \\
    \end{tabular}
\end{center}
\vspace{0.05in}
    
 \noindent   where the vertical bars are merely visual aids and have no mathematical significance. 
\end{ex}

\begin{lem} \label{lem: SigmaBspecialise} Let $B$ be an $n\times n$ symmetric matrix, and let $\sE \subset \{(i,j): 1\leq i\leq j\leq n\}$ denote a set of an even number $m$ of pairs. 
It follows that with these conventions, for any ordering on $\sE$, one has 
\begin{equation}\label{Sigmadef1}
\Sigma(B,\sE) =  \sum_{k=1}^{m/2}   (-2)^k\,  k! \,     \sum_{  \{\{I_1,J_1\} , \ldots, \{ I_k, J_k\} \} \in \mathcal{D}^k_{m}  }   \det B_{\sE_{I_1}, \sE_{J_1}} \ldots \det B_{\sE_{I_k}, \sE_{J_k}}\ .
\end{equation}
Note that the right-hand side does not depend on the choice of ordering of $\sE$. 
\end{lem}
\begin{proof}This follows from specialisation of the definition of $\Sigma(B,{S_{m}\cup T_{m}}).$
\end{proof}

\section{Reduction to the symmetric case}\label{sectsymbi}
We use   \eqref{eqthm1a} to derive a formula for $\tr(B \Upsilon)^{2n-1}$, where both $B$ and  $\Upsilon$ are   \emph{symmetric}   $n\times n$ matrices  with entries $B=(b_{ij})$, $b_{ij}=b_{ji}$,  and $\Upsilon_{ij}=\Upsilon_{ji}=\omega_{ij}$ for $1\leq i\leq j \leq n$.  Applying Corollary \ref{eqtrcor} leads to a formula for the leading terms of $\tr(B \Upsilon)^{2n-1}$ in the form of an antisymmetrised permanent.

\subsection{Symmetric Case} 

The analogue of Proposition \ref{propBOmegaasSumoftypes} gives 
\begin{equation} \label{BUpsilonSumoverTypes}
(B \Upsilon)^{2n-1} = \sum_{\nu}  (B\Upsilon_{\nu})^{2n-1}\ , 
\end{equation}
where the sum is over all types $\nu$ of rank $n$. Note that in this formula $\Upsilon_{\nu}$  
contains many more non-zero entries than $\Omega_{\nu}$ since it is symmetric. 
The terms $(B\Upsilon_{\nu})^{2n-1}$ are  typically complicated, but there are many cancellations in the case when $\nu$ contains the diagonal elements
\begin{equation}\label{eqdiag}
\diag_n  = \{ (1,1), (2,2), \ldots, (n,n) \}  \ .
\end{equation}
We shall say that a type $\nu$ of rank $n$ is \emph{upper-triangular} if $\diag_n \subset \nu$ and  furthermore 
\begin{equation} \label{nuisUT} \nu \ \subset \  \{(i,j): 1\leq i\leq j\leq n\}  \ . \end{equation}

\begin{thm} \label{antisymthm}  Let $\nu$ be a type of rank $n$ which is upper-triangular.
If $n$ is odd then
\[  \tr \left( (B \Upsilon_{\nu})^{2n-1}\right) =(2n-1) \det(B) \,  \Sigma(B,{\nu}\setminus\diag_n) \, \omega_{\nu} \ . \]
If $n$ is even then $\tr \left( (B \Upsilon_{\nu})^{2n-1}\right)$ vanishes.    
\end{thm}

\begin{ex} Let $n=3$. Consider the type $\nu=\{(1,1),(1,2),(2,2), (2,3),(3,3)\}$. Then with 
\[  B =  \begin{pmatrix} 
b_{11} & b_{12}  & b_{13}  \\
b_{12} & b_{22} & b_{23}  \\ 
 b_{13} & b_{23} & b_{33}
\end{pmatrix}  \quad  \hbox{ and } \quad  \Upsilon_{\nu} = \begin{pmatrix} 
\omega_{11} & \omega_{12}  & 0  \\
\omega_{12} & \omega_{22} & \omega_{23}  \\ 
 0 & \omega_{23} & \omega_{33}
\end{pmatrix} \]
 Equation \eqref{BUpsilonSumoverTypes} provides the formula (see Example \ref{ExamplesofEIforWheels}): 
\[ \tr (B\Upsilon_{\nu})^5 =   - 10 \, \det(B)  \det(B_{12,23})\,
 \omega_{11} \wedge \omega_{22} \wedge \omega_{33} \wedge \omega_{12} \wedge \omega_{23},\]
where $\det(B_{12,23})=b_{12} b_{23}  -  b_{13}b_{22}$.
\end{ex}
Theorem \ref{antisymthm} only gives a partial  description of $(B \Upsilon)^{2n-1}$.   It would be interesting, but is not needed for our purposes, to extend Theorem  \ref{antisymthm} to cover all cases, including when $\nu$ does not contain $\diag_n$.

\begin{ex}\label{wheelex3}
Let $n$ be odd and  let $\nu = \diag_n \cup \{12,23,\ldots,(n-1)n\}$.
By Example \ref{nex},  
\begin{equation}\label{w3eq}
\tr \left( (B \Upsilon_{\nu})^{2n-1}\right)=(-1)^{\binom{n}2}(2n-1)\det  (B)  \,   \Sigma(B,\{12,23,\ldots,(n-1)n\}) \,
\omega_{11}\wedge\ldots\wedge\omega_{nn}\wedge\omega_{12}\wedge\ldots\wedge\omega_{(n-1)n}\ .
\end{equation}
\end{ex}

\subsection{Proof of Theorem \ref{antisymthm}}
Consider the set of types $\nu$ of rank $n$ with the property:
\begin{equation} \label{nucondition} (i,j) \in \nu  \hbox{ implies that }  (j,i) \notin \nu , \hbox{  whenever  }  i\neq j \ .
\end{equation}
We say that two such types $\nu_1$, $\nu_2$ 
are equivalent, denoted by $\nu_1 \sim \nu_2$, if they are equal as sets of unordered pairs $\{i,j\}$, or in other words,  for all $1\leq i,j \leq n$:
\[  (i,j ) \hbox{ or } (j,i)  \in \nu_1       \Longleftrightarrow  (j,i) \hbox{ or } (i,j) \in \nu_2 \]
In particular, every type $\nu$  satisfying \eqref{nucondition} is equivalent to a unique type satisfying \eqref{nuisUT}.

We shall   work in the differential graded algebra $R^{\mathrm{sym}}$ which is  defined to be the quotient of $R$ modulo the ideal generated by the relations $\omega_{ij}=\omega_{ji}$.
If $\nu$ satisfies   \eqref{nucondition}, then the following   identity holds in $R^{\mathrm{sym}}$.
\[  \Upsilon_{\nu} = \sum_{\nu' \sim \nu} \Omega_{\nu'} \ , \]
and consequently,  for every $\nu$ satisfying \eqref{nuisUT}, we have: 
\[  (B\Upsilon_{\nu})^{2n-1} = \sum_{\nu' \sim \nu} (B \Omega_{\nu'})^{2n-1} \ . \]
We  deduce from Corollary \ref{eqtrcor} that 
\[ \tr\left( (B\Upsilon_{\nu})^{2n-1}\right) = (2n-1) \det(B) \sum_{\nu' \sim \nu} P_{\qq(\nu')-\one, \pp(\nu')-\one  }(B)  \, \omega_{\nu'} \ . \]

Suppose that $\nu$ contains $\diag_n$. In this case we may write $\nu= \diag_n \cup \sE$ as a disjoint union, where $\sE = \{(i_1,j_1), \ldots, (i_{n-1},j_{n-1})\}$ with  $i_k<j_k$.  Then
\[  P_{\qq(\nu')-\one, \pp(\nu')-\one  } = \perm B_{S_{n-1},T_{n-1}} \Big|_{(s_k,t_k) = \sE_k}\]
is a specialisation of the generic permanent of indexing sets $S_{n-1}=\{s_1,\ldots, s_{n-1}\}, T_{n-1}=\{t_1,\ldots, t_{n-1}\}$, on setting $s_k=i_k$, $t_k=j_k$ for $1\leq k \leq n-1$. It does not depend on the ordering of elements in $\sE$.

Lemma \ref{onulem} $(vii)$ implies  the following identity   in $R^{\sym}$:  
\begin{equation}  \label{omegasignflip} \omega_{\nu'} = -  \omega_{\nu}\end{equation}
whenever $\nu'\sim \nu$ is obtained from $\nu$ by replacing a single off-diagonal element $(i,j)$ in $\nu$ with $(j,i)$.
Let $G$ be the group of order $2^{n-1}$ generated by the elements $g_i$, for $1\leq i\leq n-1$, which act upon $S_{n-1} \cup T_{n-1}$ by interchanging $s_i$ and $t_i$ and leaving all other elements fixed.  Let $\nu= \diag_n \cup \sE$ be of the form \eqref{nuisUT}.
Fixing an order on $\sE$ we deduce that
\[ \tr \left((B\Upsilon_{\nu})^{2n-1}\right) = (2n-1) \det(B) \left( \sum_{ g \in G }   \chi(g) \, g (  \perm B_{S_{n-1},T_{n-1}})\right)\Bigg|_{(s_k,t_k)=\sE_k}  \,    \omega_{\nu} \ , \]
where the term $\chi(g)$ accounts for the sign in \eqref{omegasignflip}, since the types $\nu'$ equivalent to $\nu$ are in one-to-one correspondence with elements of $G$. 
The statement   follows from Theorem \ref{thm: PermSigmaGeneric} and Lemma \ref{lem: SigmaBspecialise}. 

\section{A formula for the invariant form}\label{sectformulaA}
In this section, we  further specialize to the case when  $B=A^{-1}$ and $\Omega=\dd A$, where $A$ is a symmetric invertible $n \times n$ matrix,
and give a closed formula for  the leading terms of $\omega_A^{2n-1}=\tr(A^{-1}\dd A)^{2n-1} $.

\subsection{Notation}
Let $M$ be an $n\times n$ matrix.  
Let $I=(i_1,i_2, \ldots, i_k)$ and $J=(j_1, j_2, \ldots, j_k)$ be $k$-tuples of distinct elements in $\{1,\ldots, n\}$.  
Let us denote by 
$$M_{I,J} : \hbox{ the submatrix spanned by rows } I , \hbox{ columns } J.$$
$$ M^{I,J} : \hbox{ the matrix obtained from $M$ by deleting  rows } I , \hbox{  and columns } J.$$
We will need a sign convention for ordered subsets of $\{1,\ldots, n\}$. 
\begin{defn}\label{defsign}
Let $I= (i_1,\ldots, i_k)$ where $1\leq i_1,\ldots, i_k \leq n$ are distinct.
Define 
$$
\sigma(I)=(-1)^{i_1+ \ldots + i_k} \sgn \pi_I\ ,
$$
where $\pi_I$ is the unique permutation of $I$ such that $\pi(i_1)<  \ldots < \pi(i_k)$ are in increasing order.  In the case when $I$ has repeat indices, we define $\sigma(I)=0$.
\end{defn}
The following lemma is a reformulation of an  identity due to Jacobi (see, e.g., Lemma 28  in \cite{PeriodsFeynman}) and, crucially, enables one to remove all occurrences of the inverse matrix in our formulae. 
\begin{lem}\label{detlem}
Let $A$ be an invertible $n\times n$ matrix and $I$, $J$ as above. 
\begin{equation}\label{deteqgen}
\det\,(A^{-1})_{I,J}=\frac{\sigma(I)\sigma(J)}{\det A}\det A^{J,I}.
\end{equation}
\end{lem}

\begin{thm}\label{invformthm}
Let $A=(a_{ij})$ be a symmetric, invertible $n\times n$ matrix.  If $n=m+1$ is odd, then
\begin{multline}\label{formula2A}
\omega_A^{2n-1}=(2n-1)\sum_{\genfrac{}{}{0pt}{}{\sE\subseteq\{(i,j): 1\leq i<j\leq n\}}{|\sE|=m}}\sum_{k=1}^{m/2} (-2)^k k! \, \frac{ Q_k(A,\sE)}{\det (A)^{k+1}}\; \omega_{\diag_n\cup\sE}(A) \\
+ \left(\text{terms of degree } <n \hbox{ in }\dd a_{11},\dd a_{22},\ldots,\dd a_{nn}\right),
\end{multline}

where $\omega_{\nu}(A)$ is defined to be $\omega_{\nu}\big|_{\Omega=\dd A}$ for any type $\nu$, and 
 $Q_k(A,\sE)\in\ZZ[a_{ij}]$ are the polynomials:
\begin{equation}\label{Qdef}
Q_k(A,\sE)= \sum_{  \{\{I_1,J_1\} , \ldots, \{ I_k, J_k\} \} \in \mathcal{D}^k_{m}  }  \prod_{i=1}^k\sigma(\sE_{I_i})\sigma(\sE_{J_i})\det A^{\sE_{I_i},\sE_{J_i}}.
\end{equation}
If  $n$ is even then all terms in  $\omega_A^{2n-1}$ are of degree $<n$ in  $\dd a_{11},\ldots,\dd a_{nn}$.
\end{thm}
\begin{proof}
After using (\ref{BUpsilonSumoverTypes}) we specialise  Theorem \ref{antisymthm} to  the case $B=A^{-1}$ and $\Omega=\dd A$.
The result follows from substituting (\ref{deteqgen}) in (\ref{Sigmadef1}) and the fact that 
 $\det A^{\sE_{J_i},\sE_{I_i}}=\det A^{\sE_{I_i},\sE_{J_i}}$ since $A$ is symmetric. 
\end{proof}

The point  of Theorem \ref{invformthm} is that $\omega_A^{2n-1}$ is expressed in terms of determinants of co-factors of $A$ without use of its inverse.
The polynomials $Q_k(A,\sE)$ may be computed  quite efficiently.

\section{The invariant form of a graph Laplacian}\label{sectformula}
Let $G$ be a finite connected graph. The invariant forms  $\omega^{2n-1}_{\Lambda_G} = \tr (\Lambda_G^{-1} \dd \Lambda_G)^{2n-1}$ do not, of course, depend on the choice of graph Laplacian $\Lambda_G$, but the formulae 
we have derived above, do. 
For this reason there is much to gain by choosing a particularly simple form for $\Lambda_G$.

To this end,  suppose that $h_G=n$ and choose a spanning tree $T$ in $G$. 
It is well-known that there exist closed cycles  $c_1,\ldots, c_n \in \ZZ^{E_G}$ whose images in   $H_1(G;\ZZ)$ are independent, with the property  that $c_i$ meets the complement of $T$ in a single edge which we denote by  $e_i$, for $i=1,\ldots, n$.  The graph Laplacian with respect to this basis of $H_1(G;\ZZ)$ has the property that the variables $x_i$ corresponding to edge $e_i$ for $i=1,\ldots,n$ occur exactly once in the matrix, with $x_i$ in the $i^{\mathrm{th}}$ row and column.

We may assume that $G$ has $2n$ edges (otherwise the  corresponding  invariant form vanishes).  
Let  $T\subset G$ be a spanning tree, and choose any edge in $T$ which we may assume is numbered $2n$. Let $\Lambda_G(T)$ denote a choice of graph Laplacian, as above, in which we set $x_{2n}=1$. Since the invariant form $\omega^{2n-1}_G$ is projective, it is uniquely determined by restriction to the affine chart $x_{2n}=1$. Since it is proportional to 
  $\dd x_{1}\wedge \ldots \wedge \dd x_{2n-1}$,  the  `lower order terms'  in Theorem \ref{invformthm} vanish.

\begin{thm}\label{Qdefthm}
Let $G$ be a finite connected graph with $n=h_G$ independent cycles and $2n$ edges. Let $T$ be a spanning tree in $G$ and $\Lambda_G(T)$ a choice of graph Laplacian as above.  If $n=m-1$ is odd, then
\begin{align}\label{formula2}
\omega_G^{2n-1}&=(2n-1)\sum_{\genfrac{}{}{0pt}{}{\sE\subseteq\{(i,j): 1\leq i<j\leq n\}}{|\sE|=m}}\sum_{k=1}^{m/2}  (-2)^k k!\, \frac{ Q_k(\Lambda_G(T),\sE)}{\Psi_G^{k+1}}\;\omega_{\diag_n \cup \sE} (\dd\Lambda_G(T)),\quad\text{where}\nonumber\\
&\quad
Q_k(\Lambda_G(T),\sE)= \sum_{  \{\{I_1,J_1\} , \ldots, \{ I_k, J_k\} \} \in \mathcal{D}^k_{m}  }  \prod_{i=1}^k\sigma(\sE_{I_i})\sigma(\sE_{J_i})\det \Lambda_G(T)^{\sE_{I_i},\sE_{J_i}}.
\end{align}
If $n$ is even then  $\omega_G^{2n-1}=0$.
\end{thm}
\begin{proof}
Substitute $A = \Lambda_G(T)$ into  Theorem \ref{invformthm}. Note that $\det\Lambda_G(T)=\Psi_G$ is the graph polynomial.
\end{proof}

 \begin{remark}
In practice it seems  beneficial to choose a tree $T$ with many leaves. 
 \end{remark}

\begin{ex}\label{wheelex5}
For the wheel with $n$ spokes $W_n$,  choose $T$ to be the spokes
 numbered from $n+1, \ldots, 2n$ in cyclic order. The spoke $n+i$ meets edges $i$ and $i+1$
(mod $n$) from the rim.
Every consecutive pair of spokes, together with a unique edge in the rim, forms a triangle, which we orient counter-clockwise and take as our  cycle basis. We obtain the matrix (\ref{introLambdaW}) for $\Lambda_{W_n}(T)$ and set $x_{2n}=1$.

We shall assume that $n$ is odd.  By inspection of this matrix,  one  verifies that there is a unique subset  $\sE \subset \{(i,j): 1\leq i<j\leq n\}$ such  that $\omega_{\diag_n\cup\sE}(\dd\Lambda_{W_n}(T))$ is non-zero, namely  
\[  \sE=\{(1,2) \ , \ (2,3) \ , \ \ldots, \ (n-1,n)\} \ . \]
Because $\sE$ consists of consecutive pairs, one checks that either  $\sigma(\sE_{I_i})=\sigma(\sE_{J_i})=(-1)^{|I_i|}$  or vanishes (in the case when $\sE_{I_i}$ has repeated elements) for all the terms occurring in \eqref{formula2}. 
By (\ref{oWn}), we have 
$$
\omega_{\diag\cup\sE}  (\dd\Lambda_{W_n}(T))=(-1)^{\binom{n}2}\dd x_1\wedge\ldots\wedge\dd x_n\wedge(-\dd x_{n+1})\wedge\ldots\wedge(-\dd x_{2n-1})=(-1)^{\binom{n}2}\dd x_1\wedge\ldots\wedge\dd x_{2n-1},
$$
where we reduced (by antisymmetry) the terms $\dd x_i+\dd x_{n+i-1}+\dd x_{n+i}$ on the diagonal of $\dd\Lambda_{W_n}$ to $\dd x_i$.
The sum over $I_i$, $J_i$ in (\ref{formula2}) can be restricted to sets $\sE_{I_i}, \sE_{J_j}$  with distinct elements which considerably reduces the number of terms in the formula.  Examples for $I_i$ and $J_i$ are in Example \ref{ExamplesofEIforWheels}. 
\end{ex}

\subsection{Example: the wheel with 5 spokes}\label{sect: 5spokes}
Let  $n=5$. As is the case for all wheel graphs, the first sum of 
\eqref{formula2} reduces to a single term $\mathcal{E}= \{(1,2),(2,3),(3,4),(4,5)\}$.   There are a priori 6 terms occurring in the formula, corresponding to the 3 terms in $\mathcal{D}^1_4$ and $\mathcal{D}^2_4$ (Example \ref{examplesofDk}). With  the above ordering on the elements of $\mathcal{E}$, we obtain the six products of determinants corresponding to the terms 
in Example \ref{ExamplesofEIforWheels}. The first and third involve repeated indices, and therefore vanish, leaving only one  term from $\mathcal{D}^1_4$:
\[ \det \left(\Lambda_{W_5}^{1234, 2345} \right) = (\Lambda_{W_5})_{5,1} =-1  \ .  \]
The three terms coming from  $\mathcal{D}^2_4$ are all equal: 
\begin{eqnarray*}
\det \left(\Lambda_{W_5}^{12,23}\right)  \det \left( \Lambda_{W_5}^{34,45} \right) 
&= &  
\begin{vmatrix}
    0 & -x_8 &  0 \\
    0 &x_4+x_8+ x_9 & -x_9 \\
    -1 & -x_9 & x_5+ x_9 +1
\end{vmatrix}
\begin{vmatrix}
    x_1 +x_6 +1  & -x_6 &  0 \\
     - x_6 &x_2+x_6+ x_7 & -x_7 \\
    -1 & 0 & 0
\end{vmatrix} \\ 
& = &  x_6x_7x_8 \nonumber 
x_9\ ,\\
\det  \left( \Lambda_{W_5}^{12,34} \right) \det  \left(\Lambda_{W_5}^{23,45} \right)&=& x_6x_7x_8
x_9\ ,\qquad
\det  \left(\Lambda_{W_5}^{12,45} \right) \det \left(\Lambda_{W_5}^{23, 34}\right)\;=\; x_6x_7x_8
x_9\ .
\end{eqnarray*}
Taking all contributions together gives
\[ \omega^9_{W_5} \big|_{x_{10}=1} = \left( 18 \, \frac{1} {\Psi^2_{W_5}} +   216 \, \frac{x_6x_7x_8x_9x_{10}}{\Psi_{W_5}^3} \right)\Bigg|_{x_{10}=1} \dd x_1 \wedge \ldots \wedge \dd x_9 \ , \]
in agreement with \cite{BrSigma}.

\subsection{Remarks on convergence and MZVs}
Each monomial $M(x)$ in $Q_k(\Lambda_G(T),\sE)$ provides the parametric representation $M(x)/\Psi^{k+1}$ of a graphical function  $f_{G(M)}^{(k)}$  in $2k+2$ dimensions \cite{schnetz2014graphical,GolzPanzerSchnetz,BorinskySchnetz}
whose underlying graph is $G$ but with  differing  edge weights.  An exponent $k-1,k-2,\ldots,0$ of $x_i$ in the homogenisation  of $M(x)$ 
corresponds to an edge-weight $1/k, 2/k,\ldots, 1$   of the edge associated to $x_i$.

\begin{remark}
It was shown in \cite{BrSigma} that $\omega_G^{2n-1}$ is convergent. In fact,  very extensive numerical evidence suggests  a much stronger statement:   for every  monomial 
$M(x)$ in $Q_k(\Lambda_G(T),\sE)$ the  graphical function $f_{G(M)}^{(k)}$  is convergent.  In particular, we expect that $Q_1(T)=0$ if the graph $G$ has a sub-divergence
in four spacetime dimensions (i.e.\ $G$ is not primitive). 
\end{remark}

\begin{thm}\label{contructthm} Let $G$ be a \emph{constructible} graph with $n$ independent cycles and $2n$ edges.
Suppose that $f_{G(M)}^{(k)}$  is indeed  convergent  for every monomial $M(x)$.
 Then
$I_G(\omega^{2n-1})$ is a multiple zeta value (MZV). 
\end{thm}
\begin{proof}
In Theorem 45 of \cite{BorinskySchnetz} it is proved (using Theorem 93 of \cite{schnetz2021generalized}) that any period of a constructible graph in even dimensions $\geq4$ is an MZV.
\end{proof}

\begin{remark}
It is expected that Theorem \ref{contructthm} holds unconditionally after regularizing 
graphical functions to deal with divergent terms, for instance in  dimensional regularization \cite{Schnetznumfunct,7loops}.  Experiments confirm that
all Laurent coefficients in the regulator $\epsilon$  for constructible graphs with subdivergences are MZVs, and   the poles in $\epsilon$ cancel after taking the sum of all terms, leaving a finite part  equal to  the canonical integral.
\end{remark}

\section{Restriction to wheel cycles and proof of main theorem} \label{sec8}
Henceforth let $G=W_n$ be the  wheel with $n$ spokes.  We use the set of  spokes, labelled $n+1,\ldots, 2n$ as a spanning tree $T$, and  label the edges of the rim  $1,\ldots,n$ as in the previous section. The graph Laplacian will be the matrix (\ref{introLambdaW}).  Denote  the product of all spokes by
\begin{equation}\label{Sdef}
S(x)=\prod_{r=1}^nx_{n+r}\ .
\end{equation}

\subsection{Calculation of the wheel integrands}
\begin{lem}\label{wheelfactorlem}
Let $I,J\subseteq\{(1,2) ,(2,3) ,\ldots,(n-1,n)\}$ be two disjoint  non-empty sets with $|I|=|J|=p$  elements.
The sets $I,J$ may be  ordered as follows: 
\[   I=\{(i_1, i_1+1),\ldots, (i_p,i_p+1)\} \ , \   J=\{(j_1, j_1+1),\ldots,(j_p, j_p+1)\}\ ,\]
where $i_1<\ldots<i_p$, $j_1<\ldots<j_p$. By interchanging $I$ and $J$ if necessary, we may assume that $i_1<j_1$.

In the notation of Example \ref{wheelex5} we have
\begin{equation}\label{detLeq}
\det\Lambda_{W_n}(T)^{I,J}=\left\{\begin{array}{cl}-\frac{S(x)}{\prod_{k=1}^px_{n+i_k}x_{n+j_k}}&\text{, if }i_1<j_1<i_2<j_2<\ldots<i_p<j_p,\\
0&\text{, otherwise.}\end{array}\right.
\end{equation}
In other words, the determinant is non-zero if and only if the sets $I,J$ are interlaced. 
  \end{lem}
\begin{proof}
Consider the matrix $L=\Lambda_{W_n}(T)$ as given in  (\ref{introLambdaW}) (a worked example is in Section \ref{sect: 5spokes}). If one deletes the two consecutive rows $i_1,i_1+1$, then  column $i_1+1$  contains a unique  non-zero entry  $-x_{n+i_1+1}$ in row $i_1+2$ (henceforth, the numbering of  rows and columns refer to the numbering in the matrix $L$, even when we consider minors of $L$ obtained by deleting rows and columns). 
If $j_1>i_1+1$, then this column occurs in $L^{I,J}$ and we may perform an  expansion of $\det(L^{I,J})$ with respect to column $i_1+1$. We obtain a factor $x_{n+i_1+1}$ multiplied by the determinant of $L^{I\cup (i_1+2),J\cup (i_1+1)}$. Now,
in column $i_1+2$ the only non-zero entry is $-x_{n+i_1+2}$ in row $i_1+3$. We continue expanding the determinant along columns in this manner until the first entry $(i_1+k+1,i_1+k)$ is not in the matrix $L^{I,J}$.
There are two ways this may happen. The first is if  $i_1+k+1 \in I$, which implies that  $i_2=i_1+k+1$. In this case,  column $i_1+k+1$ is zero in $L^{I,J}$ and $\det L^{I,J}=0$. Let us assume that this does not occur. 
The second is if $i_1+k \in J$ which, having ruled out the first situation, therefore implies that $j_1=i_1+k$ and hence  $j_1<i_2$.   The expansion of the determinant $\det(L^{I,J})$ has so far produced a  factor $x_{n+i_1+1}\ldots x_{n+j_1-1}$ (which is $1$ if $j_1=i_1+1$) obtained by multiplying elements immediately under the diagonal.   We continue to expand the determinant, this time by removing  the pair of consecutive columns $j_1,j_1+1$ and expanding along row $j_1+1$. This either gives zero (if $i_2\geq j_2$) or gives a factor  $x_{n+j_1+1}\ldots x_{n+i_2-1}$ obtained by multiplying elements of $L$ which lie above the diagonal. Continue in this manner (in both directions as there may be entries above the rows we have eliminated first if $i_1\geq2$), alternately multiplying consecutive sequences along the lines immediately above and below the diagonal. We find that the determinant vanishes if the sets $I,J$ do not interlace, and in the case that they do, we are left eventually
with the $1\times1$ matrix $(-x_{2n})$. Altogether we obtain the product of all $x_{n+r}$, for $r=1, \ldots, n$ except for  $r=i_k, j_k$, where $k=1,\ldots,p$. A careful analysis of the signs yields an overall  minus sign in (\ref{detLeq}) which  comes from $\det((-x_{2n}))$.
\end{proof}

The interlacing  sets $I,J$ corresponding to non-vanishing determinants  in the previous lemma are in one-to-one correspondence with subsets $K$ of  $\{1,2,\ldots, n-1 \} $ with $2p$ elements $\{i_1,j_1,\ldots, i_p, j_p\}$.  Conversely, to  any such set  $K = \{k_1,\ldots, k_{2p}\}$ with $k_1<k_2<\ldots< k_{2p}$ we can associate
\[ I_K = \{ (k_1, k_1+1), \ldots, (k_{2p-1}, k_{2p-1}+1)\}   \qquad \hbox{ and } \qquad   J_K = \{ (k_2, k_2+1), \ldots, (k_{2p}, k_{2p}+1)\} \ .  \]
The corresponding determinant is simply 
\begin{equation} \label{detminorforK} \det\Lambda_{W_n}(T)^{I_K,J_K} = -\frac{S(x)}{\prod_{k\in K}x_{n+k} } \ . 
\end{equation}

\begin{lem}
Let $G=W_n$ and $T$ as above, where $n=m+1$ is odd. In the notation of Theorem \ref{Qdefthm} we have
\begin{equation}\label{sumprodeq}
Q_k(\Lambda_{W_n}(T),\sE)=(-1)^k \frac{c_{m,k}}{k!} S(x)^{k-1}x_{2n}\ ,
\end{equation} 
where
\begin{equation}\label{sumprodeq2}
c_{m,k}=\sum_{\genfrac{}{}{0pt}{}{p_1,\ldots,p_k \geq1}{\sum_i 2p_i=m}} \binom{m}{2p_1, 2p_2,\ldots, 2p_k}=\sum_{\genfrac{}{}{0pt}{}{p_1,\ldots,p_k \geq1}{\sum_i 2p_i=m}}\frac{ m!}{\prod_{i=1}^k (2p_i)!} \ . 
\end{equation}
\end{lem}

\begin{proof} Recall that $\sE= \{(1,2), (2,3), \ldots, (n-1,n)\}.$
By Lemma \ref{wheelfactorlem} and the discussion which follows, the sets $\{\{I_1,J_1\}, \ldots, \{I_k, J_k\}\}$ in $\mathcal{D}^k_{m}$ which 
give rise to non-zero determinants  $\det \Lambda_{W_n}^{\sE_{I_i}, \sE_{J_i}}$
are in one-to-one correspondence with disjoint unions
\begin{equation} \label{inproofKunion} \{1,2,\ldots, n-1\} = K_1 \cup K_2 \cup \ldots \cup K_{k}\ ,\end{equation}
where each $K_i$ has an even number of elements. The set $K_i$ is given by the initial entries of all elements of  $\sE_{I_i}$ and $\sE_{J_i}$.  By \eqref{detminorforK} we have 
\[ \prod_{i=1}^k \det \Lambda_{W_n}^{\sE_{I_i}, \sE_{J_i}}    =  \prod_{i=1}^k \frac{(-S(x))}{\prod_{k_i\in K_i}x_{n+k_i}}=(-1)^k S(x)^{k-1}x_{2n}\ .\] 

By Example \ref{wheelex5},  $\sigma(I_i)=\sigma(J_i)=(-1)^{|I_i|}$.  Therefore the  calculation of
$Q_\lambda(\Lambda_{W_n}(T),\sE)$ reduces to counting the
number of ways of writing $\{1,\ldots, m\}$ as an unordered  disjoint union of $k$ sets of size $2p_1,\ldots, 2p_k$. This is simply given by $1/k!$ times a  multinomial coefficient:
\[   \frac{1}{k!}  \binom{m}{2p_1,\ldots, 2p_k} =  \frac{1}{k!} \frac{m!}{(2p_1)! \ldots (2p_k)!} \ .\]
The result follows.
\end{proof}

\begin{cor} \label{cor: Wheelintegrandintermsofccoeffs}
We deduce that  for $n=m+1$ odd, 
\[ \omega_{W_n}^{2n-1}\big|_{x_{2n}=1} = (-1)^{m/2}(2n-1) \sum_{k=1}^{m/2}2^kc_{m,k}\frac{S(x)^{k-1}}{\Psi^{k+1}_{W_n}} \, \dd x_1\wedge \ldots \wedge \dd x_{2n-1}\ . \]
In particular, $(-1)^{m/2}\omega_{W_n}^{2n-1}$ is positive and  the canonical integral $I_{W_n}$ is non-zero.
\end{cor}
\begin{proof}
Substitute (\ref{oWn}) and (\ref{sumprodeq}) into Equation (\ref{formula2}). 
\end{proof}

\subsection{The coefficients   \texorpdfstring{$c_{m,k}$}{c}}
The coefficients $c_{m,k}$ count the number of ordered  partitions of  a set  of size $m$ into $k$ subsets, each of which has a positive  even number of elements. 
\begin{lem} \label{lem: cmk}
Let $m\geq2$ be even. Then 
\begin{equation}\label{ceq1}
c_{m,k}  =    \frac{2}{2^k}   \sum_{r=1}^k (-1)^{k-r} \binom{2k}{k-r} r^{m}  \ . 
\end{equation}
\end{lem}

\begin{proof}Consider the  exponential generating series
\[ C_k(x)= \sum_{m\geq0\;\text{even}} \frac{c_{m,k}}{m!} x^m\ .\]
By definition of $c_{m,k}$ and multiplicative properties of exponential generating series, one has
\[ C_k(x) = C_1(x)^k \ .\]
Since $c_{0,1}=0$ and $c_{m,1}=1$ for $m\geq2$ we have
\[ C_1(x) = \frac{e^x + e^{-x}}{2} -1 =   \frac{1}{2} (e^{x} -1)^2 \, e^{-x}\ .\]
The following identity follows from rearranging the terms in the binomial theorem:
\[ \left(y-1\right)^{2k} y^{-k} =(-1)^{k} \binom{2k}{k} +   \sum_{r=1}^k (-1)^{k-r} \binom{2k}{k-r}(y^r+y^{-r})\ .  \]
Applying it with $y=e^x$ gives
\[ 2^k  C_k(x) =     (e^{x} -1)^{2k} \, e^{-kx} = (-1)^{k} \binom{2k}{k} +   \sum_{r=1}^k (-1)^{k-r} \binom{2k}{k-r}(e^{rx}+e^{-rx})\ .   \]
Reading off the coefficient of the term $x^m$ gives the stated formula.
\end{proof}

\begin{ex}\label{ex120}
For all even $m\geq 2$ we  have $c_{m,1}=1$ and 
\[ c_{m,2} = \frac{2^m}{2}-2\   \qquad  , \qquad  \ c_{m,3}= \frac{3^m}{4} -3 \frac{2^m}{2}+ \frac{15}{4}\  . \]
\end{ex}

\subsection{Edge-weighted Feynman integrals for  wheel graphs}
The calculations above imply that the canonical wheel integrals are linear combinations of Feynman integrals for graphs with weighted edges. 
\begin{defn} Let $G$ be a finite connected graph with a choice of weighting $\nu_e>0$, for every edge $e\in E_G$.     Define an associated Feynman integral for $\lambda\in\RR$,  whenever it converges, by
\begin{equation}\label{pareq}
P_G=\frac{\Gamma(\lambda+1)}{\prod_e\Gamma(\lambda\nu_e)}\int_{\sigma_G} \Omega_G\,\frac{\prod_ex_e^{\lambda(1-\nu_e)}}{\Psi_G^{\lambda+1}}\ ,
\end{equation}
where $\Psi_G$ is the graph polynomial of $G$, $\sigma_G$ is the positive coordinate simplex, and
$$
\Omega_G=\sum_{i=1}^{|E_G|}(-1)^ix_i\dd x_1\wedge\ldots\wedge\widehat{\dd x_i}\wedge\ldots\wedge\dd x_{|E_G|}\ .
$$
\end{defn}

This definition is equivalent to that of the weighted period of a graph in $2\lambda+2$ dimensions in \cite{BorinskySchnetz} by applying the change of variables $x_e \mapsto x_e^{-1}.$

Let $W_{n,2\lambda+2}$ denote  the edge-weighted wheel graph with $n$ spokes $W_n$, where the edges along the rim have weight $1$, and the spokes have weight $\lambda^{-1}$. 
A formula for $P_{W_m,2k+2}$ is derived in \cite{BorinskySchnetz}, using the fact that the dimension of the corresponding Feynman integrals are $2k+2$, which is even. 
\begin{thm}\label{Wthm} If $1\leq k\leq n-2$, 
the Feynman period $P_{W_n,2k+2}$ of the weighted wheel with $n$ spokes is
\begin{equation}\label{PWeq}
P_{W_{n,2k+2}}=\frac{\genfrac(){0pt}{}{2n-2}{n-1}}{(2k-1)!(k-1)!^{n-1}}\sum_{r=1}^\infty\frac{\prod_{\ell=1}^{k-1}(r^2-\ell^2)}{r^{2n-3}} \ . 
\end{equation}
\end{thm}
\begin{proof}
This is in Example 87 of \cite{BorinskySchnetz}. 
\end{proof}
 
\begin{cor}\label{Wncor}
 Let $n=m+1$ be odd. Then the canonical wheel integral is a linear combination of edge-weighted Feynman integrals,
\begin{equation}\label{IWeq}
I_{W_n}=(2n-1)\sum_{k=1}^{m/2}\frac{2^kc_{m,k}(k-1)!^n}{k!}P_{W_n,2k +2}\ .
\end{equation}
\end{cor}
\begin{proof}
This follows from Corollary \ref{cor: Wheelintegrandintermsofccoeffs} and the definition of $P_{W_n,2k+2}$, upon restricting the integrand to the affine chart $x_{2n}=1$. Note that the prefactor in front of the integral (\ref{pareq}) for $P_{W_n,2k +2}$ is $k!/(k-1)!^n$. By definition we orient the domain of integration such that $I_{W_n}\geq0$.
\end{proof}

\subsection{Completing the proof}
\begin{lem} Let us set  $p_0(x)=1$, $p_2(x)= x^2-1$, $p_4(x)=(x^2-1)(x^2-4)$, and 
\[ p_{2k}(x) =\prod^{k}_{\ell=1} (x^2-\ell^2) \ . \]
Let $m\geq2$ be even. Then 
\begin{equation} \label{psasxs} \sum_{k=1}^{m/2}  \frac{2^k}{(2k)!} c_{m,k}\, p_{2k-2}(x) = x^{m-2} \ .\end{equation}
\end{lem}
\begin{proof}
Let $m=2\mu$, $\mu\geq1$. The set of  polynomials $p_{2k}$ for $k\geq 0$ form a basis for the free $\ZZ$-module  $\ZZ[x^2]$.  Therefore there exist unique integers $t_{\mu,k}$ such that 
    \[  \sum_{k=1}^{\mu}  t_{\mu,k} p_{2k-2} = x^{2\mu-2}\ .   \]
The matrix $T=(t_{\mu,k})$ satisfies  $T (p_0,p_2,\ldots,p^{2k},\ldots )^T = (1,x^2,x^4,\ldots , x^{2k},\ldots)^T $, viz: 
    \[ T=  \begin{pmatrix} t_{1,1}  &   & &  &\\ t_{2,1} & t_{2,2} &  & &&   \\ t_{3,1} & t_{3,2} & t_{3,3}  & &&  \\   t_{4,1} & t_{4,2} & t_{4,3} & t_{4,4}  & &\\ 
    t_{5,1} &   t_{5,2} & t_{5,3} &  t_{5,4} &   t_{5,5} &\\
    t_{6,1} &   t_{6,2} &  t_{6,3} &  t_{6,4} &  t_{6,5} &  t_{6,6}  \\
    \vdots & & & & & &  \ddots  
    \end{pmatrix}  = \begin{pmatrix} 1  & & & &  &\\ 1& 1 &  & &&   \\ 1& 5 & 1  & &&  \\   1 & 21 & 14 & 1  & &\\ 
    1 &   85 & 147 &  30 &   1 &\\
    1 &   341 &  1408 &  627 &  55 &  1 \\
      \vdots & & & & & &  \ddots  
    \end{pmatrix}\ . \]
The numbers $t_{\mu,k}$ are known as the central factorial numbers with closed formula $2^k/(2k)!$ times \eqref{ceq1}
\cite{Sloane}, A036969.
\end{proof}

\begin{thm}\label{themainthm}
Let $n\geq3$ be an odd integer. Then
\begin{equation}\label{IWeq3}
I_{W_n}=n\genfrac(){0pt}{}{2n}{n}\zeta(n).
\end{equation}
\end{thm}
\begin{proof}
Let $n=m+1$ be odd.  Since $n\geq3$,  $m/2 \leq n-2$, and  Theorem \ref{Wthm} applies.  Corollary \ref{Wncor} yields
\begin{eqnarray*}
I_{W_n} & = & (2n-1)\sum_{k=1}^{m/2} \frac{2^kc_{m,k}\genfrac(){0pt}{}{2n-2}{n-1}}{(2k-1)!k}\sum_{r=1}^\infty\frac{\prod_{\ell=1}^{k-1}(r^2-\ell^2)}{r^{2n-3}}\\
&=& (2n-1)\binom{2n-2}{n-1}  \sum_{r=1}^\infty \frac{1}{r^{2n-3}}  \sum_{k=1}^{m/2}\frac{2^{k+1}c_{m,k}}{(2k)!} p_{2k-2}(r)
=n\genfrac(){0pt}{}{2n}{n}\sum_{r=1}^\infty\frac{r^{m-2}}{r^{2n-3}} 
=n\genfrac(){0pt}{}{2n}{n}\zeta(n)\ ,\end{eqnarray*}
where we have used \eqref{psasxs}.
\end{proof}

\section{An identity \texorpdfstring{for $(B \Omega)^{2n-1}$.}{}} \label{Appendix1}
Let $B$, $\Omega$ be  $n\times n$ matrices of $0$ and $1$-forms respectively, and set 
\begin{equation}\label{Xdef}
X(B,\Omega) =  (B \Omega)^{2n-1} \ . 
\end{equation}
Employing the same notations as in Section \ref{sec: BOmeganotations}, we   denote by 
\begin{equation}\label{Ydef}
Y(B,\Omega)=\sum_{\genfrac{}{}{0pt}{}{\nu\subseteq \{(i,j): 1\leq i, j \leq n\}}{|\nu|=2n-1}}\Phi_{\nu}(B)\,\omega_\nu \ , 
\end{equation}
where $\omega_{\nu}$ was defined in Definition \ref{defn: omeganu} in the case of the generic matrix $\Omega_{ij}=(\omega_{ij})$. In general it is defined to be  the corresponding  exterior product of 1-forms in the entries $\Omega_{ij}$ of $\Omega$. In this section we prove that $X(B, \Omega) = \det(B)\, Y(B,\Omega)$. 
The strategy is first to establish  equality in  the case when $B=I_n$ is  the identity matrix, and then use covariance properties of both $X(B,\Omega)$ and $Y(B,\Omega)$ with respect to $B, \Omega$ to deduce the identity for all $B$.

\subsection{Properties of \texorpdfstring{$X(B,\Omega)$}{X}  }
\subsubsection{Weights and types}
In this subsection we derive a structural result for $X(B,\Omega)$ from first principles, which will not be required for our later results, but may be of conceptual value to the reader.

From the definition (\ref{Xdef}) it is evident that for any invertible matrices $A_1,A_2\in\mathrm{GL}_n(R^0)$ we have
\begin{equation}\label{bicov}
X(A_1BA_2,A_2^{-1}\Omega A_1^{-1})=A_1X(B,\Omega)A_1^{-1}\ .
\end{equation}
Let us specialise  (\ref{bicov}) to the case when $A_1=D_1$ is diagonal with entries  $\lambda_1,\ldots, \lambda_n$,  and $A_2=D_2$ is diagonal with entries $\mu_1,\ldots, \mu_n$.  
The action $B \mapsto D_1 B D_2$ (resp. $\Omega\mapsto D_2^{-1} \Omega D_1^{-1}$)  defines an action of $\mathbb{G}_m^n \times \mathbb{G}^n_m$ on $\ZZ[b_{ij}]$ given by $b_{ij} \mapsto  \lambda_i b_{ij} \mu_j$ (resp.\  on  the differential graded algebra $R$ defined in Section \ref{sect:Forms} via $\omega_{ij} \mapsto \mu_i^{-1} \omega_{ij} \lambda^{-1}_j$.)   Decompose $X(B,\Omega)$  by type (see Definition \ref{defn: omeganu})
\begin{equation} \label{Xdecomp}
X(B,\Omega) =  \sum_{\nu}  F_{\nu}(B)\,\eta_{\nu}\ ,
\end{equation}
where $F_{\nu}(B)\in M_{n\times n}(\ZZ[b_{ij}])$, and  $\nu$  ranges over all types of rank $n$.  As in Section \ref{sect: weightsoftypes}, denote its  weights    by $w(\nu) = (\pp(\nu), \qq(\nu)) $. 
If we assign to $b_{ij}$ the multi-degree $(\eee_i,\eee_j)$, it follows from \eqref{bicov}  that  the  $(k,\ell)$th entry $(F_{\nu}(B))_{k, \ell} \in \ZZ[b_{ij}]$    necessarily 
has multi-degrees $(\qq(\nu)+\eee_k-\eee_{\ell}, \pp(\nu))$   
(see Example \ref{BOmegaexplicitex} where $\det B$ has multi-degree $(\one,\one)$).

\subsubsection{Factorisation of the determinant}
\begin{lem}\label{detBlem}
All entries of the matrix $X(B,\Omega)$ have a factor of $\det(B)$. Thus we may write: 
\begin{equation}\label{detB}
X(B,\Omega)=(\det B) X'(B,\Omega)
\end{equation}
for some  matrix $X'(B,\Omega)$ of  $(2n-1)$-forms whose  coefficients are homogeneous polynomials in   $\ZZ[b_{ij}]$ of total degree $n-1$.
\end{lem}
\begin{proof}   
Write $X(B,\Omega)= \sum_{\nu} F_{\nu}(B)\, \eta_{\nu}$ as above.  It suffices to show that 
$X(B,\Omega)$ vanishes if $\det(B)=0$, since in that case each $F_{\nu}(B) \in \ZZ[b_{ij}]$ is a polynomial which vanishes whenever $\det(B)$ vanishes.   It is therefore divisible by $\det(B)$,  using the well-known fact that $\det(B)$ is  irreducible (this follows easily from the fact that the determinant is of degree at most one in each  entry of $B$).  
 Now observe that if $\det(B)$ vanishes, then the matrix $B$, and hence $B\Omega$,  has  rank $m<n$. 
 It follows from Proposition \ref{prop: BOmega2nvanishes} that \emph{a fortiori} $(B\Omega)^{2m}=0$ and hence $X(B, \Omega)=0$.
 \end{proof}

\begin{cor} \label{cor: weightsfB}
If we write the coefficients in \eqref{Xdecomp}  in the form $F_{\nu}(B) = \det(B) f_{\nu}(B)$, then it follows that  $f_{\nu}(B) \in \ZZ[b_{ij}]$ has multi-degree $(\qq(\nu)- \one + e_k-e_{\ell}, \pp(\nu)-\one).$
Naturally, if one of these degrees is negative, then $f_{\nu}(B)$ must vanish.
\end{cor}

\subsubsection{Elementary transformations}
We consider the  behaviour of $Y(B,\Omega)$ with respect to the action  of elementary transformations $T_{ij}=I_n+E_{ij}$,
where  $(E_{ij})_{ab}=\delta_{a,i}\delta_{b,j}$ are  elementary matrices.

\begin{lem}\label{lemP} Let $\pp = (p_1,\ldots, p_n)$ and $\qq=(q_1,\ldots, q_n)$, where $p_i,q_i \geq 0$ such that $\sum_{i=1}^n p_i = \sum_{i=1}^n q_i$.
Let $i,j\in \{1,\ldots n\}$. With notations as in Section  \ref{sec: BOmeganotations} we have
\begin{equation}\label{PTij}
P_{\pp,\qq}(BT_{ij})=\sum_{k=0}^{q_j}\genfrac(){0pt}{}{q_j}{k}P_{\pp,\qq+k(\eee_i-\eee_j)}(B)\ .
\end{equation}
\end{lem}
\begin{proof}
Right multiplication $B \mapsto B T_{ij}$ adds column $i$ to column $j$ in $B$. 
It follows that $(BT_{ij})_{S_\pp, S_\qq}$ has a  copy of the  column of $B$ indexed by $i$ added to all $q_j$ columns indexed by $j$.   Since the permanent is invariant under permutations of  columns, it suffices to consider 
the permanent of a matrix with column vectors  $(\underline{a} , \ldots, \underline{a}, \underline{c}_1, \ldots, \underline{c}_r)$, where $\underline{a}$ occurs with multiplicity $q$. By multilinearity of the permanent, 
\[ \perm ( \underbrace{\underline{a} +\underline{b}, \ldots,  \underline{a}+\underline{b}}_q, \underline{c}_1, \ldots, \underline{c_r})  = \sum_{k=0}^q \binom{q}{k}   \perm ( \underbrace{\underline{a} , \ldots,  \underline{a}}_k, \underbrace{ \underline{b} ,\ldots,   \underline{b}}_{q-k}, \underline{c}_1, \ldots, \underline{c_r})\ ,\]
where the right-hand side follows from invariance under column permutations.   Equation (\ref{PTij}) follows by applying this identity  to the matrix  $(BT_{ij})_{S_\pp, S_\qq}$, where $q=q_j$, $\underline{a}$ is the $i^{\mathrm{th}}$ column of $B$,
$\underline{b}$ is the $j^{\mathrm{th}}$ column of $B$, and $\underline{c}_1,\ldots, \underline{c}_r$ denote the $n-q_j$ remaining columns of $B$.
\end{proof}

\begin{lem} \label{lem: TijOmega}  Let $\mu, \nu \subset \{(i,j): 1\leq i, j \leq n\}$ where $|\mu|= |\nu|=2n-1$. Let $i,j\in \{1,\ldots, n\}$ be distinct and $\omega_{\mu}(T_{ij} \Omega)$
be the differential form obtained from $\omega_\mu$ by replacing every form   $\omega_{ab}$, for $(a,b)\in \mu$ in its definition as an exterior product with the corresponding  entry  $(T_{ij} \Omega)_{ab}$ of the matrix $T_{ij} \Omega$.
Let $(\omega_{\mu}(T_{ij} \Omega))_{\nu}$ be the $\nu$-isotypical component of $\omega_{\mu}(T_{ij} \Omega)$.

(i).   Then  $(\omega_{\mu}(T_{ij} \Omega))_{\nu}=0$ unless for some $r\geq 1$
\begin{equation}  \label{munusubst} \mu = \nu \cup \{ (i,k_1),\ldots, (i,k_r)\} \setminus  \{(j,k_1), \ldots, (j,k_r) \} \ , \end{equation}
where $(j,k_\ell) \in \nu$ and $(i,k_\ell) \notin \nu$ for $1\leq \ell \leq r$. In other words, $\mu$ is obtained from $\nu$ by replacing a number of elements of the form $(j,k)$ with $(i,k)$. 
 In particular, $\qq(\mu) = \qq(\nu)$ and $\pp(\mu) = \pp(\nu) + r (\eee_i -\eee_j)\ .$
 
(ii).  Let $S_r(\nu)$ denote the set of types $\mu$ satisfying \eqref{munusubst}. Then 
\begin{equation}\label{eq: TijOmega}
\sum_{\mu \in S_r(\nu)}  (\omega_{\mu} (T_{ij} \Omega))_{\nu} = \binom{\pp(\nu)_j-1}{r}\, \omega_{\nu} \ .
\end{equation}
\end{lem}

\begin{proof} The matrix $T_{ij} \Omega$ is obtained from $\Omega$ by adding row $j$ to row $i$. Therefore 
$T_{ij} \Omega = \pi^* \Omega$ where 
\[  \pi^*  \omega_{ab} = \begin{cases}  \omega_{ab} &\hbox{, if } a \neq i \\ 
\omega_{ib}+ \omega_{jb}    &\hbox{, if } a = i \ . \end{cases}\]
Thus $\omega_{\mu}(T_{ij} \Omega) =  \pi^* \omega_{\mu}$, and in particular there exists an  $\varepsilon \in \{-1,1\}$ such that
by antisymmetry of the exterior product
\begin{equation}\label{eq: muTijO}
\omega_{\mu}(T_{ij}\Omega )   = \varepsilon \bigwedge_{(a,b) \in \mu, a\neq i}  \omega_{ab} \wedge \bigwedge_{(i,b) \in \mu} (\omega_{ib} + \omega_{jb})= \varepsilon \bigwedge_{(a,b) \in \mu, a\neq i}  \omega_{ab} \wedge\bigwedge_{(i,b),(j,b), \in \mu}  \omega_{ib} \wedge \bigwedge_{(i,b) \in \mu,(j,b) \notin \mu,} (\omega_{ib} + \omega_{jb}) \ . 
\end{equation}

By expanding out this expression, one sees that a $\nu$ with a non-trivial $\nu$-isotypical component differs from $\mu$
by replacing $(i,b)$ with $(j,b)$ for some $b$. This proves $(i)$.

For $(ii)$,  we may choose the order in Definition \ref{defn: omeganu}  consistently for all $\mu \in S_r(\nu)$ 
by replacing elements $(j,k)$ with $(i,k)$ and respecting their order.
Consequently, the coefficient of $\eta_{\nu}$ in $\eta_{\mu}(T_{ij}\Omega)$ does not depend on $\mu$. The identity to be proven then reduces to a statement about determinants of matrices $M^{\emptyset,c}_{\mu}$ for any $c$.  
The matrix $M_{\mu}$, for $\mu \in S_{r}(\nu)$, is related to that of $M_\nu$ by removing  a $-1$ in row $(j,k_\ell)$ and column $j$, 
and replacing it with a  $-1$ in the same row (now indexed by $(i,k_{\ell})$) in column $i$ for $\ell =1, \ldots, r$.  Thus, in effect, $r$ matrix entries move from column $j$ to column $i$. 
If we choose $c=i$ and  delete column $i$ from all these matrices $M_{\mu}$, 
then we obtain  $|S_r(\nu)|$   matrices which differ only in a single column $j$. Since the determinant is multilinear with respect to columns, we deduce that:
\[
\sum_{\mu \in S_{r}(\nu)}\det M_{\mu}^{\emptyset,i}=\det \left(\sum_{\mu \in S_{r}(\nu)} M_{\mu}^{\emptyset,i}\right) \ .
\]
We may extend the sum to the larger set $S'_r(\nu)$ consisting of $\mu$ of the form \eqref{munusubst}, without the condition that $(i,k_\ell) \notin \nu$. Such a $\mu$ has repeated pairs $(i,k_\ell)$ and defines, as in Definition \ref{defn: omeganu},  a matrix $M_{\mu}$ with at least two repeated rows, and hence $\det M^{\emptyset,i}_{\mu}=0$.  We deduce that
\[ 
\sum_{\mu\in S_r(\nu)}\det M_{\mu}^{\emptyset,i}=\det \left(\sum_{\mu \in S'_r(\nu)} M_{\mu}^{\emptyset,i}\right)=\frac{\pp(\nu)_j-r}{\pp(\nu)_j}
\binom{\pp(\nu)_j}{r}  \det M^{\emptyset,i}_\nu=\binom{\pp(\nu)_j-1}{r}  \det M^{\emptyset,i}_\nu,
\] 
 since  $S'_r(\nu)$  has $\binom{\pp(\nu)_j}{r}$ elements, and because the sum over all $\mu$ uniformly distributes $(\pp(\nu)_j-r)\genfrac(){0pt}{}{\pp(\nu)_j}r$ entries $-1$ over the $\pp(\nu)_j$ slots in  row $(j,k_{\ell})$ and column $j$. We deduce that by Definition \ref{defn: omeganu}
 \[  \sum_{\mu \in S_r(\nu)}  (\omega_{\mu} (T_{ij} \Omega))_{\nu} =   \sum_{\mu \in S_r(\nu)} \epsilon \det M^{\emptyset, i}_{\mu}\,    \eta_{\nu} =     \binom{\pp(\nu)_j-1}{r}  \epsilon \det M^{\emptyset, i}_{\nu}\,    \eta_{\nu} =  \binom{\pp(\nu)_j-1}{r}  \omega_{\nu}\  , \]
with $\epsilon=(-1)^{\binom{n}2+i-1}$.
\end{proof}

\begin{ex}\label{ex: k1TijOmega}
Consider the case $r=1$ in Lemma \ref{lem: TijOmega},
$$
\mu=\nu_k=\nu\cup(i,k)\setminus(j,k)\ .
$$
The types $\nu$ and $\nu_k$ are identical except for the pair $(j,k)$ in $\nu$ which
is replaced by $(i,k)$. From (\ref{eq: muTijO}) we see that $\omega_{\nu_k}(T_{ij}\Omega)$ is obtained from $\omega_{\nu_k}$ by replacing $\omega_{ib}$ with $\omega_{ib}+\omega_{jb}$  for all  pairs $(i,b)\in\nu_k$ with $(j,b)\notin\nu_k$. 
If we project this onto its $\nu$-isotypical component, then for $b\neq k$ we
have $(j,b)\notin \nu$, so that we can in effect set $\omega_{jb}=0$ in $\omega_{ib}+\omega_{jb}$. In other words  $\left(\omega_{\nu_k}(T_{ij}\Omega)\right)_{\nu}$ is obtained from $\omega_{\nu_k}$ by replacing the single entry $\omega_{ik}$  with $\omega_{ik}+\omega_{jk}$.
We obtain 
$$
(\omega_{\nu_k}(T_{ij}\Omega))_\nu=\left.\nu_k\right|_{\omega_{ik}\mapsto\omega_{jk}}\ .
$$
By summing over all $k$ we obtain from Equation (\ref{eq: TijOmega}),
$$
\sum_{\mu \in S_1(\nu)}  (\omega_{\mu} (T_{ij} \Omega))_{\nu} =
\sum_{k:(j,k)\in\nu,(i,k)\notin\nu}  (\omega_{\nu_k} (T_{ij} \Omega))_{\nu} =
(\pp(\nu)_j-1)\, \omega_{\nu} \ .
$$
\end{ex}
\subsection{Case  when \texorpdfstring{$B=I_n$}{B is the identity}.}\label{sect:BIn}
Let us write $F_{\nu}= F_{\nu}(I_n)$. Then \eqref{Xdecomp} takes the form:
\begin{equation}\label{Fij1}
X(I_n,\Omega_\nu)=F_\nu\, \eta_\nu\ . 
\end{equation}
\subsubsection{Hamiltonian circuits}
Since 
\[ \Omega = \sum_{1 \leq i, j \leq n}  \omega_{ij} E_{ij}\]
it follows from the identity $E_{ij} E_{k\ell} = \delta_{jk} E_{i\ell}$ that 
\begin{equation}  \label{XInucycles} X(I_n, \Omega)_{ij} = \left(\Omega^{2n-1} \right)_{ij} = \sum_{k_1,\ldots, k_{2n-2}} \omega_{i k_1} \wedge \omega_{k_1 k_2} \wedge \ldots \wedge \omega_{k_{2n-2} j}  \ ,
\end{equation}
whose $\nu$-isotypical component is
\begin{equation}  \label{XInucycles1}
X(I_n, \Omega_\nu)_{ij} = \sum_{(k_r,k_{r+1})\in\nu}
\omega_{i k_1} \wedge \omega_{k_1 k_2} \wedge \ldots \wedge \omega_{k_{2n-2} j}\ ,
\end{equation}
where $r$ runs from $0$ to $2n-2$, and $k_0=i$, $k_{2n-1}=j$.

\begin{lem} \label{lem: shapeFnu} (i). The matrix $F_{\nu}$ vanishes unless one of the following  two situations occurs:
\begin{enumerate}
\item   If $\pp(\nu)=\qq(\nu)$, in which case  $F_\nu$ is diagonal.
\item   If $\pp(\nu)-\qq(\nu)=\eee_i-\eee_j$ for $i\neq j$, then
$F_\nu$ is a multiple of the elementary matrix $E_{ij}$. 
\end{enumerate}
(ii). If there exists a $k\in\{1,\ldots,n\}$ such that $\pp(\nu)_k=0$ or $\qq(\nu)_k=0$, then $F_\nu=0$.
\end{lem}
\begin{proof}
For $(i)$ notice that the  set of second indices for each form $\omega_{\bullet \bullet}$ occurring  in (\ref{XInucycles}) match the first indices
with the exception of $i$ and $j$.

$(ii).$ Let $\pp(\nu)_k=0$. The statement is invariant under permuting labels. We hence
may assume without restriction that $k=1$. This implies that $\Omega$
has zero first row. For any integer $r\geq1$
we obtain
$$
\Omega^r=\left(\begin{array}{cc}0&0\\
(\Omega^{11})^{r-1}\wedge\underline{\alpha}&(\Omega^{11})^r\end{array}\right)\ ,
$$
where $\Omega^{11}$ is $\Omega$ with first row and column deleted and $\underline{\alpha}$ is the column vector $(\omega_{21},\ldots, \omega_{n1})^T$. 
We set $r=2n-1$ and use $(\Omega^{11})^{2n-2}=0$, which follows from
the last sentence in the proof of Lemma \ref{detBlem}, to deduce the result.

The result for $\qq(\nu)_k=0$ follows by transposition.
\end{proof}

To any type $\nu$,  consider the oriented graph $G_\nu$ which
has vertices $1,\ldots,n$ and an   edge from vertex $i$ to vertex $j$ for every pair $(i,j) \in\nu$, (the second construction  of Remark \ref{rem:typeisgraph}).
A term $\omega_{ik_1}\wedge\omega_{k_1k_2}\wedge\ldots\wedge\omega_{k_{2n-2}j}$ in  \eqref{XInucycles}
corresponds to an oriented path from
$i$ to $j$ through all $2n-1$ edges of $G_\nu$. The coefficients in the matrices $F_{\nu}$ therefore count the number of such paths between edges of $G_{\nu}$. 
Using \eqref{XInucycles}, we may write 
\begin{equation}\label{sumpath}
(F_\nu)_{ij}=\sum_{\gamma_{ij}\in\nu}\sgn(\gamma_{ij})\ ,
\end{equation}
where the sum is over all oriented $(2n-1)$-paths $\gamma_{ij}$ from $i$ to $j$ and $\sgn(\gamma_{ij}) \in \{1,-1\} $ is the sign of the ordering of the edges in this path relative to $\nu$.

Now we specialize to the case $i=j$, so that  $\gamma_{ii}$ becomes a closed
path from $i$ to $i$.
Let $\gamma$ be any closed oriented path in $G_{\nu}$, which passes through vertices $\gamma_1,\ldots, \gamma_{2n-1},\gamma_1$ in order. Since there are $2n-1$ terms in the wedge product we have
\[   \omega_{\gamma_1 \gamma_2}  \wedge   \omega_{\gamma_2 \gamma_3} \wedge \ldots \wedge \omega_{\gamma_{2n-1} \gamma_{1}} =    \omega_{\gamma_2 \gamma_3}  \wedge \ldots \wedge \omega_{\gamma_{2n-1} \gamma_{1}} \wedge \omega_{\gamma_1 \gamma_2}\ ,\]
which is therefore invariant under cyclic permutations of  $(\gamma_1, \gamma_2, \ldots, \gamma_{2n-1})$.  It follows that the sign $\sgn(\gamma_{ii})=\sgn([\gamma])$ in (\ref{sumpath}) only depends on the closed oriented loop  $[\gamma]$  (defined to be the equivalence class of oriented closed paths with different base points with respect to cyclic, but not dihedral, permutations). 
We may thus define
\[ \eta_{[\gamma]} = \bigwedge_{e\in E(\gamma)}  \omega_{e}\ ,\]
where the wedge product is over the oriented edges in a representative $\gamma$ of $[\gamma]$, taken in the order in which they appear in $\gamma$, starting with any edge.
The closed loop  $[\gamma]$ may be broken open at any of the vertices $i=1,\ldots, 2n-1$ to obtain a total of $2n-1$ paths, of which there are exactly $\pp(\nu)_i$ starting  and ending at vertex $i$. This is because $\pp(\nu)_i=|\{j: \gamma_j=i\}|$ is equal to the number of oriented edges in $G_{\nu}$ whose source is vertex $i$. 
It follows that
\begin{equation}\label{eq: HamCyc}
X(I_n,\Omega_\nu)_{ii}=\pp(\nu)_i\sum_{[\gamma]\in\nu}\eta_{[\gamma]}\qquad\text{and}\qquad 
(F_\nu)_{ii}=\pp(\nu)_i\sum_{[\gamma]\in\nu}\sgn([\gamma])\ ,
\end{equation}
where the sums are over all oriented Hamiltonian loops $[\gamma]$ in $G_{\nu}$. 
\subsubsection{Formula for $X(I_n, \Omega)$.}
Let $\pp=(p_1,\ldots, p_n)$, where all $p_1,\ldots, p_n\geq 0 $ are integers.  Define

\begin{equation}\label{cdef}
c(\pp)_i=\prod_{r:p_r+\delta_{i,r}\geq1}(p_r+\delta_{i,r}-1)!\ ,
\end{equation}
where the empty product is 1.

\begin{prop} \label{prop: Xid}
For all $1\leq i, j\leq n$, we have:
\begin{equation}\label{eq: Xid}
   X(I_n,\Omega_\nu)_{ij}= \delta_{\pp(\nu)-\eee_i,\qq(\nu)-\eee_j}c(\pp(\nu))_j \, \omega_{\nu}\ .
\end{equation}
\end{prop}

\subsection{Proof of Proposition \ref{prop: Xid}}  The proof is by induction over $n$.
If $\pp(\nu)-\qq(\nu)\neq\eee_i-\eee_j$, then
both sides of (\ref{eq: Xid}) are zero, see Lemma \ref{lem: shapeFnu} $(i)$. We hence may assume that $\pp(\nu)-\qq(\nu)=\eee_i-\eee_j$.

The case $n=1$ is trivial. The induction step from $n-1$ to $n$ has two stages.
In the first, we establish   the case $\pp(\nu)=\qq(\nu)$ (i.e.\ $i=j$) using the induction hypothesis for $n-1.$ In the second, we reduce the case  when   $i\neq j$ to the case $\pp(\nu)=\qq(\nu)$.

\subsubsection{The case $\pp(\nu)=\qq(\nu)$} \label{sect: casepnuisqnu}
Throughout this subsection assume $\pp(\nu)=\qq(\nu)$.
From Lemma \ref{lem: shapeFnu} $(i)$ we obtain that
$X(I_n,\Omega_\nu)_{ij}=0$ for $i\neq j$. We hence may
restrict ourselves to the case $i=j$.

If $\Omega_\nu$ has a zero row, then Equation (\ref{eq: Xid}) is trivial by Lemma \ref{lem: shapeFnu} $(ii)$ and Lemma
\ref{onulem} $(ii)$.
If not, then, since  $\Omega_{\nu}$ contains $2n-1$ non-zero entries, there must
exist a row $r$  containing a single non-zero
entry $\omega_{ra}$ and hence $\pp(\nu)_r=1$.   Because $\qq(\nu)_r=\pp(\nu)_r=1$, column $r$ of $\Omega_\nu$ also has a single non-zero entry $\omega_{br}$ lying in  some  row $1\leq b \leq n$.

The Hamiltonian path $(i,k_1,\ldots, k_{2n-2}, i)$ corresponding to each summand of (\ref{XInucycles1}) has the same start and end point, so
that we can express $X(I_n,\Omega_\nu)_{ii}$ as a sum over  Hamiltonian
loops $[\gamma]$, see (\ref{eq: HamCyc}).
Since $\pp(\nu)_r=\qq(\nu)_r=1$, every such loop $[\gamma]$
runs through the vertex $r$  exactly once. Since $G_{\nu}$ has a unique edge entering and leaving $r$, we may open all such  loops  $[\gamma]$ at the
vertex $r$ to obtain 
$$
X(I_n,\Omega_\nu)_{ii}=\pp(\nu)_i\,\omega_{ra}\wedge X(I_{n-1},\Omega_{\nu\setminus\{(r,a),(b,r)\}})_{ab}\wedge\omega_{br}\ ,
$$
on applying (\ref{XInucycles1}) in the situation $i=a$, $j=b$. The type $\nu\setminus\{(r,a),(b,r)\}$, is  a subset of $2n-3$ pairs of the $n-1$ element set $\{1,\ldots, r-1,r+1,\ldots, n\}$. Note that $\Omega_{\nu\setminus\{(r,a),(b,r)\}}\cong  (\Omega_{\nu}^{r,r})$.

By induction hypothesis we may apply \eqref{eq: Xid} to $X(I_{n-1},\Omega_{\nu\setminus\{(r,a),(b,r)\}})_{ab} $ to  deduce that 
\begin{equation}\label{Xii}
X(I_n,\Omega_\nu)_{ii}=\pp(\nu)_i\,
c(\pp(\nu\setminus\{(r,a),(b,r)\}))_b\,\omega_{ra}\wedge\omega_{\nu\setminus\{(r,a),(b,r)\}}\wedge\omega_{br}\ ,
\end{equation}
since $\pp(\nu\setminus\{(r,a),(b,r)\})  -  \qq(\nu\setminus\{(r,a),(b,r)\}) = \pp(\nu) -\qq(\nu) +\eee_a-\eee_b=\eee_a-\eee_b$ and so the Kronecker delta in the formula  is $1$. 
For $k\neq r$ we have $\pp(\nu\setminus\{(r,a),(b,r)\})_k=\pp(\nu)_k-\delta_{k,b}$. We obtain
$$
\pp(\nu)_i\,c(\pp(\nu\setminus\{(r,a),(b,r)\}))_b=\pp(\nu)_i\,\prod_{k\neq r:\pp(\nu)_k\geq1}(\pp(\nu)_k-1)!=c(\pp(\nu))_i\ ,
$$
where we have used $\pp(\nu)_r=1$ and $\pp(\nu)_i\geq1$ for all $i.$
We may  reorder  the exterior product of forms in   (\ref{Xii}) to place  $\omega_{br}$ on  the  far left, which multiplies by a factor $(-1)^{2n-2}=1.$
 Proposition \ref{prop: omegaexpand} implies that 
$$
\omega_{br}\wedge\omega_{ra}\wedge\omega_{\nu\setminus\{(r,a),(b,r)\}}=\omega_\nu\ .
$$
Equation (\ref{eq: Xid}) follows.

\subsubsection{The case $\pp(\nu)-\qq(\nu)=\eee_i-\eee_j$ for $i\neq j$}
Equation (\ref{XInucycles1}) leads to the decomposition
\begin{equation}  \label{XInurecursion}
X(I_n, \Omega_\nu)_{ij} = \sum_{k:(i,k)\in\nu} \omega_{ik}\wedge X_{kj}\ ,
\end{equation}
where
\begin{equation}  \label{XInucycles2}
X_{kj} = \sum_{(k_r,k_{r+1})\in\nu\setminus(i,k)}\omega_{k k_2} \wedge \ldots \wedge \omega_{k_{2n-2} j}\ ,
\end{equation}
with $1\leq r\leq 2n-2$ and $k_1=k$, $k_{2n-1}=j$. Fix a choice of  $k$ (which may be equal to $j$) and consider $X_{kj}$. We   distinguish two cases: $(j,k)\in\nu\setminus(i,k)$ or 
$(j,k)\notin\nu\setminus(i,k)$.
If $(j,k)\in\nu\setminus(i,k)$, then for each summand $\sigma$ in (\ref{XInucycles2})  there exists a unique $r=r_{\sigma}$ 
such that  $(k_r, k_{r+1})=(j,k)$.  One can only have $r_\sigma=1$, or $r_\sigma=2n-2$,  if $j=k$.
For every term   $\sigma$ in  (\ref{XInucycles2})  we may uniquely  write
\[ \sigma =   \underbrace{ \omega_{k k_2}\wedge\ldots\wedge\omega_{k_{r_\sigma-1} j}}_{\alpha_{\sigma}} \wedge \omega_{jk} \wedge \underbrace{\omega_{k k_{r_\sigma+2}} \wedge\ldots\wedge \omega_{k_{2n-2} j}}_{\beta_{\sigma}} = \alpha_{\sigma} \wedge \omega_{jk} \wedge \beta_{\sigma}\ , \]
where $\alpha_\sigma=1$ for $r_\sigma=1$ and $\beta_\sigma=1$ for $r_\sigma=2n-2$.
Since $\deg(\alpha_{\sigma})=r_\sigma-1$ and $\deg(\beta_{\sigma})=2n-2-r_\sigma$, 
\[  \sigma =  (-1)^{r_\sigma-1}  \omega_{jk}  \wedge \alpha_{\sigma}\wedge  \beta_{\sigma} =  (-1)^{r_\sigma-1}(-1)^{ r_\sigma(2n-2-r_\sigma)}    \beta_{\sigma}  \wedge\omega_{jk} \wedge \alpha_{\sigma}\ , \]
Since $ r_\sigma-1  + r_\sigma (2n-2-r_\sigma) \equiv r_\sigma(r_\sigma+1)+1 \equiv 1 \mod 2 $, the sign is always $-1$. The $(2n-2)$-form  $\beta_{\sigma}  \wedge\omega_{jk} \wedge \alpha_{\sigma}$ corresponds to a Hamiltonian path from $k$ to $j$ and hence to a unique summand $\sigma'$ in   (\ref{XInucycles2}). 
The map $\sigma\mapsto \sigma'$ which corresponds to interchanging $\beta_{\sigma'}=\alpha_{\sigma}$ and $\alpha_{\sigma'}=\beta_{\sigma}$ is involutive and necessarily a bijection between summands of (\ref{XInucycles2}). The above identity  implies that the contributions from $\sigma$ and $\sigma'$  have opposite signs in 
 (\ref{XInucycles2}) and thus  cancel.  It follows that  $X_{kj}$ vanishes in the case when $(j,k)\in\nu\setminus(i,k)$. 

Now suppose that  $(j,k)\notin\nu\setminus(i,k)$. In this case we can define a new set of pairs $\nu_k=\nu\cup (j,k)\setminus(i,k)$. It satisfies $\pp(\nu_k)=\pp(\nu)+\eee_j-\eee_i$ and $\qq(\nu_k)=\qq(\nu)$. Since $\pp(\nu)-\qq(\nu) = \eee_i-\eee_j$, we deduce that  $\pp(\nu_k)=\qq(\nu_k)$ and we may apply the results of Section \ref{sect: casepnuisqnu}  to obtain: 
\begin{equation} \label{inproof: XIOmega1}
    X(I_n,\Omega_{\nu_k})_{jj}=c(\pp(\nu_k))_j\,\omega_{\nu_k}\ .
\end{equation}
Since $(j,k) \in \nu_k$, there exists a  unique $(2n-2)$-form $\alpha_k$ not involving $\omega_{jk}$ such that 
$$
\omega_{\nu_k}=\omega_{jk}\wedge\alpha_k \ . 
$$
We may open all Hamiltonian loops (\ref{eq: HamCyc}) for $X(I_n,\Omega_{\nu_k})_{jj}$  at the edge $(j,k)$. This yields
\begin{equation}\label{inproof: XIOmega2}
    X(I_n,\Omega_{\nu_k})_{jj}=\pp(\nu_k)_j\,\omega_{jk}\wedge X_{kj}, 
\end{equation}
where $X_{kj}$ is defined in \eqref{XInucycles2}.
Since $(j,k) \in \nu_k$ we have  $\pp(\nu_k)_j>0$.
By equating  \eqref{inproof: XIOmega1}  and \eqref{inproof: XIOmega2}, we obtain 
$$
X_{kj}=\frac{c(\pp(\nu_k))_j}{\pp(\nu_k)_j}\,\alpha_k=\Big(\prod_{s:\pp(\nu_k)_s\geq1}(\pp(\nu_k)_s-1)!\Big)\,\alpha_k.
$$
This formula holds for all $k$. Therefore in Formula 
(\ref{XInurecursion}) we may apply this for every $k$ after discarding all terms $X_{kj}$ for $k$ such that $(j,k)\in \nu$, which vanish by the previous argument. This leads to the formula:
\begin{equation}  \label{XInurecursion1}
X(I_n, \Omega_\nu)_{ij} =\Big(\prod_{s:\pp(\nu_k)_s\geq1}(\pp(\nu_k)_s-1)!\Big) \sum_{k:(i,k)\in\nu,(j,k)\notin\nu} \omega_{ik}\wedge\alpha_k\ .
\end{equation}
For each $k$ in this formula, 
the expression  $\omega_{ik}\wedge\alpha_k$ is obtained by  replacing the one-form $\omega_{jk}$ in $\omega_{\nu_k}$ by $\omega_{ik}$.  
By Example \ref{ex: k1TijOmega} with $i,j$ interchanged, the sum
on the right hand side of (\ref{XInurecursion1}) equals $(\pp(\nu)_i-1)\,\omega_\nu$.

If $\pp(\nu)_i=1$, then $\qq(\nu)_i=0$ and both
sides of (\ref{eq: Xid}) vanish by Lemma \ref{lem: shapeFnu} $(ii)$ and Lemma \ref{onulem} $(ii)$. Otherwise, $\pp(\nu)_i>1$ and we conclude that 
\[
(\pp(\nu)_i-1)\prod_{s:\pp(\nu_k)_s\geq1}(\pp(\nu_k)_s-1)!=
(\pp(\nu)_i-1)\prod_{s:\pp(\nu)_s+\delta_{j,s}-\delta_{i,s}\geq1}(\pp(\nu)_s+\delta_{j,s}-\delta_{i,s}-1)!=c(\pp(\nu))_j \ . 
\] 
Equation (\ref{eq: Xid}) follows and completes the induction step. 

\subsection{Properties of \texorpdfstring{$Y(B, \Omega)$}{Y} } It follows from   \eqref{Ydef} that
\begin{equation} \label{Ynu}
Y(B,\Omega_{\nu})=\Phi_{\nu}(B)\,\omega_\nu \ . 
\end{equation}

\begin{prop} \label{prop: XYid} We have the identity $X(I_n, \Omega_{\nu}) = Y(I_n,\Omega_{\nu})$. 
\end{prop}
\begin{proof}
Let $\pp=(p_1,\ldots, p_n)$ and $\qq=(q_1,\ldots, q_n)$ be the
frequency vectors of the multisets $S_\pp$ and $S_\qq$. It follows from the definition that $P_{\pp,\qq}(I_n)$
is the number of bijections between the multisets $S_\pp$ and $S_\qq$ mapping $i\in S_\pp$ to $i\in S_\qq$ for all $i\in\{1,\ldots,n\}$.
If there exist an $r\in\{1,\ldots,n\}$, such that
$p_r\neq q_r$, then no such bijection exists and  $P_{\pp,\qq}(I_n)=0$. 
Otherwise every such bijection may be identified with a product of $n$ permutations on $p_r=q_r$ elements. We obtain
$$
P_{\pp,\qq}(I_n)= \prod_{r=1}^n   \delta_{p_r,q_r} \, q_r! =   \delta_{\pp,\qq}  \prod_{r=1}^n    \, q_r!\ .
$$
By \eqref{Phidef} we deduce that if $\pp(\nu)_r\geq1$
for all $r$, then
$$
Y(I_n,\Omega_\nu)_{ij}=\delta_{\qq(\nu)+\eee_i-\eee_j,\pp(\nu)}
(\qq(\nu)_j+\delta_{i,j}-1)\Big(\prod_{r=1}^n(\pp(\nu)_r-1)!\Big)\,\omega_\nu.
$$
Suppose that $\qq(\nu)+\eee_i-\eee_j=\pp(\nu)$. Then taking $j^{\mathrm{th}}$ components gives  $\qq(\nu)_j+\delta_{i,j}-1
=\pp(\nu)_j$. The previous expression simplifies
with (\ref{cdef}) to
$$
Y(I_n,\Omega_\nu)_{ij}=
\delta_{\pp(\nu)-\eee_i,\qq(\nu)-\eee_j}
c(\pp(\nu))_j\,\omega_\nu\ .
$$
This is equal to $X(I_n,\Omega_\nu)_{ij}$ by Proposition \ref{prop: Xid}.
If $\qq(\nu)+\eee_i-\eee_j\neq\pp(\nu)$, then $X(I_n,\Omega_\nu)=0$ by Lemma \ref{lem: shapeFnu} $(i)$.
If  there exists an $r$ such that $\pp(\nu)_r=0$ then $X(I_n,\Omega_\nu)=0$ by Lemma \ref{lem: shapeFnu} $(ii)$ and 
$Y(I_n,\Omega_\nu)=0$ by Lemma \ref{onulem} $(ii)$. In both cases the identity holds trivially. 
\end{proof}

\subsubsection{Action of elementary transformations}

\begin{prop}\label{propTij}
Let $n\geq2$,  and $i\neq j$ where $i,j \in \{1,\ldots, n\}$. Then  
\begin{equation}\label{Tij}
Y(BT_{ij},\Omega)=Y(B,T_{ij}\Omega) \ . 
\end{equation}
\end{prop}

\begin{proof} Let $\nu$ be a type of rank $n$.  By taking $\nu$-isotypical components \eqref{Tij} is equivalent to 
\begin{equation}\label{eq: phinuphimu}
\Phi_{\nu}(B T_{ij})\,  \omega_{\nu} =     \sum_{\mu}  \Phi_{\mu}(B) \, (\omega_{\mu}(T_{ij} \Omega))_{\nu}\ ,    
\end{equation}
where the sum is over all types $\mu$.
 The $(a,b)$th entry of the right-hand side is by definition
 \[ \sum_{\mu}  ( \qq(\mu)_{b} + \delta_{a,b}-1)     P_{\qq(\mu) + \eee_a -\eee_b-\one, \pp(\mu)-\one}(B)  
  \,  (\omega_{\mu}(T_{ij} \Omega))_{\nu} \ . 
 \]
 By Lemma \ref{lem: TijOmega} $(i)$ this reduces to 
 $$
 ( \qq(\nu)_{b} + \delta_{a,b}-1)  \sum_{r\geq 0 }  \sum_{\mu \in S_r(\nu)}      P_{\qq(\nu) + \eee_a -\eee_b-\one, \pp(\nu)+r(\eee_i -\eee_j)-\one}(B)  
  \,  (\omega_{\mu}(T_{ij} \Omega))_{\nu}\ .
$$
We move the sum over $\mu$ past the permanent and obtain from Lemma \ref{lem: TijOmega} $(ii)$
$$
( \qq(\nu)_{b} + \delta_{a,b}-1)  \sum_{r\geq 0 }   \binom{\pp(\nu)_j-1}{r}     P_{\qq(\nu) + \eee_a -\eee_b-\one, \pp(\nu)+r(\eee_i -\eee_j)-\one}(B)      \nonumber    
  \,  \omega_{\nu} \ .
$$
By Lemma \ref{lemP} this equals
\[( \qq(\nu)_{b} + \delta_{a,b}-1) P_{\qq(\nu) + \eee_a -\eee_b-\one, \pp(\nu)-\one}(BT_{ij})\,\omega_\nu\ ,\]
which is the $(a,b)$th entry of the left-hand side of (\ref{eq: phinuphimu}).
\end{proof}

\subsubsection{Completion of the proof}
\begin{thm} For $B, \Omega$ as above, 
\begin{equation} \label{XisYthm} X(B,\Omega) = \det(B) Y(B,\Omega) \ .  \end{equation}
\end{thm}
\begin{proof}
The case when $B=I_n$ was established in Proposition \ref{prop: XYid}.
For any elementary transformation $T_{ij}$, where $1\leq i,j \leq n$, one has $X(BT_{ij},\Omega)=X(B, T_{ij}\Omega)$  by definition.
The substitution $B\mapsto B (T_{ij})^k$, $\Omega\mapsto (T_{ij})^k\Omega$ for $k\in\ZZ$
in this equation and in (\ref{Tij}) ensures that \eqref{XisYthm} holds for all $B$ in the subgroup $\Gamma_n \leq \mathrm{SL}_n(\ZZ)$ generated by elementary column transformations (it is easy to see that
$\Gamma_n=\mathrm{SL}_n(\ZZ)$).

Lemma \ref{detBlem} and Equation (\ref{Ynu}) imply that \eqref{XisYthm} is equivalent to proving that
\begin{equation}\label{Xnu}
X'(B,\Omega_\nu) = \Phi_{\nu}(B)\,\omega_\nu
\end{equation}
holds for all $B, \nu$. For each $\nu$, the $(i,j)$th  entry of \eqref{Xnu} reduces to an  identity of polynomials in the entries of $B$ with rational coefficients, which we have already established   for all $B\in \Gamma_n$. Since   $\Gamma_n$ is Zariski-dense 
in the algebraic group $\mathrm{SL}_n$, we deduce that the identity of polynomials holds in the ring
\[  \mathcal{O}(\mathrm{SL}_n) = \QQ[b_{ij}] / (\det(B) - 1) \ ,  \]
i.e., that the entries of the matrices on both sides  of \eqref{Xnu} coincide up to powers of $ \det(B)$. 
However, since their  entries have total degree $n-1$, which is strictly less than the  degree of $\det(B)$, this proves that \eqref{Xnu} must hold
exactly (i.e., in the ring $\QQ[b_{ij}]$), which proves \eqref{XisYthm}. \end{proof}

\section{A determinantal identity for antisymmetrised permanents} \label{Appendix2}

In this section we prove Theorem \ref{thm: PermSigmaGeneric}. 

\subsection{Preliminaries}
Let notations be as in Section \ref{sect: AntiSymPerm}.

\subsubsection{Preliminary lemmas}
We use the following notation. For $\sigma \in \Sigma_m$, we write
$$b_{\sigma} = b_{s_1t_{\sigma(1)}} b_{s_2t_{\sigma(2)}} \ldots  b_{s_mt_{\sigma(m)}} \ . $$

\begin{lem} \label{lem: cancellationsunderpi}  Let $m\geq1$, $k\in\{0,1,\ldots, m\}$
and $\sigma \in \Sigma_m $ be a permutation. Then
$$
g_1g_2\cdots g_k\pi_{k+1,\ldots,m} \, b_{\sigma}=(-1)^{m-k}\pi_{k+1,\ldots,m} \, b_{\sigma^{-1} }\ . $$
\end{lem} 

\begin{proof} The case $k=m$ holds because
\[   g_1\cdots g_m b_\sigma=
b_{ s_{\sigma(1)t_1 }}b_{ s_{\sigma(2)t_2}} \cdots b_{ s_{\sigma(m)t_m}} =b_{s_1t_{\sigma^{-1}(1)}}b_{s_2t_{\sigma^{-1}(2)}} \cdots b_{s_mt_{\sigma^{-1}(m)}} \ .\]
For any element $x$ in the polynomial ring generated by $b_{uv}$ with $u,v \in S_m \cup T_m$, and 
$g \in G<G_m$,  we have 
$\pi_G (gx) = \chi(g) \pi_G(x)$. In particular, since $\chi(g_i)=-1$, we have $\pi_G(x)=- \pi_G(g_ix)$ if $g_i\in G$. 
It follows that 
\[\pi_{k+1,\ldots,m}\, b_\sigma=
(-1)^{m-k}\pi_{k+1,\ldots,m}g_{k+1}\cdots g_mb_\sigma\ .\]
Multiply on the left by $g_1\cdots g_k$ (which commutes with $\pi_{k+1,\ldots,m}$) and apply the case $k=m$ to conclude.
\end{proof}

\begin{lem}\label{prelimlem2} Let $m\geq1$. Then
$$
\perm B_{[S_m,T_m]} =\left\{\begin{array}{cl}0&,\ m\text{ odd}\\
2\, \pi_{2,\ldots,m}\,\perm B_{S_m,T_m}&,\ m\text{ even.}\end{array}\right.
$$
\end{lem} 
\begin{proof} For odd $m$ we take the sum over $\sigma\in\Sigma_m$ of the case $k=0$ in the previous lemma, yielding $\perm B_{[S_m,T_m]}=(-1)^m \perm B_{[S_m,T_m]}$.
For even $m$, the case $k=1$ of the lemma implies that 
\[ g_1\pi_{2,\ldots,m}\,\perm B_{S_m,T_m}=-\pi_{2,\ldots,m}\,\perm B_{S_m,T_m}\ .  \]
Since $\pi_{1,\ldots,m}=(1-g_1)\pi_{2,\ldots,m}$ the result follows.
\end{proof}

\subsection{A recursion for \texorpdfstring{$\perm\, B_{[S_m,T_m]}$}{the permanent}}
\begin{prop}\label{recprop1}
Let $m\geq2$ even. We define $\perm B_{\emptyset,\emptyset}=\perm B_{[\emptyset,\emptyset]}=1$. Then
\begin{equation}\label{permrec}
\perm B_{[S_m,T_m]} =-\frac12\sum_{i,j=2}^{m}\pi_{i,j} \left(\det B_{s_1t_1,s_it_j}
\left.\perm B_{[S_{m}\setminus  \{s_1,s_i\},T_{m}\setminus \{t_1,t_i\}]}\right|_{t_j=t_i}\right)  \ . 
\end{equation}
\end{prop}
\begin{proof}
We expand $\perm B_{S_m,T_m}$ with respect to row 1 followed by  column 1  yielding
$$
B_{S_m,T_m}=\sum_{j=1}^mb_{s_1t_j}\perm B_{S_m,T_m}^{1,j}\\
=b_{s_1t_1}\perm B_{S_m,T_m}^{1,1}+
\sum_{i,j=2}^mb_{s_1t_j}b_{s_it_1}\perm B_{S_m,T_m}^{1i,1j}\ ,
$$
where  $B^{I,J}$ denotes the matrix $B$ with rows indexed by $I$ and columns indexed by $J$ removed.
The antisymmetriser $\pi_1$ annihilates the  term $b_{s_1t_1}\perm B_{S_m,T_m}^{1,1}$ which is symmetric with respect to $g_1$. 
The sum on the right-hand side may in turn  be written as a sum of two terms:
\begin{equation} \label{inproof:sumtwoterms} \sum_{i=2}^m  b_{s_1t_i}b_{s_it_1}\perm B_{S_m,T_m}^{1i,1i}  + \sum_{2 \leq i\neq j \leq m }b_{s_1t_j}b_{s_it_1}\perm B_{S_m,T_m}^{1i,1j}  \ .      \end{equation}
Applying the antisymmetrisation operators to the summands in the sum on the left gives:
\begin{eqnarray}\pi_{1,\ldots,m}b_{s_1t_i}b_{s_it_1}\perm B_{S_m,T_m}^{1i,1i} & = &
\left(\pi_{1,i}b_{s_1t_i}b_{s_it_1} \right) \pi_{2,\ldots,\hat{i},\ldots,m}\perm B_{S_m\setminus \{s_1,s_i\},T_m\setminus \{t_1,t_i\}}   \nonumber \\
& = & -2 \det B_{s_1t_1,s_it_i}   \perm B_{[S_m\setminus \{s_1,s_i\},T_m\setminus \{t_1,t_i\}]}  \nonumber  \end{eqnarray}
by Example \ref{N12ex}, and 
by definition of the antisymmetrised permanent.  Now consider the diagonal terms $i=j$ in the sum in the right hand side of  (\ref{permrec}). They can be written in the form: 
\[   -\frac{1}{2} \pi^2_{i} \left(\det B_{s_1t_1,s_it_i} \perm B_{[S_{m}\setminus  \{s_1,s_i\},T_{m}\setminus \{t_1,t_i\}]}\right)  =   -2  \, \det B_{s_1t_1,s_it_i}   \perm B_{[S_{m}\setminus  \{s_1,s_i\},T_{m}\setminus \{t_1,t_i\}]}   \]
since $g_i$ acts trivially on $\det B_{s_1t_1,s_it_i}$ and hence 
$\pi_{i,i}=\pi_i^2$ multiplies it  by $4$. 
It follows that  the terms $i=j$ in the right hand side of  (\ref{permrec}) come from the left hand sum in \eqref{inproof:sumtwoterms} in the
expansion of $\perm B_{S_m,T_m}$.

There remain the cases $i\neq j$. Antisymmetrising  the terms in the right hand sum of   \eqref{inproof:sumtwoterms} gives
\[ \pi_{i,j}\left(\pi_1b_{s_1t_j}b_{s_it_1}\right)\left(\pi_{2,\ldots,\hat{i},\hat{j},\ldots,m}\perm 
B_{S_m\setminus \{s_1,s_i\},T_m\setminus \{t_1,t_j\}}   \right)\ . \] 
The second factor simplifies to  $\pi_1b_{s_1t_j}b_{s_it_1}=-\det B_{s_1t_1,s_it_j}$, 
which can be confirmed by direct calculation.
Using the invariance of the permanent with respect to column permutations, we may write 
\begin{eqnarray} \pi_{2,\ldots,\hat{i},\hat{j},\ldots,m}\perm 
B_{S_m\setminus \{s_1,s_i\},T_m\setminus \{t_1,t_j\}}  & = &   \pi_{2,\ldots,\hat{i},\hat{j},\ldots,m}\perm 
B_{S_m\setminus \{s_1,s_i\},T_m\setminus \{t_1,t_i\}} \Big|_{t_j=t_i}    \nonumber \\
& = &  \frac{1}{2}\,   B_{[S_m\setminus \{s_1,s_i\}],[T_m\setminus \{t_1,t_i\}]} \Big|_{t_j=t_i}       \nonumber 
\end{eqnarray} upon applying   Lemma \ref{prelimlem2} with respect to the index set $i,2,\ldots,\hat{i},\hat{j},\ldots,m$. 
This recovers the remaining terms $i\neq j$ in the right-hand side of \eqref{permrec} and completes the proof. 
\end{proof}

\subsection{A sum of determinants}
\begin{defn}\label{Sigmadefn}
Let $B$ be a symmetric $m\times m$ matrix and $\sE=\{(s_i,t_i),\,i=1,\ldots,m\}$ be a set of pairs  with $m$ even. For $\sE=\emptyset$ we define $\Sigma(B,\emptyset)=1$. Otherwise $\Sigma(B,\sE)=\Sigma(B,S_m\cup T_m)$ in Definition \ref{defn: Sigma}. Moreover we denote the
second sum in (\ref{eq: Sigma}) by
\begin{equation} \label{Tdef}
T_k(B, \sE) = \sum_{\genfrac{}{}{0pt}{}{\sE=\bigcup_{i=1}^kI_i\cup J_i}{|I_i|=|J_i|}}\det B_{I_1,J_1}\cdots\det B_{I_k,J_k}\ ,
\end{equation}
where the sum is over all decompositions of $\sE$ into (unordered) pairs of sets (of pairs $(s_k,t_k)$) $I_i, J_i$ of equal size. 
We consider $I_k$, $J_k$ as ordered sets in which each $s_i$ is followed by $t_i$ as in  Definition  \ref{defn: sE}.
\end{defn}

\begin{remark}\label{symrk}
The expression on the right hand side of (\ref{Tdef}) is well-defined. It does not depend on the order of the pairs in $I_i$ or $J_i$ as swapping two pairs amounts to
swapping pairs of rows or columns in $B_{I_i,J_i}$. Any such permutation is even.
It is also  invariant on interchanging  $I_i\leftrightarrow J_i$ since $B$ is symmetric.
\end{remark}

\subsection{A recursion  \texorpdfstring{for $\Sigma(B,\sE)$}{.}} 
Let $I,J$ such that $|I| = |J|$ and $1\in I$.
Let  $E=\{(s_i,t_i),i\in I\}$  and $F=\{(s_j,t_j),j\in J\}$ and consider the matrix
$B_{E,F}$, where $E$ stands for $\{s_{i_1},t_{i_1},\ldots , s_{i_k},t_{i_k}\}$ where $I=\{i_1,\ldots, i_k\}$ and likewise $F$, along the lines of 
Definition \ref{defn: sE}. Note that $\det B_{E,F}$ does not depend on the order of the pairs of elements in $E$ and $F$,
see Remark \ref{symrk}. 

\begin{lem}\label{expandlem} With the above notation we have
\begin{multline}\label{expandeq}
\det B_{E,F}=\sum_{j\in J}\det B_{s_1t_1,s_jt_j}\det B_{E\setminus(s_1,t_1),F\setminus (s_j,t_j)} \\
 -\sum_{i<j\in J}\pi_{i,j}\left( \det B_{s_1t_1,s_it_j}\det B_{E\setminus (s_1,t_1),F\cup (s_j,t_i)\setminus \{(s_i,t_i),(s_j,t_j)\}} \right),
\end{multline}
where the second sum is over ordered pairs in $J$.
\end{lem}
\begin{proof}
Recall that for any $n\times n$ matrix $A$, with $n\geq 2$, an expansion along the first two rows leads to  the following formula for its  determinant:
\begin{equation}  \label{detA2x2}
\det A = \sum_{1\leq a<b\leq n}  (-1)^{a+b+1} \det(A_{12,ab}) \det(A^{12,ab})\ ,
\end{equation}
where $A_{12,ab}$ denotes the $2\times 2$ minor of $A$ with rows $1,2$ and columns $a,b$, and $A^{12,ab}$ denotes the complementary $(n-2)\times(n-2)$ minor obtained by removing rows $1,2$ and columns $a,b$ from $A$. 

Apply this formula to $A= B_{E,F}$. The columns are indexed by elements of $F$ and are given by $\{s_j, t_j\}_{j\in J}.$  The  indices $a,b$ in \eqref{detA2x2}, which correspond to elements of $F$, may therefore be as follows:

$(i).$  $(a,b) = (s_j,t_j)$ for some index $j\in J$. Since $a$ is in odd position and $j$ is in even position, the sign $(-1)^{a+b+1}$ equals $1$ and we obtain the first term in \eqref{expandeq}. 

$(ii).$ All remaining terms may be collected in groups of  four:
\[  (a,b) =  (s_i, s_j) \ , (s_i, t_j) \ , \ (t_i, s_j) \ , (t_i, t_j) \]
for every  $i<j \in J$.  The signs $(-1)^{a+b+1}$ are, respectively,  $-1,1,1,-1$. These four terms are the orbit of $(s_i,t_j)$ under the group $\langle g_i,g_j\rangle$ and therefore the total contribution from these four terms comes from the action of $\pi_{i,j}$ on a single term. It suffices to consider, then, the term  $(s_i,t_j)$.
We may order the pairs in $J$ such that $(s_i,t_i)$ is followed by $(s_j,t_j)$. 
Deletion of the columns corresponding to $s_i$ and $t_j$  leaves a sequence of columns  in the order: $\ldots,t_i,s_j,\ldots$  which, after interchanging $t_i$ and $s_j,$ corresponds to the  term 
$\det B_{E\setminus (s_1,t_1),F\cup (s_j,t_i)\setminus \{(s_i,t_i),(s_j,t_j)\}}$. The minus sign in the second line of \eqref{expandeq} comes from the fact that  interchanging $t_i$ and $s_j$ multiplies this determinant by $-1$. 
We thus obtain the second term in \eqref{expandeq}.
\end{proof}

\begin{prop}\label{recprop2}
Let $m\geq2$ even, $B$ be a symmetric $m\times m$ matrix and $\sE=\{(s_i,t_i),\,i=1,\ldots,m\}$ be a set of ordered pairs. Then
\begin{equation}\label{sumdetrec}
\Sigma(B,\sE)=-\frac12\sum_{i,j=2}^m\pi_{i,j} \, \left( \det B_{s_1t_1,s_it_j}\Sigma(B,\sE\cup (s_j,t_i)\setminus  \{(s_1,t_1),(s_i,t_i),(s_j,t_j)\})\right)   \  ,
\end{equation}
where, in the case $i=j$,  $\sE\cup (s_j,t_i)\setminus  \{(s_1,t_1),(s_i,t_i),(s_j,t_j)\}$ simplifies to $\sE \setminus \{(s_1,t_1), (s_j,t_j)\}$.
\end{prop}
\begin{proof}
We use Formula (\ref{eq: Sigma}) for $\Sigma(B,\sE)=\Sigma(B,S_m\cup T_m)$, where the
second sum is denoted by $T_k(B,\sE)$ as in Definition \ref{Sigmadefn}.
The pair $(s_1,t_1)$ in $\sE$ lies in exactly one of the sets $I_i,J_i$ for some $i$. By Remark \ref{symrk}, we may assume that $(s_1,t_1)$ is contained in the set $I_1$ and
split the decomposition of $\sE$ into $I_1\cup J_1$ and a decomposition of $\sE'=\sE\setminus (I_1\cup J_1)$ into $k-1$ pairs of equal size.
We expand out the factor $\det B_{I_1,J_1}$ by applying  Lemma \ref{expandlem} with $E=I_1$, $F=J_1$. Then we interchange the sums over $j$ and $i<j$ in equation \eqref{expandeq} with the
sum over the possible choices for  $I_1$ and $J_1$ in $\sE$ to obtain
\begin{multline} \label{TkBexpression}
T_k(B,\sE)=\sum_{j=2}^m\det B_{s_1t_1,s_jt_j}\sum_{\genfrac{}{}{0pt}{}{I_1\cup J_1\subseteq\sE,\,\sE'=\bigcup_{i=2}^kI_i\cup J_i}{(s_1,t_1)\in I_1,\,(s_j,t_j)\in J_1}}\det B_{I_1\setminus(s_1,t_1),J_1\setminus(s_j,t_j)}X\\
-\sum_{2\leq i<j\leq m}\pi_{i,j}\Big(\det B_{s_1t_1,s_it_j}\sum_{\genfrac{}{}{0pt}{}{I_1\cup J_1\subseteq\sE,\,\sE'=\bigcup_{i=2}^kI_i\cup J_i}{(s_1,t_1)\in I_1,\,(s_i,t_i),(s_j,t_j)\in J_1}}\det B_{I_1\setminus (s_1,t_1),J_1\cup (s_j,t_i)\setminus \{(s_i,t_i),(s_j,t_j)\}}X\Big)\ ,\\
\end{multline}
where $X=\det B_{I_2,J_2}\cdots\det B_{I_k,J_k}$ and all pairs of sets $I_i$, $J_i$  satisfy $|I_i|=|J_i|$.

We first consider the sums in the  first line of \eqref{TkBexpression}, and  distinguish the two cases when $I_1=(s_1,t_1)$, and $J_1=(s_j,t_j)$ consist of  single pairs, or when $I_1,J_1$ have $\geq 2$ pairs. 
In the first case $\sE'$ is a decomposition of $\sE^{1j}:=\sE\setminus\{(s_1,t_1),(s_j,t_j)\}$ into $k-1$ pairs of equal size and we have   $\det B_{I_1\setminus(s_1,t_1),J_1\setminus(s_j,t_j)}=1$. Therefore this part of the inner sum on the first line of \eqref{TkBexpression} reduces to $T_{k-1}(B,\sE^{1j})$. These terms contribute
\[  \sum_{j=2}^m (\det B_{s_1t_1,s_jt_j}) T_{k-1}(B, \sE^{1j})  \]
to the first line.
In the second case, the choice of $I_1$ and $J_1$ may be  combined with
the partition of $\sE'$ to form a partition of $\sE^{1j}= \sE' \cup (I_1 \setminus (s_1,t_1)) \cup (J_1 \setminus (s_j,t_j))$   into $k$ pairs $I'_{\ell}, J'_{\ell}$ (with $1\leq \ell \leq k$), where $|I'_k| = |J'_k|$. 
Amongst  these $k$ pairs there are $2k$ possibilities for $J_1\setminus(s_j,t_j)$, namely one of $\{I'_1, J'_1, \ldots ,I'_k ,J'_k\}$.
This choice uniquely determines $I_1$ and $J_1$ (since if $J_1 =J'_\ell\cup (s_j,t_j) $ then $I_1 =I'_\ell\cup (s_1,t_1)$, or similar with $I'_{\ell}, J'_{\ell}$ interchanged).  In this case, therefore, the contribution to the first line of \eqref{TkBexpression} is:
\[  \sum_{j=2}^m (\det B_{s_1t_1,s_jt_j})  2kT_k(B, \sE^{1j})\ .  \]

Now consider the second line of \eqref{TkBexpression}. The  terms in the inner sum are in bijection with a decomposition of 
$\sE^{1ij}:=\sE\cup (s_j,t_i)\setminus \{(s_1,t_1),(s_i,t_i),(s_j,t_j)\}$ into $k$ pairs of equal size. The term which comes from 
the set $J_1$ is uniquely determined    since it is the
set which contains the pair $(s_j,t_i)$ (in particular, it cannot be empty).
The second sum is hence equivalent to an unordered partition of $\sE^{1ij}$ into $k$ pairs
of equal size. The inner sum therefore reduces to  $T_k(B,\sE^{1ij})$ and the second
line of \eqref{TkBexpression} is
\[  \sum_{2\le i<j\leq m} (\det B_{s_1t_1,s_it_j})  T_k(B, \sE^{1ij})\ .  \]

Combining the above yields the following  expression for  $\Sigma(B,\sE)=\sum_k (-2)^k k!T_k(B,\sE)   $  (by  \eqref{eq: Sigma}):
$$
\sum_{k=1}^{m/2}(-2)^kk!\Bigg(\sum_{j=2}^m\det B_{s_1t_1,s_jt_j}(T_{k-1}(B,\sE^{1j})+2kT_k(B,\sE^{1j}))-\sum_{2\leq i<j\leq m}\pi_{i,j}(\det B_{s_1t_1,s_it_j}T_k(B,\sE^{1ij}))\Bigg)\ ,
$$
where the term $T_0(B,\sE^{1j})$ vanishes, as do $T_{m/2}(B,\sE^{1 j})$ and $T_{m/2}(B,\sE^{1 i j})$,  because $\sE^{1 j}$ and $\sE^{1 i j}$ only have a total of  $m-2$ pairs.

The term in brackets  consists of two sums. In the first sum we shift $k\mapsto k+1$ in the first summand and use $(-2)^{k+1}(k+1)!+2k(-1)^kk!=-2(-2)^kk!$ to obtain
$$
-2\sum_{k=1}^{m/2-1}(-2)^kk!\sum_{j=2}^m\det B_{s_1t_1,s_jt_j}T_k(B,\sE^{1j})\ .
$$
We interchange the sums and use $\pi_{j,j}\det B_{s_1t_1,s_jt_j}=4\det B_{s_1t_1,s_jt_j}$
to see that the first term reproduces the summands corresponding to the case $i=j$   in \eqref{sumdetrec}.

In the second term the summands are invariant under the action of $g_ig_j$ since $\chi(g_ig_j)=1$. Because $B$ is symmetric, the action of $g_ig_j$ is equivalent to swapping
$i\leftrightarrow j$. We may hence extend the sum over $i<j$ to all $i\neq j$ if we introduce
a factor $1/2$. This provides
the summands with $i\neq j$ in (\ref{sumdetrec}) and the result follows.
\end{proof}

\subsection{Proof of Theorem \ref{thm: PermSigmaGeneric}}
If $m$ is even, we prove the theorem by induction over $m$ in steps of two.
For $m=0$ both sides of \eqref{genericPermasSigma} are $1$ (for $m=2$ see examples \ref{N12ex} and
\ref{symex1}).
For $m\geq2$ we obtain from propositions \ref{recprop1} and \ref{recprop2}
that both sides obey the same recursion because $\sE\cup(s_j,t_i)\setminus\{(s_1,t_1),(s_i,t_i),(s_j,t_j)\})=\sE\setminus \{(s_1,t_1),(s_i,t_i)\}|_{t_j=t_i}$. The case of odd  $m$  was proved in Lemma \ref{prelimlem2}.

\section{Appendix: A matrix identity of Amitsur-Levitzki type}

\subsection{The Amitsur-Levitzki theorem} Let $R$ be a commutative ring and let $A_1,\ldots, A_{2n} \in M_{n\times n}(R)$ be a set of $2 n$ square matrices with $n$ rows and columns.  A famous theorem of Amitsur-Levitzki states that the complete antisymmetrisation of
 the product of $2n$ matrices vanishes: 
\begin{equation}
    \sum_{\pi \in \Sigma_{2n}} \sgn(\pi) A_{\pi(1)}\cdots A_{\pi(2n)} = 0 \ . 
\end{equation}
It was observed by Rosset  \cite{Rosset76, Procesi}  that this formula is equivalent  to the identity $\Omega^{2n}=0$, where $\Omega$ is a certain  $n\times n$ matrix whose entries are one-forms (see below). 
A by-product of the results of this paper is an explicit formula for the antisymmetrisation of $2n-1$ such matrices. 

\subsection{Antisymmetrisation of \texorpdfstring{$k$}{k} matrices} More generally, given $k$ matrices $A_i\in M_{n\times n}(R)$ as above, it is interesting to ask for a formula for the antisymmetrisation 
\[  [A_1,\ldots,A_k] =  \sum_{\pi \in \Sigma_k} \sgn(\pi) A_{\pi(1)}\cdots A_{\pi(k)} \ .  \]
The case $k=2$ is simply the commutator $[A_1,A_2]= A_1A_2-A_2A_1$.

Following Definition \ref{defn: omeganu} we choose any ordering on
ordered pairs of elements  in $\{1,\ldots,n\}$ and define types $\nu=((i_1,j_1),\ldots,(i_k,j_k))$ to be sets of $k$ distinct such pairs in the chosen order.
Let $\Omega \in M_{n\times n} (R^1)$ be an $n\times n$ matrix whose entries are $1$-forms in a differential graded algebra $R^{\bullet}= \bigoplus_{n\geq 0} R^n$ where $R^0=R$. Its  $k$-fold product  $\Omega^k\in M_{n\times n}(R^k)$  is a matrix of $k$-forms and may be expanded by type:
\begin{equation} \label{Omegakexpansion} \Omega^k=\sum_{\nu}F_\nu\eta_\nu\ , 
\end{equation}
as in  (\ref{Xdecomp}),  where $F_\nu\in M_{n\times n}(R)$ and $\eta_\nu=\Omega_{i_1j_1}\wedge\ldots\wedge\Omega_{i_kj_k}$
if $\nu=((i_1,j_1),\ldots,(i_k,j_k))$ 
(see Definition \ref{defn: omeganu} $(ii)$ for the case $k=2n-1$).
\begin{prop}\label{prop: AMk} In the above notation we have
\begin{equation}\label{eq: AMk}
[A_1,\ldots,A_k]=\sum_{\nu}F_\nu\det \Big((A_r)_{ij}\Big)_{r,ij\in\nu}\ ,
\end{equation}
where $((A_r)_{ij})_{r,ij\in\nu}$ is the $k\times k$ matrix with entries $(A_r)_{ij}$
in row  $r=1,\ldots,k$ and column $(i,j)$, where $(i,j)$ ranges over the $k$ elements in $\nu$ in order. 
\end{prop}
\begin{proof}
Following Rosset, consider the graded exterior algebra $R^{\bullet}=\bigwedge \bigoplus^k_{i=1} R e_i $ 
where $e_1,\ldots, e_k$ are generators in degree $1$. We have $R^0=R$ and $R^k =  R \, e_1\wedge \ldots \wedge e_k.$  Consider the matrix 
\[ \Omega = \sum_{r=1}^k  A_r e_r  \quad \in \quad M_{n\times n} \left( R^1\right)  \ . \]
Then by antisymmetry, we have
\[ \Omega^k=[A_1,\ldots, A_k]  \, e_1\wedge\ldots \wedge e_k\ . \]
Now the $(i,j)$th entry of $\Omega$ may be written 
\[\Omega_{ij}= (A_1)_{ij}e_1 + \ldots +(A_k)_{ij}e_k\ .\]
Substituting into $\eta_\nu$ for $\nu=((i_1,j_1),\ldots,(i_k,j_k))$ yields
\[ \eta_\nu(e_1,\ldots,e_k)=\sum_{r_1,\ldots,r_k=1}^k
(A_{r_1})_{i_1j_1}e_{r_1}\wedge\ldots\wedge(A_{r_k})_{i_kj_k}e_{r_k}\ .\]
By antisymmetry again,  this expression reduces to a determinant and  \eqref{Omegakexpansion} yields
\[ \Omega^k=\sum_{\nu}F_\nu\det \Big((A_r)_{ij}\Big)_{r=1,\ldots,k,ij\in\nu}e_1\wedge\ldots\wedge e_k\ .\]
Comparing the coefficients of $e_1\wedge\ldots\wedge e_k$ gives the result.
\end{proof}

It is an interesting question to find formulae for the coefficients $F_{\nu}$ in \eqref{Omegakexpansion}, and hence for $[A_1,\ldots,A_k]$. The theorem of Amitsur-Levitzki is equivalent to the statement that all $F_{\nu}$ vanish when $k=2n$. 
The results obtained in this paper give an explicit formula in the case $k=2n-1.$

\subsection{Antisymmetrisation of \texorpdfstring{$2n-1$}{2n-1} matrices of rank \texorpdfstring{$n$}{n}} 
Recall that the weight $w(\nu)=(\pp(\nu),\qq(\nu))$ of $\nu$ is defined
by (see Definition \ref{defn: weight})
$$
\pp(\nu)=|\{i:(i,j)\in\nu\}|\ ,\qquad\qq(\nu)=|\{j:(i,j)\in\nu\}|\ .
$$
The pairs in $\nu$ have entries from $1$ to $n$,  where $n$ is the rank of $\nu$.
\begin{thm}
The $(i,j)$th component of the antisymmetrization of the $2n-1$ generic matrices $A_1,$ $\ldots,$ $A_{2n-1}\in M_{n\times n}(R)$
is given by
\begin{equation}
[A_1,\ldots,A_{2n-1}]_{ij}=(-1)^{\binom{n}{2}+s-1}\sum_{\nu:\pp(\nu)-\qq(\nu)=\eee_i-\eee_j}\det(M_{\nu}^{\emptyset,s})\,\Big(\prod_{r=1}^{2n-1}(\pp(\nu)_r+\delta_{j,r}-1)!\Big)\det \Big((A_r)_{ij}\Big)_{r,ij\in\nu}\ ,
\end{equation}
where $s$ is any number in $\{1,\ldots,2n\}$ (the formula does not depend on which one), the sum is over types $\nu$ of rank $n$, and the factorial of a negative number is defined to be $1$. 
The matrix $M_\nu$ and the notation  $M_{\nu}^{\emptyset,s}$   are defined in Definition \ref{defn: omeganu} $(i)$ and $(iii)$. 
\end{thm}

\begin{proof}
Use propositions \ref{propBOmegaasSumoftypes} (with $B=I_n$) and \ref{prop: Xid}, and substitute
(\ref{omeganudefnANDepsilonk}) into the formula for $\Omega^{2n-1}$. The result follows from Proposition \ref{prop: AMk}.
\end{proof}

\begin{ex}  Let $A,B,C$ be three  $2 \times 2$ matrices.
The antisymmetrisation
\[ [A,B,C]=   ABC- BAC +BCA - CBA+CAB- ACB\]
is given by determinants of the following $3\times 3$ matrices (see Example \ref{ex:Bomeganis2}
with $B=I_2$):
\[ X_1= \begin{pmatrix}   a_{11} & a_{12} & a_{21} \\
b_{11} & b_{12} & b_{21} \\
c_{11} & c_{12} & c_{21} \\
\end{pmatrix} 
\ , \ 
X_2  = \begin{pmatrix}   a_{11} & a_{12} & a_{22} \\
b_{11} & b_{12} & b_{22} \\
c_{11} & c_{12} & c_{22} \\
\end{pmatrix}  
\ , \ 
X_3  = \begin{pmatrix}   a_{11} & a_{21} & a_{22} \\
b_{11} & b_{21} & b_{22} \\
c_{11} & c_{21} & c_{22} \\
\end{pmatrix}  
\ , \ 
X_4  = \begin{pmatrix}   a_{12} & a_{21} & a_{22} \\
b_{12} & b_{21} & b_{22} \\
c_{12} & c_{21} & c_{22} \\
\end{pmatrix}
\] 
as
\[ 
[A,B,C] = \begin{pmatrix}
2 \det X_1 - \det X_4 &  \det X_2 \\ 
-\det X_3 &  \det X_1 - 2 \det X_4
\end{pmatrix}\ .
\]   
\end{ex} 

It would be interesting to see to what extent our methods in Section \ref{Appendix1} can be used to
derive results for $\Omega^k$ in the case $k<2n-1$.

\renewcommand\refname{References}

\bibliographystyle{alpha}

\bibliography{biblio}

\newcommand{\etalchar}[1]{$^{#1}$}
\begin{thebibliography}{SEVKM19}

\bibitem[AB12]{AlexeevBrunyate}
Valery Alexeev and Adrian Brunyate.
\newblock Extending the {T}orelli map to toroidal compactifications of {S}iegel
  space.
\newblock {\em Invent. Math.}, 188(1):175--196, 2012.

\bibitem[Ash24]{AshSharblies}
Avner Ash.
\newblock {On the cohomology of $\mathrm{SL}_n(\mathbb{Z})$}.
\newblock {\em Preprint: arXiv:2402.08840}, Jan 2024.

\bibitem[Bak11]{Baker}
O.~Baker.
\newblock The {J}acobian map on {O}uter space.
\newblock {\em Ph.D Thesis, Cornell University. Available at
  \url{https://ecommons.cornell.edu/handle/1813/30769}}, 2011.

\bibitem[BB03]{BelkaleBrosnan}
Prakash Belkale and Patrick Brosnan.
\newblock Matroids, motives, and a conjecture of {K}ontsevich.
\newblock {\em Duke Math. J.}, 116(1):147--188, 2003.

\bibitem[BBC{\etalchar{+}}20]{TopWeightAg}
M.~Brandt, J.~Bruce, M.~Chan, M.~Melo, G.~Moreland, and C.~Wolfe.
\newblock On the top-weight rational cohomology of {$\mathcal{A}_g$}.
\newblock {\em Preprint: arXiv:2012.02892}, 2020.

\bibitem[BD21]{BrownDupont}
Francis Brown and Cl\'{e}ment Dupont.
\newblock Single-valued integration and superstring amplitudes in genus zero.
\newblock {\em Comm. Math. Phys.}, 382(2):815--874, 2021.

\bibitem[BEK06]{BEK}
Spencer Bloch, H\'{e}l\`ene Esnault, and Dirk Kreimer.
\newblock On motives associated to graph polynomials.
\newblock {\em Comm. Math. Phys.}, 267(1):181--225, 2006.

\bibitem[BL97]{BismutLott}
Jean-Michel Bismut and John Lott.
\newblock {Torus bundles and the group cohomology of
  $\mathrm{GL}(N,\mathbb{Z})$}.
\newblock {\em Journal of Differential Geometry}, 47(2):196 -- 236, 1997.

\bibitem[BMV11]{BMV}
Silvia Brannetti, Margarida Melo, and Filippo Viviani.
\newblock On the tropical {T}orelli map.
\newblock {\em Adv. Math.}, 226(3):2546--2586, 2011.

\bibitem[BNM01]{GraphComplexComputations}
D.~Bar-Natan and B.~McKay.
\newblock Graph cohomology: an overview and some computations.
\newblock {\em available at
  \url{http://www.math.toronto.edu/~drorbn/papers/GCOC/GCOC.ps}}, 2001.

\bibitem[Bor74]{Borel}
Armand Borel.
\newblock Stable real cohomology of arithmetic groups.
\newblock {\em Ann. Sci. \'{E}cole Norm. Sup. (4)}, 7:235--272 (1975), 1974.

\bibitem[Bor23]{BorinskyTMCQ}
Michael Borinsky.
\newblock Tropical {M}onte {C}arlo quadrature for {F}eynman integrals.
\newblock {\em Annales de l'Institut Henri Poincar{\'{e}} D}, 10(4):635--685,
  {J}an 2023.

\bibitem[BPP20]{BanksPanzerPym}
Peter Banks, Erik Panzer, and Brent Pym.
\newblock Multiple zeta values in deformation quantization.
\newblock {\em Invent. Math.}, 222(1):79--159, 2020.

\bibitem[Bro10]{PeriodsFeynman}
Francis C.~S. Brown.
\newblock On the periods of some {F}eynman integrals.
\newblock {\em Preprint: arXiv:0910.0114}, 2010.

\bibitem[Bro12]{BrMTZ}
Francis Brown.
\newblock Mixed {T}ate motives over {$\mathbb{Z}$}.
\newblock {\em Ann. of Math. (2)}, 175(2):949--976, 2012.

\bibitem[Bro21]{BrSigma}
Francis Brown.
\newblock Invariant differential forms on complexes of graphs and {F}eynman
  integrals.
\newblock {\em SIGMA Symmetry Integrability Geom. Methods Appl.}, 17:Paper No.
  103, 54, 2021.

\bibitem[Bro23]{brown2023bordifications}
Francis Brown.
\newblock Bordifications of the moduli spaces of tropical curves and abelian
  varieties, and unstable cohomology of $\mathrm{GL}_g(\mathbb{Z})$ and
  $\mathrm{SL}_g(\mathbb{Z})$.
\newblock {\em Preprint: arXiv:2309.12753}, 2023.

\bibitem[BS12]{BrownSchnetz}
Francis Brown and Oliver Schnetz.
\newblock {A K3 in {$\phi^4$}}.
\newblock {\em Duke Math. J.}, 161(10):1817--1862, 2012.

\bibitem[BS22]{BorinskySchnetz}
M.~Borinsky and O.~Schnetz.
\newblock Graphical functions in even dimensions.
\newblock {\em Comm.\ in Number Theory and Physics}, 16(3):515--614, 2022.

\bibitem[CFP14]{ChurchFarbPutman}
Thomas Church, Benson Farb, and Andrew Putman.
\newblock A stability conjecture for the unstable cohomology of
  {$\mathrm{SL}_n(\mathbb{Z})$}, mapping class groups, and
  {$\mathrm{Aut}(F_n)$}.
\newblock In {\em Algebraic topology: applications and new directions}, volume
  620 of {\em Contemp. Math.}, pages 55--70. Amer. Math. Soc., Providence, RI,
  2014.

\bibitem[CGP22]{CGP}
M.~Chan, S.~Galatius, and S.~Payne.
\newblock Topology of moduli spaces of tropical curves with marked points.
\newblock In {\em Facets of algebraic geometry. {V}ol. {I}}, volume 472 of {\em
  London Math. Soc. Lecture Note Ser.}, pages 77--131. Cambridge Univ. Press,
  Cambridge, 2022.

\bibitem[CV03]{ConantVogtmann}
James Conant and Karen Vogtmann.
\newblock {On a theorem of Kontsevich}.
\newblock {\em {Algebraic \& Geometric Topology}}, 3(2):1167 -- 1224, 2003.

\bibitem[CV10]{CaporasoViviani}
Lucia Caporaso and Filippo Viviani.
\newblock Torelli theorem for graphs and tropical curves.
\newblock {\em Duke Math. J.}, 153(1):129--171, 2010.

\bibitem[Del89]{deligneP1}
P.~Deligne.
\newblock Le groupe fondamental de la droite projective moins trois points.
\newblock In {\em Galois groups over $\mathbf{Q}$ ({B}erkeley, {CA}, 1987)},
  volume~16 of {\em Math. Sci. Res. Inst. Publ.}, pages 79--297. Springer, New
  York, 1989.

\bibitem[Dri90]{Drinfeld}
V.~G. Drinfeld.
\newblock On quasitriangular quasi-{H}opf algebras and on a group that is
  closely connected with {$\mathrm{Gal}(\overline{\mathbb{Q}}/{\mathbb{ Q}})$}.
\newblock {\em Algebra i Analiz}, 2(4):149--181, 1990.

\bibitem[EVGS13]{ElbazVincentGanglSoule}
Philippe Elbaz-Vincent, Herbert Gangl, and Christophe Soul\'{e}.
\newblock Perfect forms, {K}-theory and the cohomology of modular groups.
\newblock {\em Adv. Math.}, 245:587--624, 2013.

\bibitem[GPS16]{GolzPanzerSchnetz}
Marcel Golz, Erik Panzer, and Oliver Schnetz.
\newblock Graphical functions in parametric space.
\newblock {\em Letters in Mathematical Physics}, 107(6):1177--1192, dec 2016.

\bibitem[Gro13]{Grobner}
Harald Grobner.
\newblock Residues of {E}isenstein series and the automorphic cohomology of
  reductive groups.
\newblock {\em Compos. Math.}, 149(7):1061--1090, 2013.

\bibitem[Hai02]{Hain3folds}
Richard Hain.
\newblock The rational cohomology ring of the moduli space of abelian 3-folds.
\newblock {\em Math. Res. Lett.}, 9(4):473--491, 2002.

\bibitem[Inc24]{Sloane}
OEIS~Foundation Inc.
\newblock Triangle read by rows, entry a036969 in the on-line encyclopedia of
  integer sequences, published electronically at
  \url{https://oeis.org/A036969}, 2024.

\bibitem[Kon93]{Kontsevich93}
Maxim Kontsevich.
\newblock Formal (non)commutative symplectic geometry.
\newblock In {\em The {G}elfand {M}athematical {S}eminars, 1990--1992}, pages
  173--187. Birkh\"{a}user Boston, Boston, MA, 1993.

\bibitem[KWv17]{SpectralSequenceGC2}
Anton Khoroshkin, Thomas Willwacher, and Marko \v{Z}ivkovi\'{c}.
\newblock Differentials on graph complexes.
\newblock {\em Adv. Math.}, 307:1184--1214, 2017.

\bibitem[Loo93]{Looijenga}
Eduard Looijenga.
\newblock Cohomology of {$\mathcal{M}_3$} and {$\mathcal{M}^1_3$}.
\newblock In {\em Mapping class groups and moduli spaces of {R}iemann surfaces
  ({G}\"{o}ttingen, 1991/{S}eattle, {WA}, 1991)}, volume 150 of {\em Contemp.
  Math.}, pages 205--228. Amer. Math. Soc., Providence, RI, 1993.

\bibitem[LS78]{LeeSzczarba}
Ronnie Lee and R.~H. Szczarba.
\newblock {On the torsion in $K_4(\mathbb{Z})$ and $K_5(\mathbb{Z})$}.
\newblock {\em Duke Mathematical Journal}, 45(1):101 -- 129, 1978.

\bibitem[MP14]{NAUTY}
Brendan~D. McKay and Adolfo Piperno.
\newblock Practical graph isomorphism, ii.
\newblock {\em Journal of Symbolic Computation}, 60:94--112, 2014.

\bibitem[MV12]{MeloViviani}
Margarida Melo and Filippo Viviani.
\newblock Comparing perfect and 2nd {V}oronoi decompositions: the matroidal
  locus.
\newblock {\em Math. Ann.}, 354(4):1521--1554, 2012.

\bibitem[MZ08]{MikhalinZharkov}
Grigory Mikhalkin and Ilia Zharkov.
\newblock Tropical curves, their {J}acobians and theta functions.
\newblock In {\em Curves and abelian varieties}, volume 465 of {\em Contemp.
  Math.}, pages 203--230. Amer. Math. Soc., Providence, RI, 2008.

\bibitem[Nag97]{Nagnibeda}
Tatiana Nagnibeda.
\newblock The {J}acobian of a finite graph.
\newblock In {\em Harmonic functions on trees and buildings ({N}ew {Y}ork,
  1995)}, volume 206 of {\em Contemp. Math.}, pages 149--151. Amer. Math. Soc.,
  Providence, RI, 1997.

\bibitem[OO21]{OdakaOshima}
Yuji Odaka and Yoshiki Oshima.
\newblock {\em Collapsing {K}3 surfaces, tropical geometry and moduli
  compactifications of {S}atake, {M}organ-{S}halen type}, volume~40 of {\em MSJ
  Memoirs}.
\newblock Mathematical Society of Japan, Tokyo, 2021.

\bibitem[Por23]{Portner}
Jean-Luc Portner.
\newblock {Computing the single-valued Knizhnik-Zamolodchikov associator}.
\newblock {\em Master's Thesis, available at
  \url{https://www.maths.ox.ac.uk/system/files/inline-files/Computing%20the%20single-valued%20Knizhnik-Zamolodchikov%20associator.pdf}},
  2023.

\bibitem[Pro15]{Procesi}
Claudio Procesi.
\newblock On the theorem of amitsur-levitzki.
\newblock {\em Israel Journal of Mathematics}, 207:151--154, 2015.

\bibitem[Ros76]{Rosset76}
Shmuel Rosset.
\newblock A new proof of the {A}mitsur-{L}evitski identity.
\newblock {\em Israel J. Math.}, 23(2):187--188, 1976.

\bibitem[RW14]{RossiWillwacher}
Carlo~Antonio Rossi and Thomas Willwacher.
\newblock {P. Etingof's conjecture about Drinfeld associators}.
\newblock {\em Preprint: arXiv:1404.2047}, 2014.

\bibitem[Sch14]{schnetz2014graphical}
Oliver Schnetz.
\newblock {Graphical functions and single-valued multiple polylogarithms}.
\newblock {\em Commun. Num. Theor. Phys.}, 08:589--675, 2014.

\bibitem[Sch18]{Schnetznumfunct}
Oliver Schnetz.
\newblock Numbers and functions in quantum field theory.
\newblock {\em Phys. Rev. D}, 97:085018, Apr 2018.

\bibitem[Sch21]{schnetz2021generalized}
Oliver Schnetz.
\newblock Generalized single-valued hyperlogarithms.
\newblock {\em Preprint: arXiv:2111.11246}, 2021.

\bibitem[Sch23]{7loops}
Oliver Schnetz.
\newblock ${\ensuremath{\phi}}^{4}$ theory at seven loops.
\newblock {\em Phys. Rev. D}, 107:036002, Feb 2023.

\bibitem[Sch24]{hlog}
Oliver Schnetz.
\newblock Hyperlog{P}rocedures,
  \url{https://www.math.fau.de/person/oliver-schnetz/}, 2024.

\bibitem[SEVKM19]{GL8}
Mathieu~Dutour Sikiri\'c, Philippe Elbaz-Vincent, Alexander Kupers, and Jacques
  Martinet.
\newblock {Voronoi complexes in higher dimensions, cohomology of
  $\mathrm{GL}_N(\mathbb{Z})$ for $N\geq 8$ and the triviality of
  $K_8(\mathbb{Z})$}.
\newblock {\em Preprint: arXiv:1910.11598}, 2019.

\bibitem[Sou78]{SouleSL3}
Christophe Soul\'e.
\newblock {The cohomology of $\mathrm{SL}_3(\mathbb{Z})$}.
\newblock {\em Topology}, 17(1):1--22, 1978.

\bibitem[SS19]{SchlottererSchnetz}
Oliver Schlotterer and Oliver Schnetz.
\newblock Closed strings as single-valued open strings: a genus-zero
  derivation.
\newblock {\em Journal of Physics A: Mathematical and Theoretical},
  52(4):045401, jan 2019.

\bibitem[VZ22]{VanhoveZerbini}
Pierre Vanhove and Federico Zerbini.
\newblock Single-valued hyperlogarithms, correlation functions and closed
  string amplitudes.
\newblock {\em Adv. Theor. Math. Phys.}, 26(2):455--530, 2022.

\bibitem[Wil15]{WillwacherGRT}
Thomas Willwacher.
\newblock M. {K}ontsevich's graph complex and the
  {G}rothendieck-{T}eichm\"{u}ller {L}ie algebra.
\newblock {\em Invent. Math.}, 200(3):671--760, 2015.

\end{thebibliography}

\end{document}